\documentclass[10pt,reqno]{amsart}
\usepackage[utf8]{inputenc}
\usepackage{amsmath}
\usepackage{amsfonts}
\usepackage{amssymb}
\usepackage{graphicx}
\usepackage[left=2.4cm,right=2.4cm,top=2.5cm,bottom=2.5cm]{geometry}
\numberwithin{equation}{section}
\graphicspath{ {images/} }
\usepackage{amsmath,amssymb, amsthm} 
\usepackage{lipsum}
\usepackage{stmaryrd}
\usepackage{dsfont}
\usepackage{comment}
\usepackage[all]{xy}

\makeatletter
\def\@settitle{\begin{center}%
  \baselineskip14\p@\relax
  \bfseries
  \uppercasenonmath\@title
  \@title
  \ifx\@subtitle\@empty\else
     \\[1ex]\uppercasenonmath\@subtitle
     \footnotesize\mdseries\@subtitle
  \fi
  \end{center}%
}
\def\subtitle#1{\gdef\@subtitle{#1}}
\def\@subtitle{}
\makeatother

\usepackage{tikz-cd}
\usepackage{hyperref}
\usepackage{bm}
\hypersetup{colorlinks, linkcolor=blue, citecolor=black, urlcolor=black}

\usepackage[english]{babel}
\usepackage[colorinlistoftodos]{todonotes}

\theoremstyle{plain}
\newtheorem{thm}{Theorem}[subsection] 

\usepackage{mathrsfs}

\theoremstyle{definition}
\newtheorem{defi}[thm]{Definition}

\newtheorem{rmk}[thm]{Remark}

\newtheorem{rem}[thm]{Remark}
\newtheorem*{rmk-intro}{Remark}
\theoremstyle{definition}

\theoremstyle{plain}
\newtheorem{prop}[thm]{Proposition}
\theoremstyle{plain}
\newtheorem{lemma}[thm]{Lemma}
\theoremstyle{plain}
\newtheorem{cor}[thm]{Corollary}
\theoremstyle{plain}
\newtheorem{conj}[thm]{Conjecture}
\theoremstyle{plain}
\newtheorem{thmintro}{Theorem}

\theoremstyle{plain}
\newtheorem{corintro}{Corollary}


\newcounter{parentnumber}

\DeclareMathOperator{\ord}{ord}

\newcommand{\Hom}{\operatorname{Hom}}

\newcommand{\ind}{\operatorname{ind}}
\newcommand{\Gal}{\operatorname{Gal}}

\newcommand{\Fil}{\mathrm{Fil}}

\newcommand{\al}{\alpha} 

\usepackage{bbm}

\makeatletter  
\newcommand{\colim@}[2]{%
  \vtop{\m@th\ialign{##\cr
    \hfil$#1\operator@font colim$\hfil\cr
    \noalign{\nointerlineskip\kern1.5\ex@}#2\cr
    \noalign{\nointerlineskip\kern-\ex@}\cr}}%
}
\newcommand{\colim}{%
  \mathop{\mathpalette\colim@{\rightarrowfill@\scriptscriptstyle}}\nmlimits@
}
\renewcommand{\varprojlim}{%
  \mathop{\mathpalette\varlim@{\leftarrowfill@\scriptscriptstyle}}\nmlimits@
}
\renewcommand{\varinjlim}{%
  \mathop{\mathpalette\varlim@{\rightarrowfill@\scriptscriptstyle}}\nmlimits@
}
\makeatother

\usepackage{multicol}

\newcommand{\Z}{\mathbb{Z}}
\newcommand{\Q}{\mathbb{Q}}

\newcommand{\Qp}{\mathbb{Q}_{p}}
\newcommand{\Zp}{\mathbb{Z}_{p}}

\newcommand{\R}{\mathbb{R}}

\newcommand{\F}{\mathcal{F}}

\newcommand{\p}{\mathfrak{P}}

\newcommand{\X}{\mathfrak{X}}

\newcommand{\sk}{\vspace{0.1in}}

\newcommand{\bkappa}{\boldsymbol{\kappa}}

\font\wncyr=wncyr9.8
\newcommand{\sha}{\text{\wncyr{W}}}
\newcommand{\CE}{\mathcal{E}}
\newcommand{\CM}{\mathcal{M}}

\DeclareMathOperator{\frob}{Frob}
\DeclareMathOperator{\gal}{Gal}

\DeclareMathOperator{\End}{End}

\DeclareMathOperator{\loc}{loc}

\DeclareMathOperator{\length}{length}

\newcommand{\Lcal}{\mathcal{L}}
\newcommand{\cF}{\mathcal{F}}
\newcommand{\CF}{\mathcal{F}}
\newcommand{\Fcal}{\mathcal{F}}

\newcommand{\cL}{\mathcal{L}}
\newcommand{\cN}{\mathcal{N}}
\newcommand{\kap}{\kappa}
\newcommand{\rank}{\mathrm{rank}}
\newcommand{\fm}{\mathfrak{m}}
\newcommand{\rH}{\mathrm{H}}

\begin{document}
\title{Non-vanishing of Kolyvagin systems and Iwasawa theory}

\author{Ashay Burungale}
\address[A.~Burungale]{University of Texas at Austin, USA}
\email{ashayk@utexas.edu }

\author{Francesc Castella}
\address[F.~Castella]{University of California Santa Barbara, South Hall, Santa Barbara, CA 93106, USA}
\email{castella@ucsb.edu}

\author{Giada Grossi}
\address[G.~Grossi]{CNRS, Institut Galilée, Université Sorbonne Paris Nord, 93430 Villetaneuse, FRANCE}
\email{grossi@math.univ-paris13.fr}

\author{Christopher Skinner} 
\address[C.~Skinner]{Princeton University, Fine Hall, Washington Road, Princeton, NJ 08544-1000, USA}
\email{cmcls@princeton.edu}

\begin{abstract}
Let $E/\Q$ be an elliptic curve and $p$ an odd prime. In 1991 Kolyvagin conjectured that the system of cohomology classes for torsion quotients of the $p$-adic Tate module of $E$
derived from Heegner points over ring class fields of 
a suitable imaginary quadratic field $K$ (i.e., the Heegner point Kolyvagin system of $E/K$) is non-trivial. 
In this paper we prove Kolyvagin's conjecture when $p$ is a prime of good ordinary reduction for $E$ that splits in $K$. 
In particular, 
our results cover many cases where $p$ is an Eisenstein prime for $E$, complementing Wei Zhang's earlier results on the conjecture by a  different approach.

Our methods also yield a proof of a refinement of Kolyvagin's conjecture expressing the divisibility index of the Heegner point Kolyvagin system in terms of the Tamagawa numbers of $E$, as conjectured by Wei Zhang in 2014, as well as proofs of analogous results for the Kolyvagin system obtained from Kato's Euler system. 
\end{abstract}

\date{\today}
\maketitle
\tableofcontents

\section*{Introduction}
\addtocontents{toc}{\protect\setcounter{tocdepth}{1}}

Let $E/\Q$ be an elliptic curve of conductor $N$, and let $p$ be an odd prime of good ordinary reduction for $E$. In this paper we prove Kolyvagin's conjecture on the non-vanishing of the $p$-adic Heegner point Kolyvagin system under mild conditions (see Theorem~\ref{thmintroKoly}). When $p$ is non-Eisenstein for $E$, the conjecture was first proved by Wei Zhang in many cases. The approach introduced in this paper covers the general non-Eisenstein case, as well as the first general cases where $p$ is Eisenstein for $E$. Moreover, in the former case we 
prove the \emph{refined} Kolyvagin conjecture made by Wei Zhang, expressing the divisibility index of the Heegner point Kolyvagin system in terms of the Tamagawa numbers of $E$ (see Theorem~\ref{thmintroKoly-div}). 
Following a similar strategy, we also prove analogous  results 
for the Kolyvagin system derived from Beilinson--Kato elements (see Theorem~\ref{thmintroKato}). 



\subsection{Main results}

\subsubsection*{Kolyvagin's conjecture} 

Let $K$ be an imaginary quadratic field of discriminant $-D_K<0$ such that
\begin{equation}\label{eq:intro-Heeg}
\textrm{every prime $\ell\vert N$ splits in $K$,}\tag{Heeg}
\end{equation} 
and fix an integral ideal $\mathfrak{N}\subset\mathcal{O}_K$ with $\mathcal{O}_K/\mathfrak{N}=\Z/N\Z$. Assume also that
\begin{equation}\label{eq:intro-disc}
\textrm{$D_K$ is odd and $D_K\neq -3$.}\tag{disc}
\end{equation}
For each positive integer $n$ coprime to $N$, let $\mathcal{O}_n=\Z+n\mathcal{O}_K$ be the order of $K$ of conductor $n$ and denote by  $K[n]$ the corresponding ring class field extension of $K$.

By the theory of complex multiplication, the cyclic $N$-isogeny between complex CM elliptic curves
\[
\mathbb{C}/\mathcal{O}_n\rightarrow\mathbb{C}/(\mathfrak{N}\cap\mathcal{O}_n)^{-1}
\]
defines a point $x_n\in X_0(N)(K[n])$. Fix a modular parameterisation 
\begin{equation}\label{par}
\pi:X_0(N)\rightarrow E.
\end{equation}
The associated \emph{Heegner point} on $E$ of conductor $n$ is defined by   
\begin{equation}\label{eq:heegnerpoint}
P[n]:=\pi(x_n)\in E(K[n]).
\end{equation}
We call $\ell$ a \emph{Kolyvagin prime} if $\ell$ is inert in $K$, coprime to $Np$, and 
$$
M(\ell):=\min \{\ord_p(\ell+1), \ord_p(a_\ell)\}>0,
$$
 where $a_\ell:=\ell+1-|\tilde{E}(\mathbb{F}_\ell)|$ and $\ord_p(x)$ denotes the $p$-adic valuation of an integer $x$. Let $\mathcal{N}_{\rm Heeg}$ be the set of squarefree products of Kolyvagin primes, and for $n\in \mathcal{N}_{\rm Heeg}$ put 
$$
 M(n):=\min \{M(\ell)\;\colon\;\ell\mid n\}
$$
if $n>1$ and $M(1):=\infty$. 

Let $T=T_pE$ be the $p$-adic Tate module of $E$, and suppose that
\begin{equation}\label{eq:intro-tor}
E(K)[p]=0.\tag{tor}
\end{equation}
From the Kummer images of the Heegner points $P[n]$, Kolyvagin constructed a system of classes 
\[
\bigl\{\kappa_{n}^{\rm Heeg}\in \rH^1(K,T/I_nT)\;\colon\;n\in \mathcal{N}_{\rm Heeg}\bigr\},\quad\textrm{where $I_n=p^{M(n)}\Z_p$.}
\] 
In particular, $\kappa_{1}^{\rm Heeg}$ is the image of the Heegner point $P_K:=\operatorname{Tr}_{K[1]/K}(P[1])\in E(K)$ under the Kummer map
\[
E(K)\otimes\Z_p\rightarrow\rH^1(K,T), 
\] 
and so by the Gross--Zagier formula \cite{grosszagier}, $\kappa_1^{\rm Heeg}\neq 0$ if and only if $L'(E/K,1)\neq 0$. 

In \cite{kolystructure} Kolyvagin  conjectured
that even when the analytic rank of $E/K$ is greater than one, 
the \emph{system} $\{\kappa_{n}^{\rm Heeg}\}_{n}$ is non-trivial,  i.e. there exists $n\in \mathcal{N}_{\rm Heeg}$ such that $$\kappa_{n}^{\rm Heeg}\neq 0.$$ 
The first major progress towards this conjecture was due to W.\,Zhang \cite{zhang_ind}: about a decade ago, he proved Kolyvagin's conjecture 
when $p\nmid 6N$ is a prime of good ordinary reduction for $E$ 
under the assumption that
\begin{equation}\label{eq:irred}
\tag{sur}
\textrm{$\bar{\rho}_E: G_{\Q}=\operatorname{Gal}(\bar{\Q}/\Q)\to{\rm Aut}_{\mathbb{F}_p}(E[p])$ is surjective}
\end{equation}
and $\bar{\rho}_E$ satisfies certain ramification hypotheses. More recently, 
some of the hypotheses have been 
relaxed by N.\,Sweeting  \cite{sweeting} using an ultrapatching method for bipartite Euler systems. 

In the first part of this paper we prove Kolyvagin's conjecture in cases where \eqref{eq:irred} is not necessarily satisfied, and we also relax the ramification hypotheses. The paper considers Eisenstein primes $p$ as well as cases where
\begin{equation}\label{eq:irr}
\tag{irr}
\text{$E[p]$ is an irreducible $G_\Q$-module.}
\end{equation}
Our method relies on anticyclotomic Iwasawa theory, recent developments towards anticyclotomic Main Conjectures, and ideas 
 introduced in \cite{eisenstein_cyc}. 
A key ingredient is 
the Kolyvagin system bound with ``error terms'' obtained in \emph{op.\,cit.}, which allows us to control the size of the Tate--Shafarevich group of certain anticyclotomic twists of $T$ 
by characters $\alpha$ with $\alpha\equiv 1\pmod{p^m}$ for $m\gg 0$. 
When \eqref{eq:irred} holds, the error terms are zero,
and when a certain Iwasawa Main Conjecture is known our methods also yield a proof of the \emph{refined Kolvyagin conjecture} made by W.\,Zhang \cite[Conj.~4.5]{zhang-CDM}.
%
\sk

We now describe the main result
 precisely. 
Let 
$G_{\Q}$ be the absolute Galois group of $\Q$ and 
$G_p\subset G_\Q$ a decomposition group at $p$. 
Denote by $\omega:G_\Q\rightarrow\mathbb{F}_p^\times$ the Teichm\"{u}ller character. 
%

\begin{thmintro}[Kolyvagin's conjecture]\label{thmintroKoly}
Let $E/\Q$ be an elliptic curve, and let $p$ be an odd prime of good ordinary reduction for $E$. Let $K$ be a quadratic imaginary field satisfying \eqref{eq:intro-Heeg}, \eqref{eq:intro-disc}, \eqref{eq:intro-tor}, and such that $p$ splits in $K$. Assume that the rational anticyclotomic Main Conjecture \ref{conj:anticyc} holds.
 Then 
\begin{center}
there exists $n\in\mathcal{N}_{\rm Heeg}$ such that $\kappa_{n}^{\rm Heeg}\neq 0$. 
\end{center}
In particular, 
$\{\kappa_{n}^{\rm Heeg}\}\neq 0$ in both of the following cases:
\begin{itemize}
\item[$\circ$] $E$ admits a rational $p$-isogeny with kernel $\mathbb{F}_p(\phi)\subset E[p]$, where $\phi:G_\Q \to \mathbb{F}_p^{\times}$ is a character such that $\phi|_{G_p}\neq\mathds{1},\omega$. 
\item[$\circ$] $p>3$ satisfies {\rm (\ref{eq:irr})}.
\end{itemize}
\end{thmintro}

The `In particular' part 
of Theorem~\ref{thmintroKoly}
relies on the results of \cite{eisenstein_cyc,bcs}, respectively, on 
the anticyclotomic Main Conjecture.

The non-vanishing of the Kolyvagin system in 
 combination with \cite[Theorem 4]{kolystructure} 
 leads to a link between the \emph{order of vanishing} of the Kolyvagin system and the rank of the $p^{\infty}$-Selmer group of $E/K$. More precisely, for any $n$, let $\nu(n)$ denote the number of prime factors of $n$, and 
for $n\in\mathcal{N}_{\rm Heeg}$ 
let
\[
\ord(\kappa^{\rm Heeg}):= \min \{r: \text{ there exists }n\in \mathcal{N}_{\rm Heeg} \text{ with } \nu(n)=r \text{ such that }\kappa_n^{\rm Heeg}\neq 0\}.
\]
 The $p^\infty$-Selmer group ${\rm Sel}_{p^{\infty}}(E/K)$ 
  has an action of complex conjugation. Denote by ${\rm Sel}_{p^{\infty}}(E/K)^\pm$ the $\pm$-eigenspaces with respect to this action and let 
$$r(E/K)^{\pm}= {\rm corank}_{\Z_p}{\rm Sel}_{p^{\infty}}(E/K)^\pm.$$
\begin{corintro}\label{corA}
For $E$, $p$ and $K$ as in Theorem \ref{thmintroKoly}, we have
\[
\ord(\kappa^{\rm Heeg})= \max\{r(E/K)^+,r(E/K)^-\} -1.
\]
\end{corintro}

\begin{rmk-intro} 
The corollary applies for arbitrary values of $r(E/K)^\pm$. In the rank one case 
it yields a $p$-converse to the Gross--Zagier and Kolyvagin theorem:
$$
{\rm corank}_{\Z_{p}}{\rm Sel}_{p^{\infty}}(E/K)=1 \implies \ord_{s=1}L(E/K,s)=1,
$$
as the Birch and Swinnerton-Dyer conjecture predicts (cf.~\cite[Theorem 1.3]{zhang_ind}).  One can also deduce the $p$-parity conjecture for $E/\Q$ (cf.~\cite[Theorem 1.2, Remark 1]{zhang_ind}).
\end{rmk-intro}

\subsubsection*{Refined Kolyvagin's conjecture} 

As formulated by W.\,Zhang \cite{zhang_ind}, 
a refinement of Kolyvagin's conjecture predicts a formula for the divisibility index of the Heegner point Kolyvagin system $\{\kappa_n^{\rm Heeg}\}$ in terms of the Tamagawa numbers $$c_\ell=[E(\Q_\ell):E^{0}(\Q_\ell)]$$ at the primes $\ell\vert N$. 

For each $n\in\mathcal{N}_{\rm Heeg}$, define $\mathscr{M}(n)\in\Z_{\geq 0}\cup\{\infty\}$ by $\mathscr{M}(n)=\infty$ if $\kappa_n^{\rm Heeg}=0$, and by
\[
\mathscr{M}(n)=\max\{\mathscr{M}\colon\kappa_n^{\rm Heeg}\in p^{\mathscr{M}}\rH^1(K,T/I_nT)\}
\]
otherwise. Put $\mathscr{M}_{r}=\min\{\mathscr{M}(n)\colon\nu(n)=r\}$. 
We have $\mathscr{M}_{r}\geq\mathscr{M}_{r+1}\geq 0$ for all $r\geq 0$ (cf.~\cite{kolystructure}).
 Put 
\[
\mathscr{M}_{\infty}=\lim_{r\to\infty}\mathscr{M}_{r}.
\]
Note that Kolyvagin's conjecture is equivalent to the finiteness of $\mathscr{M}_{\infty}$. 
W. Zhang's refinement \cite[Conj.~4.5]{zhang_ind} provides a conjectural formula for $\mathscr{M}_{\infty}$, 
which we establish:
\begin{thmintro}[Refined Kolyvagin's conjecture]\label{thmintroKoly-div}
Let $(E,p,K)$ be as in Theorem~\ref{thmintroKoly}, and assume that $p>3$ and {\rm (\ref{eq:irred})} holds. Assume that the integral anticyclotomic Main Conjecture \ref{conj:anticyc} holds.
Assume also that the fixed modular parametrisation $\pi: X_{0}(N) \to E$ as in 
\eqref{par} is $p$-optimal\footnote{That is 
$\ord_p(\deg(\pi))$ is minimal among all modular parametrisations of all curves in the $\Q$-isogeny class of $E$. }. 
Then we have
\[
\mathscr{M}_{\infty}=\sum_{\ell\mid N}{\rm ord}_p(c_\ell).
\]
\end{thmintro}
\begin{rmk-intro}
In the rank one case Theorem~\ref{thmintroKoly-div} implies the $p$-part of the conjectural Birch and Swinnerton-Dyer formula for $E/K$ (cf.~\cite[Theorem~1.6]{zhang_ind}).
\end{rmk-intro}

Our proof  of Theorem \ref{thmintroKoly-div} 
relies on the cases of 
the anticyclotomic Main 
Conjecture \ref{conj:anticyc} established in \cite{bcs}.

In \cite{jetchev} Jetchev 
showed that $\mathscr{M}_{\infty}\geq {\rm ord}_p(c_\ell)$ for any prime $\ell\vert N$. The first cases of the refined Kolyvagin's conjecture were proved by W.\,Zhang: the ramification hypotheses in \cite{zhang_ind} imply that $p\nmid c_\ell$ for all primes $\ell\vert N$, and the main result of \emph{op.\,cit.} shows that $\mathscr{M}_{\infty}=0$.

\subsubsection*{Non-vanishing of Kato's Kolyvagin system}

The strategy 
used
in this paper 
to prove Theorems~\ref{thmintroKoly} and \ref{thmintroKoly-div} can be adapted to establish analogous results for the Kolyvagin system derived from Beilinson--Kato elements. The second part of the paper is devoted to the proof of such results.

In \cite{kato-euler-systems} Kato constructed an Euler system 
for $T$ (viewed as a $G_\Q$-representation) building on Siegel units.
Let $\mathcal{L}_{\rm Kato}$ be the set of primes $\ell\nmid Np$ such that
\[
I_\ell:=(\ell-1,a_\ell-\ell-1)\subset p\Z_p,
\]
and let
$\mathcal{N}_{\rm Kato}$ be the set of squarefree products of primes in $\mathcal{L}_{\rm Kato}$. For 
$n\in \mathcal{N}_{\rm Kato}$, let
\[
I_n:=\sum_{\ell\mid n}I_\ell 
\]
for $n>1$, and put $I_1:=\{0\}$.
By the process of Kolyvagin derivatives, from Kato's Euler system 
one obtains a system of cohomology classes 
\[
\bigl\{\kappa^{\rm Kato}_{n}\in \rH^1(\Q, T/I_nT)\;\colon\;n\in \mathcal{N}_{\rm Kato}\bigr\}
\]
forming a Kolyvagin system for $T$ in the sense of Mazur--Rubin (see \cite[Thm.~3.2.4]{mazrub}). 

In this setting our main result is the following. 


\begin{thmintro}[Non-vanishing of Kato's Kolyvagin system]\label{thmintroKato}
Let $E/\Q$ be an elliptic curve without CM, and let $p$ be an odd prime of good ordinary reduction for $E$ such that $E(\Q_p)[p]=0$. Assume that the rational cyclotomic Main Conjecture \ref{conj:cyc} holds.  
Then 
\begin{center}
there exists $n\in \mathcal{N}_{\rm Kato}$ such that $\kappa_{n}^{\rm Kato}\neq 0$. 
\end{center}
In particular, $\{\kappa_{n}^{\rm Kato}\}\neq 0$ in both of the following cases:
\begin{itemize}
\item[$\circ$] $E$ admits a rational $p$-isogeny with kernel $\mathbb{F}_p(\phi)\subset E[p]$, where $\phi:G_\Q \to \mathbb{F}_p^{\times}$ is a character such that $\phi|_{G_p}\neq\mathds{1},\omega$. 
\item[$\circ$] $p>3$ satisfies \eqref{eq:irr}.
\end{itemize}
\end{thmintro}
The `In particular' part relies on the cases of Conjecture \ref{conj:cyc} established in \cite{kato-euler-systems,eisenstein_cyc} for Eisenstein $p$ and in \cite{kato-euler-systems,skinner-urban,wan-hilbert} for non-Eisenstein $p$.

We now describe an application of Theorem~\ref{thmintroKato}. 
Similarly as in the Heegner point case, 
for any $n\in\mathcal{N}_{\rm Kato}$, 
define the order of vanishing of the Kolyvagin system $\kappa^{\rm Kato}$ by
\[
\ord(\kappa^{\rm Kato}):= \min \{r: \text{ there exists }n\in \mathcal{N}_{\rm Kato} \text{ with } \nu(n)=r \text{ such that }\kappa_n^{\rm Kato}\neq 0\}.
\]
Let $\rH^1_{\rm str}(\Q,E[p^\infty])$ denote the \emph{fine Selmer group} for $E[p^{\infty}]$, obtained by imposing the strict local condition at $p$ (see $\S$3.1). Let 
$$r_{\rm str}(E/\Q)= {\rm corank}_{\Z_p}\left(\rH^1_{\rm str}(\Q,E[p^\infty])\right).$$
Combining the theorem above with \cite[Theorem 5.2.12(v)]{mazrub}, we obtain an analogue of Corollary \ref{corA} for Kato's Kolyvagin system.
\setcounter{corintro}{2}
\begin{corintro}
For $E$ and $p$ as in Theorem \ref{thmintroKato}, we have 
\[
\ord(\kappa^{\rm Kato})= r_{\rm str}(E/\Q).
\]
\end{corintro}

The divisibility index $\mathscr{M}_\infty^{\rm Kato}$ of  $\{\kappa_n^{\rm Kato}\}$ was studied by Mazur and Rubin  
\cite[\S{6.2}]{mazrub}, who showed that if ${\rm ord}_p(c_\ell)>0$ for some prime $\ell\vert N$, then $\mathscr{M}_\infty^{\rm Kato}>0$. This was  refined by Büy\"{u}kboduk \cite{kazim}, who showed that $\mathscr{M}_\infty^{\rm Kato}\geq{\rm ord}_p(c_\ell)$ for any prime $\ell\vert N$. 
Our approach to Theorem~\ref{thmintroKato} also leads to a proof of the equality
\[
\mathscr{M}_\infty^{\rm Kato}=\sum_{\ell\mid N}{\rm ord}_p(c_\ell)
\] 
under some hypotheses (see Remark~\ref{rem:refined-kato}; these also include the hypothesis that $\pi$ is $p$-optimal), yielding a proof \cite[Conj.~1.9]{kim} on the divisibility index of the analytic quantities $\delta_n\in\Z_p/I_n$ introduced by Kurihara \cite{kur-Iw2012} in terms of modular symbols (cf. \cite{kurihara-sakamoto}).

\subsection{Strategy} 

Let $\Gamma={\rm Gal}(K_\infty/K)$ be the Galois group of the anticyclotomic $\Z_p$-extension of $K$, and denote by $\Lambda=\Z_p[[\Gamma]]$ the anticyclotomic Iwasawa algebra. 

The key ingredients in the proof of Theorem~\ref{thmintroKoly} are:
\begin{enumerate}
\item[(A)] The non-vanishing of the base class $\bkappa_1^{\rm Heeg}$ of the \emph{$\Lambda$-adic} Heegner point Kolyvagin system (i.e., the proof by Cornut and Vastal of Mazur's conjecture). 
\item[(B)] Letting $\bkappa_1^{\rm Heeg}(\alpha)$ be the specialisation of $\bkappa_1^{\rm Heeg}$ at a character $\alpha$ of $\Gamma$ with $\alpha\equiv 1\;({\rm mod}\,p^m)$ for suitable $m\gg 0$ such that 
\begin{equation}\label{nv}
\bkappa_1^{\rm Heeg}(\alpha)\neq 0,
\end{equation} 
a \emph{sharp estimate} on the divisibility index of $\bkappa_1^{\rm Heeg}(\alpha)$ in terms of Tamagawa numbers and the Shafarevich--Tate group 
of the twist of $E/K$ by $\alpha$.
\item[(C)] The Kolyvagin system bound with controlled error terms obtained in \cite{eisenstein_cyc}.
\end{enumerate}

Most of the work in the first part of the paper goes into the proof of (B), which we deduce from the explicit reciprocity law\footnote{This is a $\Lambda$-adic avatar of the Bertolini--Darmon--Prasanna formula \cite{BDP}.}
for $\bkappa_1^{\rm Heeg}$ in \cite{cas-hsieh1} and an extension of the anticyclotomic control theorem in \cite{jsw} allowing character twists. Together with a congruence relation between $\bkappa_{n}^{\rm Heeg}(\alpha)$ and $\kappa_n^{\rm Heeg}$, 
the results (A)--(C) lead to a proof of Theorem~\ref{thmintroKoly}. 

To approach the refined form of Kolyvagin's conjecture as in Theorem~\ref{thmintroKoly-div}, we replace the Kolyvagin system bound in (C) with an \emph{exact formula} (in particular without error terms; \eqref{eq:irred} enters here) for the size of the above Tate--Shafarevich group in terms of the divisibility index of $\bkappa_1^{\rm Heeg}(\alpha)$ and the divisibility of the system $\{\bkappa_{n}^{\rm Heeg}(\alpha)\}_n$, generalising a consequence of
%
the structure theorem \cite[Thm.~1]{kolystructure} in the case $\alpha=1$. For potential future application, we emphasize that our result applies to the twists of $E$ by any anticyclotomic character $\alpha$ not necessarily $p$-adically close to the trivial one.
Relating the index of divisibility of $\{\bkappa_{n}^{\rm Heeg}(\alpha)\}_n$ to that of $\{\kappa_n^{\rm Heeg}\}_n$ gives Theorem~\ref{thmintroKoly-div}.

The proof of Theorem~\ref{thmintroKato} proceeds along similar lines. The analogue of (A) follows from results of Rohrlich and Kato, and we deduce the analogue of (B) from Kato's explicit reciprocity law and a twisted variant of the Euler characteristic computation originally due to Schneider and Perrin-Riou.
The analogue of (C) in the setting of Kato's Euler system requires more work; we deduce it from the results developed in the last section of the paper that extend the Kolyvagin system bound of Mazur and Rubin \cite{mazrub} to a more general setting, allowing in particular for Eisenstein primes, a result that may be of independent interest. 


\subsection{Relation to previous works}

\subsubsection*{Comparison with \cite{zhang_ind,sweeting}.} The approach used by Wei Zhang in his breakthrough work on Kolyvagin's conjecture is based on the principle of level-raising and rank-lowering, where rank refers to that of an associated mod $p$ Selmer group; this was extended by N.\,Sweeting to mod $p^m$ Selmer groups for large $m$. This approach proceeds by induction on the rank, the base case being: the triviality of the mod $p$ Selmer group associated to a weight two elliptic newform implies the $p$-indivisibility of the algebraic part of the associated central $L$-value. This rank zero 
implication 
is a consequence of the (integral) cyclotomic Main Conjecture for the newform, as established by Skinner and Urban \cite{skinner-urban} under the hypotheses of \cite{zhang_ind} (by working mod $p^m$, some of the ramification hypotheses 
were removed in \cite{sweeting}). It is precisely at this stage that Iwasawa theory enters into this approach, 
albeit only implicitly. In contrast, 
the strategy in this paper is inherently Iwasawa-theoretic, based on the variation of Heegner points along the anticyclotomic $\Z_p$-extension. 
Moreover, a {\em rational} Main Conjecture suffices for its implementation. 
The approach of \cite{zhang_ind,sweeting} essentially excludes the Eisenstein case. 
However, while the results of this paper on Kolyvagin's conjecture are for $p$ split in $K$, the results of \cite{zhang_ind,sweeting} apply also when $p$ is inert in $K$.

\subsubsection*{Comparison with \cite{kim2,kim}} 
In the cases where (\ref{eq:irred}) holds, results similar to Theorems~\ref{thmintroKoly} and~\ref{thmintroKato} are obtained by C.-H.~Kim \cite{kim2,kim}: he proves that if the anticyclotomic (resp. cyclotomic) Main Conjecture holds at the trivial character, then the system $\{\kappa_n^{\rm Heeg}\}$ (resp. $\{\kappa_n^{\rm Kato}\}$) is non-trivial. To draw a parallel with our approach, while the ingredient (A) is common, 
his proof interestingly replaces our ingredients (B) and (C) with a structure theorem for Selmer groups in terms of divisibility indices of certain specialisations of $\Lambda$-adic classes.   
In contrast, under (\ref{eq:irred}) our approach
also 
leads to a proof of the refined Kolyvagin conjecture  
and its cyclotomic analogue. 
%
%

\subsection*{Acknowledgements}
We thank K\^{a}zim Büy\"{u}kboduk, Masataka Chida, Christophe Cornut, Shinichi Kobayashi, David Loeffler and Wei Zhang for helpful comments and discussions.  
We also thank Chan-Ho Kim and Masato Kurihara  
for their interest in our work, and comments on an earlier draft. We are also grateful to the referee for their valuable suggestions. 

During the preparation of this paper, 
A.B. was partially supported by the NSF grants DMS-2303864 and DMS-2302064;
F.C. was partially supported by the NSF grant 
DMS-2101458 and DMS-2401321; C.S. was partially supported by the Simons Investigator Grant \#376203 from the Simons Foundation and by the NSF grant DMS-1901985. This work was partially supported by the National Science Foundation under Grant No. DMS-1928930 while the authors were in residence at the Simons Laufer Mathematical Sciences 
Institute in Berkeley, California during the Spring 2023 semester.

\addtocontents{toc}{\protect\setcounter{tocdepth}{2}}

\section{Heegner points and anticyclotomic Iwasawa theory}\label{s:Heegner}
\subsection{The Kolyvagin system of Heegner points}
\label{subsec:KS-Heeg}

In this section we recall the construction of the classes $\kappa_{n}=\kappa_{n}^{\rm Heeg}\in \rH^1(K,T/I_nT)$ and of their Iwasawa-theoretic analogues. The reader may refer to \cite[$\S$1.7,~$\S$2.3]{howard} for more details. 

Throughout this section, let $E/\Q$ be an elliptic curve of conductor $N$, let $p\nmid 2N$ be a prime of good ordinary reduction for $E$, and let $K$ be an imaginary quadratic field of discriminant $D_K$  prime to $Np$. We assume that the triple $(E,p,K)$ satisfies hypotheses (\ref{eq:intro-Heeg}), (\ref{eq:intro-disc}), and (\ref{eq:intro-tor}) from the Introduction. 
\subsubsection{Selmer structures}\label{secsel}

Let $\rho_E:G_{\Q}\rightarrow{\rm Aut}_{\Z_p}(T_pE)$ denote the Galois representation afforded by the $p$-adic Tate module of $E$. 
Let $\Gamma={\rm Gal}(K_\infty/K)$ be the Galois group of the anticyclotomic $\Z_p$-extension of $K$ and $\alpha:\Gamma\rightarrow\Z_p^\times$ a character. 
Consider the $G_K$-modules
\begin{equation}
T_\alpha:=T_pE\otimes_{}\Z_p(\alpha),\quad
V_\alpha:=T_\alpha\otimes_{}\Q_p,\quad W_\alpha:=T_\alpha\otimes_{}\Q_p/\Z_p\simeq V_\alpha/T_\alpha,\nonumber
\end{equation}
where $\Z_p(\alpha)$ is the free $\Z_p$-module of rank one on which $G_K$ acts via the projection $G_K\twoheadrightarrow\Gamma$ composed with $\alpha$, and the $G_K$-action on $T_\alpha$ is via $\rho_\alpha = \rho_E\otimes\alpha$. 
For $k>0$ let $T_\alpha^{(k)} = T_\alpha/p^kT_\alpha$ and note that we have natural identifications
\[
T_\alpha^{(k)}=T_pE\otimes(\Z/p^k\Z)(\alpha)\xrightarrow{\simeq}W_\alpha[p^k]\xrightarrow{\simeq}E[p^k]\otimes\Z_p(\alpha)
\]
given by $(a_n)\otimes 1\mapsto(a_n)\otimes\frac{1}{p^k}\mapsto a_k$. 


A \emph{Selmer structure} $\mathcal{F}$ on any of the $G_K$-modules $M=T_{\alpha}, V_\alpha, W_\alpha, T_\alpha^{(k)}\cong W_\alpha[p^k]$ is a finite set $\Sigma=\Sigma(\mathcal{F})$ of places of $K$ containing $\infty$, the primes above $p$, and the primes where $M$ is ramified, together with a choice of $\Zp$-submodules  $\rH^1_{\Fcal}(K_w,M)\subset\rH^1(K_w,M)$ for every $w\in\Sigma$ (``local conditions''). The associated \emph{Selmer group} is then defined by
\[
\rH^1_{\Fcal}(K,M):={\rm ker}\biggl\{\rH^1(K^\Sigma/K,M)\rightarrow\prod_{w\in\Sigma}\frac{\rH^1(K_w,M)}{\rH^1_{\Fcal}(K_w,M)}\biggr\},
\]
where $K^\Sigma$ is the maximal extension of $K$ unramified outside $\Sigma$. 

We recall the definition of the local conditions of interest in this section:
\begin{itemize}
\item For a finite prime $w\nmid p$, the \emph{finite} (or \emph{unramified}) local condition  for $V_\alpha$ is
\[
\rH^1_f(K_w,V_{\alpha}):={\rm ker}\bigl\{\rH^1(K_w,V_{\alpha})\rightarrow\rH^1(K_w^{\rm ur},V_{\alpha})\bigr\}; 
\]
the corresponding local conditions $\rH^1_f(K_w,T_\alpha)$ and $\rH^1_f(K_w,W_\alpha)$ are defined to be the inverse image and the image, respectively, of $\rH^1_f(K_w,V_\alpha)$ under the natural maps  
\begin{equation}\label{eq:nat}
\rH^1(K_w,T_\alpha)\rightarrow\rH^1(K_w,V_\alpha)\rightarrow\rH^1(K_w,W_\alpha).
\end{equation}
Similarly, the local condition $\rH^1_f(K_w,T_\alpha^{(k)})$ is defined as the preimage of $\rH^1_f(K_w,W_\alpha)[p^k]$ under the natural composite map
$\rH^1(K_w,T_\alpha^{(k)})\cong\rH^1(K_w,W_\alpha[p^k])\twoheadrightarrow \rH^1(K_w,W_\alpha)[p^k]$ (equivalently, $\rH^1_f(K_w,T_\alpha^{(k)})$ is defiend as the image of $\rH^1_f(K_w,T_\alpha)$ under the natural map $\rH^1(K_w,T_\alpha)\rightarrow\rH^1(K_w,T_\alpha^{(k)})$). If $T_\alpha$ is unramified at $w$, then all these 
are just the submodules of unramified classes (see also Remark \ref{rem:tam}).
\item Denote by $\cL_0$ the set of rational primes $\ell\neq p$ such that:
\begin{itemize}
\item $\ell$ is inert in $K$,
\item $T_pE$ is unramified at $\ell$.
\end{itemize}
For any $\ell\in\cL_0$ with $\ell+1\equiv 0\pmod{p^k}$, letting $K[\ell]$ be the ring class field of $K$ of conductor $\ell$, define the \emph{transverse} local condition for $T_\alpha^{(k)}$ at the prime $\lambda$ above $\ell$ by
\[
\rH^1_{\rm tr}(K_\lambda,T_{\alpha}^{(k)}):={\rm ker}\bigl\{\rH^1(K_\lambda,T_{\alpha}^{(k)})\rightarrow\rH^1(K[\ell]_{\lambda'},T_{\alpha}^{(k)})\bigr\},
\]
where $K[\ell]_{\lambda'}$ is the completion of $K[\ell]$ at any prime $\lambda'$ above $\lambda$.
\item For $w$ a prime of $K$ above $p$, set 
\[
{\rm Fil}_w^+(T_pE):={\rm ker}\bigl\{T_pE\rightarrow T_p\tilde{E}\bigr\},
\]
where $\tilde{E}$ is the reduction of $E$ at $w$, and put
\[
{\rm Fil}_w^+(T_\alpha):={\rm Fil}_w^+(T_pE)\otimes\Z_p(\alpha),\quad{\rm Fil}_w^+(V_\alpha):={\rm Fil}_w^+(T_\alpha)\otimes\Q_p.
\]
The \emph{ordinary} local condition for $V_\alpha$ is defined by 
\[
\rH^1_{\rm ord}(K_w,V_\alpha):={\rm im}\bigl\{\rH^1(K_w,{\rm Fil}_w^+(V_\alpha))\rightarrow\rH^1(K_w,V_\alpha)\bigr\}.
\]
Similarly as before, the corresponding local conditions $\rH^1_{\rm ord}(K_w,T_\alpha)$ and $\rH^1_{\rm ord}(K_w,W_\alpha)$ are defined to be the inverse image and the image, respectively, of $\rH^1_{\rm ord}(K_w,V_\alpha)$ under the natural maps (\ref{eq:nat}), and $\rH^1_{\rm ord}(K_w,T_\alpha^{(k)})$ is the preimage
of $\rH^1_{\rm ord}(K_w,W_\alpha)[p^k]$.
\end{itemize}

\begin{rmk}\label{rem:tam}
For $w\nmid p$, it follows immediately from the definitions that $\rH^1_f(K_w,W_\alpha)$ is contained in 
\[
\rH^1_{\rm ur}(K_w,W_{\alpha}):={\rm ker}\bigl\{\rH^1(K_w,W_{\alpha})\rightarrow\rH^1(K_w^{\rm ur},W_{\alpha})\bigr\}.
\]
In fact, for $V_\alpha$ one even has $\rH^1_f(K_w,V_\alpha)=0$, and therefore $\rH^1_f(K_w,W_\alpha)$ is also trivial. Thus the $p$-part of the Tamagawa number of $W_\alpha$ is given by
\[
c_w^{(p)}(W_\alpha):=[\rH^1_{\rm ur}(K_w,W_\alpha)\colon\rH^1_f(K_w,W_\alpha)]=\#\rH^1_{\rm ur}(K_w,W_\alpha).
\]
\end{rmk}


\begin{defi}
The \emph{ordinary} Selmer structure $\Fcal_{\rm ord}$ on $V_\alpha$ is defined by taking $\Sigma(\Fcal_{\rm \ord}) = \{w\mid pN\}$ and  
\[
\rH^1_{\Fcal_{\rm ord}}(K_w,V_\alpha):=
\begin{cases}
\rH^1_{\rm ord}(K_w,V_\alpha) & \textrm{if $w\mid p$,}\\[0.2em]
\rH^1_f(K_w,V_\alpha)& \textrm{else.}
\end{cases}
\]
Let $\Fcal_{\rm ord}$ also denote the resulting Selmer structure on $T_\alpha$, $W_\alpha$, and $T_\alpha^{(k)}$.
\end{defi}

\begin{rmk}\label{rmk:usualsel}
For $\alpha=\mathds{1}$, 
note that $\rH^1_{\Fcal_{\rm ord}}(K,W_\alpha)={\rm Sel}_{p^\infty}(E/K)$ (see e.g. \cite[Prop.\,1.6.8]{rubin-ES}).
\end{rmk}

Given a subset $\Lcal\subset\Lcal_0$, we let $\mathcal{N}(\Lcal)$ be the set of squarefree products of primes $\ell\in\Lcal$; when the choice of $\Lcal$ is irrelevant or clear from the context, we shall simply denote this by $\mathcal{N}$.  
Given a Selmer structure $\Fcal$ on $T_\alpha^{(k)}$, and $n\in\mathcal{N}$, 
we define the modified Selmer structure $\Fcal(n)$ by
\[
\rH^1_{\Fcal(n)}(K_\lambda,T_\alpha^{(k)})
=\begin{cases}
\rH^1_{\rm tr}(K_\lambda,T_\alpha^{(k)}) &\text{if }\lambda\mid n, \\[0.2em]
\rH^1_{\Fcal}(K_\lambda,T_\alpha^{(k)}) & \text{if }\lambda\nmid n.
\end{cases}
\]

\subsubsection{The Kolyvagin system of Heegner points}\label{subsubsec:HPKS}
Fix a modular parameterisation 
$
\pi:X_0(N)\rightarrow E.
$
It gives rise to the Heegner point  
$
P[n]:=\pi(x_n)\in E(K[n]) 
$ of conductor $n$ as in \eqref{eq:heegnerpoint} for any positive integer $n$ coprime to $N$. Recall that a prime $\ell\in\mathcal{L}_0$ is a \emph{Kolyvagin prime} it satisfies
$
M(\ell):=\min \{\ord_p(\ell+1), \ord_p(a_\ell)\}>0,
$
where $a_\ell:=\ell+1-|\tilde{E}(\mathbb{F}_\ell)|$.
Let
\[
\mathcal{L}_{\rm Heeg}:=\{\ell\in\mathcal{L}_0\,\colon\, a_\ell\equiv\ell+1\equiv 0\;({\rm mod}\,p)\}
\]
be the set of all Kolyvagin primes, and let $\mathcal{N}_{\rm Heeg}$ denote the set of squarefree products of priems in $\cL_{\rm Heeg}$. For any $n\in \mathcal{N}_{\rm Heeg}$, put 
\begin{equation}\label{eq:M(n)}
 M(n):=\min \{M(\ell)\;\colon\;\ell\mid n\}
\end{equation}
if $n>1$ and $M(1):=\infty$. 

Let $T=T_\mathds{1}=T_pE$. 
For $n\in \mathcal{N}$ and $k\in \Z_{\geq 1}$, we recall the definition of the derived Heegner classes 
\begin{equation}\label{eq:der-Heeg-k}
\kappa_{n,k}^{\rm Heeg}\in \rH^1_{\Fcal(n)}(K,T^{(k)}),
\end{equation}
where $T^{(k)}:=T/p^kT\simeq E[p^k]$ and $k\leq M(n)$.
%
%
We first define $\kappa_{n}^{\rm Heeg}:=\kappa_{n,M(n)}^{\rm Heeg}\in\rH^1_{\Fcal(n)}(K,T^{(M(n))})$. 

For $\ell\in \mathcal{L}_{\rm Heeg}$ we let $G_\ell = \operatorname{Gal}(K[\ell]/K[1])$, and for $n \in \mathcal{N}_{\rm Heeg}$ we set 
$\mathcal{G}(n)=\operatorname{Gal}(K[n] / K)$ and 
$G(n)=\prod_{\ell \mid n} G_{\ell}$.
Then for $m$ dividing $n$ we have the natural identification
\[
\operatorname{Gal}(K[n] / K[m]) \simeq \prod_{\ell \mid(n/m)} G_{\ell} = G(n/m)
\]
induced by the projections $\operatorname{Gal}(K[n] / K[m])\twoheadrightarrow G_\ell$.
The Kolyvagin derivative operator $D_{\ell} \in \Z_{p}[G_\ell)]$ is defined by 
\begin{equation}\label{eq:deriv}
D_{\ell}=\sum_{i=1}^{\ell} i \sigma_{\ell}^{i},
\end{equation} 
where $\sigma_{\ell}$ is a fixed generator of $G_\ell$. Let $D_{n}=\prod_{\ell \mid n} D_{\ell} \in \Z_{p}[G(n)]$. 
%
Then choosing a set $S$ of representatives for the coset space $\mathcal{G}(n)/G(n)$, one can show the inclusion 
\[
\tilde{\kappa}_n:= \sum_{\sigma \in S} \sigma D_n(P[n]) \in \bigl(E(K[n])/p^{M(n)}E(K[n])\bigr)^{\mathcal{G}(n)}
\]
(see \cite[Lem. 1.7.1]{howard}). Hence, applying the Kummer map we may consider $
\tilde{\kappa}_n\in\rH^1(K[n],T^{(M(n))})^{\mathcal{G}(n)}$. Hypothesis  \eqref{eq:intro-tor} ensures that the restriction map 
\[
\rH^1(K,T^{(M(n))}) \to \rH^1(K[n],T^{(M(n))})^{\mathcal{G}(n)}
\]
is an isomorphism, and the derived Heegner class $\kappa_{n}^{\rm Heeg}$ is defined to be  the unique class in $\rH^1(K,T^{(M(n))})$ which restricts to the image of $\tilde{\kappa}_n$ in $\rH^1(K[n],T^{(M(n))})^{\mathcal{G}(n)}$.  

The classes $\kappa_n^{\rm Heeg}$ land in the $n$-transverse Selmer group $\rH^1_{\Fcal(n)}(K,T^{(M(n))})$ (see \cite[Lem. 1.7.3]{howard}), 
and after a slight modification (not reflected in the notation) 
the resulting collection of classes
\[
\kappa^{\rm Heeg} = \{\kappa_n^{\rm Heeg}\}_{n\in\mathcal{N}_{\rm Heeg}}
\] 
forms a \emph{Kolyvagin system} for $(T, \mathcal{F}_{\rm ord}, \mathcal{L}_{\rm Heeg})$ in the sense of [\emph{op.\,cit.}, Def.~1.2.3].

Next, to define the derived Heegner classes \eqref{eq:der-Heeg-k} for $k<M(n)$ we use the following basic result.

\begin{lemma}\label{lemmamod} 
For every $n\in\cN$ and $0< i\leq k\leq M(n)$ there are natural isomorphisms
\begin{equation}\label{eq2}
\rH^1_{\Fcal(n)}(K,T^{(i)}) = \rH^1_{\Fcal(n)}(K,T^{(k)}/p^iT^{(k)})\xrightarrow{\simeq}\rH^1_{\Fcal(n)}(K,T^{(k)}[p^i])\xrightarrow{\simeq}{\rm H}_{\Fcal(n)}^1(K,T^{(k)})[p^i]\nonumber
\end{equation}
induced by the maps $T^{(i)} = T^{(k)}/p^iT^{(k)}\xrightarrow{p^{k-i}}T^{(k)}[p^i]\rightarrow T^{(k)}$.
\end{lemma}

\begin{proof}
This is \cite[Lem.~3.3.1]{eisenstein}.
\end{proof}

Using this, we define $\kappa_{n,k}^{\rm Heeg}$ to be the pre-image of $p^{M(n)-k}\kappa_{n}^{\rm Heeg}$ under the isomorphism $\rH^1_{\Fcal(n)}(K,T^{(k)})\xrightarrow{\simeq} \rH^1_{\Fcal(n)}(K,T^{(M(n))})[p^k]$ of Lemma~\ref{lemmamod}. Note that the map $\rH^1_{\Fcal(n)}(K,T^{(M(n))})\rightarrow\rH^1_{\Fcal(n)}(K,T^{(M(n))})[p^k]$ given by multiplication by $p^{M(n)-k}$ factors as
\[
\xymatrix{
\rH^1_{\Fcal(n)}(K,T^{(k)})\ar[r]^-{\simeq}&\rH^1_{\Fcal(n)}(K,T^{(M(n))}[p^k])\xrightarrow{\simeq}\rH^1_{\Fcal(n)}(K,T^{(M(n))})[p^k]\\
\rH^1_{\Fcal(n)}(K,T^{(M(n))})\ar[u]^-{{\rm mod}\,p^k}\ar[ur]_{p^{M(n)-k}}&
}
\]
where the horizontal isomorphisms are those of  Lemma~\ref{lemmamod}. Thus, equivalently, the class $\kappa_{n,k}^{\rm Heeg}$ is defined to be the image of $\kappa_{n}^{\rm Heeg}$ under the natural map induced by $T^{(M(n))}\rightarrow T^{(k)}$.


\subsubsection{The $\Lambda$-adic Kolyvagin system}\label{seckolylambda}
Let $\Lambda=\Z_p[[\Gamma]]$ be the anticyclotomic Iwasawa algebra. Put 
\[
\mathbf{T}=T_pE\otimes_{\Z_p}\Lambda
\] 
equipped with the diagonal $G_K$-action, where the Galois action on the second factor is via the tautological character $G_K\twoheadrightarrow\Gamma\hookrightarrow\Lambda^\times$.  Consider the \emph{ordinary} Selmer structure $\Fcal_\Lambda$ on $\mathbf{T}$ given by
\[
\rH^1_{\Fcal_\Lambda}(K_w,\mathbf{T})=\begin{cases}
{\rm im}\{\rH^1(K_w,{\rm Fil}_w^+(T_pE)\otimes\Lambda)\rightarrow\rH^1(K_w,\mathbf{T})\}&\textrm{if $w\mid p$,}\\[0.2em]
\rH^1_{\rm ur}(K_w,\mathbf{T})&\textrm{else.}
\end{cases}
\]
We now briefly recall the construction of the $\Lambda$-adic Heegner point Kolyvagin system, which is a collection of classes $\bkappa_{n}^{\rm Heeg}\in \rH^1(K, \mathbf{T}/p^{M(n)}\mathbf{T})= \rH^1(K, E[p^{M(n)}]\otimes\Lambda)$. 

Let $K_k$ be the subfield of $K_\infty$ of degree $p^k$ over $K$. For each $n\in\cN_{\rm Heeg}$ as above set
\[
P_k[n]:={\rm Norm}_{K[np^{d(k)}]/K_k[n]}(P[np^{d(k)}])\in E(K_k[n]),
\]
where $d(k)=\min\{d\in\Z_{\geq 0}\colon K_k\subset K[p^{d(k)}]\}$, and $K_k[n]$ is the compositum of $K_k$ and $K[n]$. Let $H_k[n]$ be the $\Z_p[{\rm Gal}(K_k[n]/K)]$-submodule of $E(K_k[n])\otimes\Z_p$ generated by 
$P_j[n]$ for $j\leq k$, and consider the $\Lambda[\mathcal{G}(n)]$-module
\[
\mathbf{H}[n]=\varprojlim_k H_{k}[n],
\] 
where the limit is taken with respect to the norm maps.
By \cite[Lem.~2.3.3]{howard}, there is a family 
\[
\{Q[n]=\varprojlim_k Q_k[n]\in\mathbf{H}[n]\}_{n\in\cN_{\rm Heeg}}
\]
such that
\begin{equation}\label{eq:multiplier}
Q_0[n]=\Phi P[n],\quad\textrm{where}\;
\Phi=\begin{cases} (p-a_p\sigma_p+\sigma_p^2)(p-a_p\sigma_p^*+\sigma_p^{*2}) & \text{$p$ splits in $K$}, \\
(p+1)^2-a_p^2 & \text{$p$ inert in $K$}.\end{cases}
\end{equation}
Here, when $p$ splits in $K$, $\sigma_p$ and $\sigma_p^*$ are the Frobenius elements in $\mathcal{G}(n)$ of the primes above $p$ in $K$.

Letting $D_n\in\Z_p[G(n)]$ be the Kolyvagin's derivative operators defined in \eqref{eq:deriv}, and choosing a set $S$ 
of representatives for $\mathcal{G}(n)/G(n)$, the class $\bkappa_{n}^{\rm Heeg}\in \rH^1(K,\mathbf{T}/p^{M(n)}\mathbf{T})$ is 
defined as the natural image of 
\[
\tilde{\bkappa}_{n}=\sum_{\sigma\in S}\sigma D_n(Q[n])\in\mathbf{H}[n]
\]
(whose reduction modulo $p^{M(n)}$ can be checked to be fixed by $\mathcal{G}(n)$) under the composite map
\[
\bigl(\mathbf{H}[n]/p^{M(n)}\mathbf{H}[n]\bigr)^{\mathcal{G}(n)}\xrightarrow{\delta(n)}\rH^1(K[n],\mathbf{T}/p^{M(n)}\mathbf{T})^{\mathcal{G}(n)}\overset{\simeq}\longleftarrow\rH^1(K,\mathbf{T}/p^{M(n)}\mathbf{T}),
\]
where $\delta(n)$ is induced by the limit of Kummer maps $\delta_k(n):E(K_k[n])\otimes\Z_p\rightarrow{\rm H}^1(K_k[n],T)$, and the second arrow is given by restriction. The latter is an isomorphism since the extensions $K[n]$ and $\Q(E[p])$ are linearly disjoint, and $E(K_\infty)[p]=0$ by hypothesis (\ref{eq:intro-tor}).

Finally, \cite[Lem.~2.3.4]{howard} \emph{et seq.} show that the classes  $\bkappa_{n}^{\rm Heeg}$ land in $\rH^1_{\Fcal_\Lambda(n)}(K,\mathbf{T}/p^{M(n)}\mathbf{T})$ and they can be slightly modified so that the result (still using the same notation) is a system $\bkappa^{\rm Heeg}=\{\bkappa_{n}^{\rm Heeg}\}_{n\in\cN_{\rm Heeg}}$ satisfying the Kolyvagin system relations.  Here $\Fcal_\Lambda(n)$ denotes the modification of the Selmer structure on $\mathbf{T}/p^{M(n)}\mathbf{T}$ induced from
$\Fcal_\Lambda$ that includes the obvious analog of the transverse local conditions at the primes $\ell\mid n$. 

\subsubsection{Kolyvagin system for anticyclotomic twists}\label{subsec:tw}

We now consider a height one prime $\mathfrak{P}\subset\Lambda$ with $\mathfrak{P}\neq p\Lambda$, and denote by $R_\p$ the integral closure of $\Lambda/\p$. 

We let $G_K$ act on $R_\p$ by the character $\alpha_{\p}: G_K\twoheadrightarrow \Gamma \to \Lambda\to R_\p$, and consider
\[
T_\p = \mathbf{T}\otimes_\Lambda R_\p = T_{\alpha_\p}
\]
with diagonal $G_K$-action. By Remark~1.2.4 and Lemma~2.2.7 in \cite{howard}, the natural map induced by $\mathbf{T}\to T_\p$ sends $\bkappa^{\rm Heeg}$ to a Kolyvagin system for $T_\p$ (for the Selmer structure $\Fcal_{\rm ord}$ defined in \S\ref{secsel}), which we denote by $\{\bkappa_{n}^{\rm Heeg}(\alpha_\p)\}_{n\in \cN_{\rm Heeg}}$. 
%
%
For the purposes of the proofs of Theorem~\ref{thmintroKoly} and Theorem~\ref{thmintroKoly-div}, we will only be 
interested in the case $\mathfrak{P} = (\gamma - \alpha(\gamma))$ where $\gamma$ is a topological generator of $\Gamma$ and 
$\alpha:\Gamma \rightarrow \Z_p^\times$ is a character. 
In this case $\alpha_\p$ is just the composition of $\alpha$ and the projection $G_K\twoheadrightarrow\Gamma$, and we write
\[
\{\bkappa_{n}^{\rm Heeg}(\alpha)\}_{n\in\mathcal{N}_{\rm Heeg}}
\]
for the corresponding Kolyvagin system.

The following simple lemma, relating the classes $\kappa_n^{\rm Heeg}$ of $\S\ref{subsubsec:HPKS}$ to certain specialisations of the $\Lambda$-adic classes $\boldsymbol{\kappa}_n^{\rm Heeg}$ will play an important role in our arguments later.

\begin{lemma}\label{lemmacongruence}
Suppose $\alpha\equiv 1\pmod{p^m}$.
For all $n\in\mathcal{N}_{\rm Heeg}$ with $M(n)\geq m$ we have the congruence
\[
\bkappa_{n}^{\rm Heeg}(\alpha)\equiv 
\begin{cases} (\alpha_p-1)^2(\beta_p-1)^2 \kappa_n^{\rm Heeg} & \text{$p$ splits in $K$} \\ 
((p+1)^2-a_p^2) \kappa_n^{\rm Heeg} & \text{$p$ inert in $K$}\end{cases} \;({\rm mod}\,p^m),
\]
where $\alpha_p,\beta_p$ are the roots of 
$x^2-a_px+p$, with $\alpha_p$ the $p$-adic unit root.
\end{lemma} 

\begin{proof}
Let $\p = (\gamma - \alpha(\gamma))$ and $\p_0 = (\gamma-1)$.
The image of $Q[n]$ under the map on $\mathbf{H}[n]$ induced by the composition $\Lambda\to \Lambda/\p\to (\Lambda/\p)/p^m(\Lambda/\p)\simeq \Z_p/p^m\Z_p$ is the same as the image induced by the composition $\Lambda \to \Lambda/\p_0 \to(\Lambda/\p_0)/p^m(\Lambda/\p_0)\simeq \Z_p/p^m\Z_p$ (since both are given by quotienting by the ideal $(\gamma-1,p^m)$).
By \eqref{eq:multiplier}, the latter image is $\Phi P[n]\;({\rm mod}\,p^m)$ from which the result follows immediately by the constructions of $\bkappa_{n}^{\rm Heeg}$ and $\kappa_n^{\rm Heeg}$ and the definition of $\Phi$.
\end{proof}

In particular, noting that $\rH^1(K,T)$ is torsion-free as a consequence of \eqref{eq:intro-tor}, Lemma~\ref{lemmacongruence} implies, combined with the Gross--Zagier formula \cite{grosszagier}, that 
\[
\bkappa_1^{\rm Heeg}(\mathds{1})\neq 0\quad\Longleftrightarrow\quad{\rm ord}_{s=1}L(E/K,s)=1.
\] 
More generally, without any conditions on the analytic rank of $E/K$, we have the following. 


\begin{thm}[Cornut--Vatsal]\label{thmCV}
For $m\gg 0$ and any $\alpha\equiv 1 \pmod{p^m}$ with $\alpha \neq 1$, the class $\bkappa_1^{\rm Heeg}(\alpha)$ is non-zero.
\end{thm} 

\begin{proof}
This is immediate from the main result of \cite{cornut} and \cite{vatsal}, showing that $\bkappa_{1}^{\rm Heeg}$ is not $\Lambda$-torsion.
\end{proof} 

\begin{rmk}
When $p$ splits in $K$, based on an extension of the Bertolini--Darmon--Prasanna formula and the nonvanishing of the $p$-adic $L$-function $\mathcal{L}_p^{\rm BDP}(f/K)$ recalled in Theorem~\ref{thm:BDP} below, one can give an alternate proof of Theorem~\ref{thmCV} (see \cite{cas-hsieh1} and \cite{bu}).
\end{rmk}


\subsection{Anticyclotomic Iwasawa Main Conjecture}

We keep the hypotheses on $(E,p,K)$ from the preceding section, and assume in addition that
\begin{equation}\label{eq:intro-spl}
\textrm{$p=v\bar{v}$ splits in $K$,}\tag{spl}
\end{equation}
with $v$ the prime of $K$ above $p$ induced by a fixed embedding $\overline{\Q}\hookrightarrow\overline{\Q}_p$.  Fix an integral ideal $\mathfrak{N}\subset\mathcal{O}_K$ with 
\begin{equation}\label{eq:N}
\mathcal{O}_K/\mathfrak{N}\simeq\Z/N\Z.\nonumber
\end{equation}

In this section we show how a certain anticyclotomic Iwasawa Main Conjecture (see Conjecture\,\ref{conj:anticyc}) yields an expression for the divisibility index of 
the Kolyvagin systems
$\{\bkappa_n^{\rm Heeg}(\alpha)\}_{n\in \cN_{\rm Heeg}}$, for $\alpha\equiv 1 \pmod{p^m}$ with $m\gg 0$, in terms of the Tamagawa numbers and the Shafarevich--Tate group of (a twist of) $E$. 

Below we shall write $a\sim_p b$ to denote that $a$ and $b$ have the same $p$-adic valuation.



%
\subsubsection{$p$-adic $L$-function} 

Let $f\in S_2(\Gamma_0(N))$ be the newform attached to 
 $E$. Recall that $\Lambda=\Z_p[[\Gamma]]$ denotes the anticyclotomic Iwasawa algebra of $K$,  and put $\Lambda^{\rm ur}=\Lambda\hat\otimes_{\Z_p}\Z_p^{\rm ur}$, 
where $\Z_p^{\rm ur}$ is the completion of the ring of integers of the maximal unramified extension of $\Q_p$. 

The next result refines the construction of a continuous anticyclototomic $p$-adic $L$-function due to Bertolini--Darmon--Prasanna \cite{BDP}. The $p$-adic measure underlying their construction was explicitly given in \cite{cas-hsieh1}, whose formulation we follow.

\begin{thm}\label{thm:BDP}
There exists an element $\Lcal_p^{\rm BDP}(f/K)\in\Lambda^{\rm ur}$ characterised by the following interpolation property: For every character $\xi$ of $\Gamma$ crystalline at both $v$ and $\bar{v}$ and corresponding to a Hecke character of $K$ of infinity type $(n,-n)$ with $n\in\Z_{>0}$ and $n\equiv 0\pmod{p-1}$, we have
\[
\Lcal_p^{\rm BDP}(f/K)^2(\xi)=\frac{\Omega_p^{4n}}{\Omega_\infty^{4n}}\cdot\frac{\Gamma(n)\Gamma(n+1)\xi(\mathfrak{N}^{-1})}{4(2\pi)^{2n+1}\sqrt{D_K}^{2n-1}}\cdot\bigl(1-a_p\xi(\bar{v})p^{-1}+\xi(\bar{v})^2p^{-1}\bigr)^2\cdot L(f/K,\xi,1),
\]  
where $\Omega_p$ and $\Omega_\infty$ are CM periods attached to $K$ as in \cite[\S{2.5}]{cas-hsieh1}. Moreover, $\Lcal_p^{\rm BDP}(f/K)$ is nonzero.
\end{thm}

\begin{proof} 
This follows from results contained in \cite[\S{3}]{cas-hsieh1}, as explained in the proof \cite[Thm.\,2.1.1]{eisenstein}.
\end{proof}

\begin{defi}\label{def:alpha_m}
A collection of characters as in Theorem \ref{thm:BDP} with $n$ essentially arbitrary can be defined as follows. Let $u\in 1 + p\Zp = 1+p\mathcal{O}_{K_v}\subset \mathcal{O}_{K_v}^\times$ be a $\Zp$-generator and let $\gamma\in \Gamma$ be a topological generator such that there exists  $h>0$ a power of $p$ for which 
$\gamma^h$ is the image of $u$ under the reciprocity map of class field theory. For $n\in\Z$,  let 
$\xi_n: \Gamma \rightarrow \Z_p^\times$ be the unique continuous character such $\xi_n(\gamma) = u^n$.  Then, for $n\equiv 0 \pmod{p-1}$, the composition of $\xi_n$ with 
the projection $G_K\twoheadrightarrow \Gamma$ is crystalline at $v$ (and hence also $\bar v$) and the corresponding  Hecke character of $K$
has infinity type $(hn,-hn)$. 
For $m>0$ we let 
\[\alpha_m = \xi_{(p-1)p^{m-1}}.
\] 
Then $\alpha_m \equiv 1 \pmod{p^m}$.
\end{defi}

\subsubsection{$\Lambda$-adic formula of Bertolini--Darmon--Prasanna formula: A consequence}\label{ss, BDP-f}

We now consider three distinguished\footnote{Some of these curves may coincide.} elliptic curves $E_0, E_1, E_\bullet$ in the $\Q$-isogeny class associated to $f$. 

{ First, $E_0/\Q$ denotes the associated optimal quotient of $X_0(N)$ with optimal modular parameterisation $\pi_0:X_0(N)\rightarrow E_0$ (so in particular the induced map $J_0(N)\rightarrow E_0$ has connected kernel). Similarly, $E_1/\Q$ is the optimal quotient of $X_1(N)$ with optimal parameterisation $\pi_1:X_1(N)\rightarrow E_1$. 
Let $\pi_{10}:X_1(N)\rightarrow X_0(N)$ be the usual map `forgetting' level structure. 
Let $\phi_{10}:E_1\rightarrow E_0$ be the cyclic isogeny such that $$\phi_{10}\circ \pi_1 = \pi_0\circ\pi_{10}.$$  The isogeny $\phi_{10}$ has constant kernel and is \'etale away from $2$.
To introduce $E_\bullet$, we consider the Galois representation $V_f$ associated to $f$ arising as the maximal quotient of $\rH^1(\overline{Y_1(N)},\Q_p(1))$ on which the Hecke operator $T_n$ acts as $a_n$ for all $n\geq 1$, where $a_n$ denotes the $n$-th Fourier coefficient of $f$. 
Let $T_f$ denote the image of $\rH^1(\overline{Y_1(N)},\Z_p(1))$ in $V_f$.  
The proof of \cite[Thm.~4]{wuthrich-int} constructs an \'etale cyclic isogeny $\phi:E_0\rightarrow A$ such that the restriction to $Y_1(N)$ of the composition
$\phi\circ\phi_{10}\circ\pi_1:X_1(N)\rightarrow A$ identifies $T_pA$ with $T_f$.  We set $E_\bullet  = A$ and write $\phi_{0\bullet}:E_0\rightarrow E_\bullet$ for $\phi$.
We let $\phi_{1\bullet}:E_1\rightarrow E_\bullet$ be the composition
\[
\phi_{1\bullet}:E_1\xrightarrow{\phi_{10}} E_0\xrightarrow{\phi_{0\bullet}} E_\bullet.
\]
This is a cyclic isogeny that is \'etale away from $2$.
The careful reader will have observed that this notation differs slightly from that in \cite{wuthrich-int}: the curve denoted $E_\bullet$ in {\em op.~cit.}~possibly differs from ours by an
isogeny of order a power of $2$.}
%

We add a subscript `$0$' to indicate the objects $T_{\alpha}$,  $T_{\alpha}^{(k)}$, $T$, $\mathbf{T}$, $\kappa^{\rm Heeg}=\{\kappa_{n}^{\rm Heeg}\}_n$, $\boldsymbol{\kappa}^{\rm Heeg}=\{\boldsymbol{\kappa}_n^{\rm Heeg}\}_n$ defined above for $E=E_0$ with the optimal quotient $\pi_0:X_0(N)\rightarrow E_0$ (i.e., $T_{0\alpha}$,  $T_{0\alpha}^{(k)}$, $T_0$, $\mathbf{T}_{0}$, $\kappa_{0}^{\rm Heeg}$, etc.), and likewise add a subscript `$\bullet$' to denote the corresponding objects for $E=E_\bullet$ with the modular parameterisation 
\[
\pi_\bullet:X_0(N)\xrightarrow{\pi_0}E_0\xrightarrow{\phi_{0\bullet}} E_\bullet.
\]
%
%
%

Let $\kappa_\infty\in\rH^1_{\Fcal_\Lambda}(K,\mathbf{T}_\bullet)$ be the $\Lambda$-adic class constructed in \cite[\S{5.2}]{cas-hsieh1}. As explained in \cite[Rem.~4.1.3]{eisenstein}, the class $\kappa_\infty$ generates the same $\Lambda$-submodule of $\rH^1_{\Fcal_\Lambda}(K,\mathbf{T}_\bullet)$ as the base class $\bkappa_{\bullet 1}^{\rm Heeg}$ of the $\Lambda$-adic Kolyvagin system $\bkappa_\bullet^{\rm Heeg}$ associated to $E_\bullet$.
%

As shown in \cite[\S{5}]{cas-hsieh1}, there exists a $\Lambda$-module map 
\[
\mathscr{L}_{T^+}:\rH^1(K_v, \Fil_v^+(T_\bullet)\otimes_{\Zp}\Lambda) \rightarrow \Lambda^{\rm ur}
\]
(a `big logarithm map') interpolating the Bloch--Kato logarithm and dual exponential maps for varying specialisations of $\Fil_v^+(T_\bullet)\otimes\Lambda$, and for which a $\Lambda$-adic extension of the main result of Bertolini--Darmon--Prasanna \cite{BDP} translates into the equality
\begin{equation}\label{eq:BDP-ERL}
\mathscr{L}_{T^+}(\loc_v(\bkappa_1^{\rm Heeg}))= \Lcal_p^{\rm BDP}(f/K)
\end{equation}
up to a $p$-adic unit (see \cite[Thm.~5.7]{cas-hsieh1}).

In the setting of this paper, $\mathscr{L}_{T^+}$ induces an injection $\rH^1(K_v, \Fil_v^+(T_\bullet)\otimes_{\Zp}\Lambda)\otimes_\Lambda\Lambda^{\rm ur}\hookrightarrow\Lambda^{\rm ur}$ with finite cokernel.  
From this we deduce the following.

\begin{lemma}\label{lem:erl}
Let $\al: \Gamma \to \Z_p^{\times}$ be a character as in Theorem~\ref{thm:BDP} with $n\geq 0$ and such that $\Lcal_p^{\rm BDP}(f/K)(\al^{-1})\neq 0$. Then $\boldsymbol{\kappa}_{1\bullet}^{\rm Heeg}(\alpha)\neq 0$ and $\rH^1_{\Fcal_{\rm ord}}(K,T_{\bullet\alpha})$ is a free $\Z_p$-module of rank one. Moreover, if $\alpha\equiv 1\;({\rm mod}\,p^m)$ with $m\gg 0$, then
\[
p^{t_\alpha} \cdot\#\rH^0(\Q_p,E_\bullet[p^\infty]) =\# \bigl({\rm H}^1_{\Fcal_{\rm ord}}(K,T_{\bullet\al})/\Z_p\cdot\bkappa^{\rm Heeg}_{\bullet 1}(\al)\bigr)  \cdot \#{\rm coker}(\loc_{\bullet{v}}),
\]
where 
\[
t_\alpha = \mathrm{length}_{\Z_p^{\rm ur}}\bigl(\Z_p^{\rm ur}/\Lcal_p^{\rm BDP}(f/K)(\al^{-1}) \bigr)
\]
and
$\loc_{\bullet{v}}:\rH^1_{\Fcal_{\rm ord}}(K,T_{\bullet\al})\to \rH^1_{\Fcal_{\rm ord}}(K_v,T_{\bullet\al})/\rH^1(K_v,T_{\bullet\al})_{\rm tors}=:\rH^1(K_v, {\rm Fil}_v^+(T_{\bullet\al}))_{/\rm tors}$.
\end{lemma}

\begin{proof}
By the interpolation property of the map $\mathscr{L}_{T^+}$ in \cite[Thm.~5.1]{cas-hsieh1}, the non-vanishing of $\Lcal_p^{\rm BDP}(f/K)(\al^{-1})$  implies that ${\rm loc}_v(\bkappa^{\rm Heeg}_{\bullet{1}}(\al))$ has nonzero image in $\rH^1(K_v,{\rm Fil}_v^+(T_{\bullet\al}))$ (in particular, $\boldsymbol{\kappa}_{\bullet{1}}^{\rm Heeg}\neq 0$). 
Furthermore the cokernel of the specialization map $\rH^1(K_v,\Fil_v^+(T_\bullet)\otimes_{\Zp}\Lambda)\rightarrow \rH^1(K_v,\Fil_v^+(T_{\bullet\al}))$ 
is naturally isomorphic to $\rH^2(K_v,\Fil_v^+(T_\bullet)\otimes_{\Zp}\Lambda)[\gamma-\alpha(\gamma)]$, whose order is easily computed via local duality to equal
$\#\rH^0(\Q_p,E_\bullet[p^\infty])$ if $m\gg 0$. Since $\rH^1(K_v,\Fil_v^+(T_{\bullet\al})$ is a free $\Zp$-module of rank one, it follows that
\begin{equation}\label{eq:loc-free}
\#\bigl(\rH^1(K_v,\Fil_v^+(T_{\bullet\al}))/\Zp\loc_v (\bkappa^{\rm Heeg}_{\bullet{1}}(\alpha))\bigr) = p^{t_\alpha}  \cdot\#\rH^0(\Q_p,E_\bullet[p^\infty]).
\end{equation}

By \cite[Thm.~6.1.1]{eisenstein_cyc}, the non-vanishing of $\boldsymbol{\kappa}_{\bullet 1}^{\rm Heeg}(\alpha)$ implies that $\rH^1_{\Fcal_{\rm ord}}(K,T_{\bullet\al})$ is a free $\Z_p$-module of rank one. Thus, the maps $\loc_{\bullet{v}}$ induces an injection 
\[
\rH^1_{\Fcal_{\rm ord}}(K,T_{\bullet\al})/\Zp\bkappa^{\rm Heeg}_{\bullet{1}}(\alpha) \hookrightarrow \rH^1(K_v,\Fil_v^+(T_{\bullet\al}))/\Zp\loc_{\bullet v}(\bkappa^{\rm Heeg}_{\bullet{1}}(\alpha)),
\]
from which it follows that
\begin{equation}\label{eq:glob-free}
\#\bigl(\rH^1(K_v,\Fil_v^+(T_{\bullet\al}))/\Zp\loc_{\bullet v} (\bkappa^{\rm Heeg}_{\bullet 1}(\alpha))\bigr) = \#\bigl(\rH^1_{\Fcal_{\rm ord}}(K,T_{\bullet\al})/\Zp \bkappa^{\rm Heeg}_{\bullet 1}(\alpha)\bigr)\cdot \#{\rm coker}(\loc_{\bullet v}).
\end{equation}
Combining \eqref{eq:loc-free} and \eqref{eq:glob-free} yields the lemma.
\end{proof}

\begin{rmk}\label{rmk:index-HeegnerPt}
In the case that $\alpha = \mathds{1}$, we have ${\bkappa}_{1}^{\rm Heeg}(\mathds\alpha) = (1-\alpha_p)^2(1-\beta_p)^2{\kappa}_1^{\rm Heeg}$. So, noting that
$\#(\Zp/(1-\alpha_p)(1-\beta_p)) = \#(\Zp/(1-\alpha_p)) = \#\rH^0(\Qp,E[p^\infty])$, the formula in Lemma~\ref{lem:erl} can be rewritten
as
\[
\#\left (\Z_p/\Lcal_p^{\rm BDP}(f/K)(\mathds{1}) \right)  =\# \bigl({\rm H}^1_{\Fcal_{\rm ord}}(K,T_{\bullet\al})/\Z_p\cdot\kappa^{\rm Heeg}_{\bullet{1}}\bigr)  \cdot \#{\rm coker}(\loc_{\bullet{v}})
\cdot\#\rH^0(\Q_p,E_\bullet[p^\infty]).
\]
\end{rmk}

The formula of Lemma~\ref{lem:erl} links the specialisation of the $p$-adic $L$-function at $\alpha$ with the base class in the Heegner point Kolyvagin system $\boldsymbol{\kappa}_\bullet^{\rm Heeg}$ for $\pi_\bullet:X_0(N)\rightarrow E_\bullet$. The result implies that the same relation holds for the optimal quotient $\pi_0:X_0(N)\rightarrow E_0$ and the associated $\bkappa_{0}^{\rm Heeg}=\{\bkappa_{0n}^{\rm Heeg}\}_n$.


\begin{lemma}\label{lemma:erlinv} 
Let the notations and hypotheses be as in Lemma~\ref{lem:erl}. Then $\bkappa_{01}^{\rm Heeg}(\alpha)\neq 0$, $\rH^1_{\Fcal_{\rm ord}}(K,T_{0\alpha})$ is $\Z_p$-free of rank one, and we have
\[
\#\bigl({\rm H}^1_{\Fcal_{\rm ord}}(K,T_{\bullet\alpha})/\Z_p\cdot\bkappa^{\rm Heeg}_{\bullet{1}}(\alpha)\bigr)  \cdot \#{\rm coker}(\loc_{\bullet{v}}) =  \# \bigl({\rm H}^1_{\Fcal_{\rm ord}}(K,T_{0\alpha})/\Z_p\cdot\bkappa^{\rm Heeg}_{0 1}(\alpha)\bigr) \cdot \#{\rm coker}(\loc_{0 v}),
\]
where $\loc_{0v}:\rH^1_{\Fcal_{\rm ord}}(K,T_{0\alpha})\to \rH^1_{\Fcal_{\rm ord}}(K_v, T_{0\alpha})/\rH^1(K_v,T_{0\alpha})_{\rm tors}=:\rH^1(K_v,\Fil_v^+(T_{0\alpha}))_{/ \rm tors}$. 

\end{lemma}

\begin{proof}
%
By construction, there is an \'etale isogeny $\phi_{0\bullet}: E_0 \to E_\bullet$ such that the induced injective map
\begin{equation}\label{eq:phi0bullet}
\phi_{0\bullet}: \rH^1_{\Fcal_{\rm ord}}(K,T_{0\alpha})\to   \rH^1_{\Fcal_{\rm ord}}(K,T_{\bullet\alpha})
\end{equation}
satisfies $\phi_{0\bullet}(\bkappa_{0 1}^{\rm Heeg}(\alpha))=\bkappa^{\rm Heeg}_{\bullet 1}(\alpha)$. Thus the nonvanishing of $\bkappa^{\rm Heeg}_{\bullet 1}(\alpha)$ implies 
$\bkappa_{0 1}^{\rm Heeg}(\alpha)\neq 0$, and by \cite[Thm.~6.1.1]{eisenstein_cyc} it follows that ${\rm H}^1_{\Fcal_{\rm ord}}(K,T_{\bullet\alpha})$ and ${\rm H}^1_{\Fcal_{\rm ord}}(K,T_{0\alpha})$ are $\Z_p$-free of rank one. Hence \eqref{eq:phi0bullet} is injective, and letting $d\in\Z_{\geq 0}$ be such that $[\rH^1_{\Fcal_{\rm ord}}(K,T_{\bullet\alpha}): \phi_{0\bullet}(\rH^1_{\Fcal_{\rm ord}}(K,T_{0\alpha}))]=p^d$, we deduce
\begin{equation}\label{eq:indexfE}
\#\bigl(\rH^1_{\Fcal_{\rm ord}}(K,T_{\bullet\alpha}))/\Zp\cdot \bkappa_{\bullet 1}^{\rm Heeg}(\alpha)\bigr) = p^d\cdot \#\bigl(\rH^1_{\Fcal_{\rm ord}}(K,T_{0\alpha}))/\Zp\cdot \bkappa_{01}^{\rm Heeg}(\alpha)\bigr).
\end{equation}
Furthermore, since ${\rm Fil}_v^+(T_{0\alpha})={\rm Fil}_v^+(T_{\bullet\alpha})$ by the etaleness at $p$ of $\phi_{0\bullet}:E_0\rightarrow E_\bullet$, we have the following commutative diagram:
\[
\begin{tikzcd}
\rH^1_{\Fcal_{\rm ord}}(K,T_{\bullet\alpha}) \arrow{r}{\loc_{\bullet v}} &\rH^1(K_v, \Fil_v^+(T_{\bullet\alpha}))_{/ \rm tors} \\
\rH^1_{\Fcal_{\rm ord}}(K,T_{0\alpha}) \arrow{u}{\phi_{0\bullet}}\arrow{r}{\loc_{0 v}} & \rH^1(K_v,\Fil_v^+(T_{0\alpha}))_{/ \rm tors} \arrow[equal]{u}
\end{tikzcd}
\]
Thus we obtain 
\begin{align*}
{\rm coker}(\loc_{0v})&=\rH^1(K_v,\Fil_v^+(T_{\bullet \alpha}))_{/ \rm tors}/{\rm loc}_{\bullet v}(\phi_{0\bullet}(\rH^1_{\Fcal_{\rm ord}}(K,T_{0\alpha})))\\
&=\rH^1(K_v,\Fil_w^+(T_{\bullet\alpha}))_{/ \rm tors}/p^d\rH^1_{\Fcal_{\rm ord}}(K,T_{\bullet\alpha}),
\end{align*}
and hence the equality 
\begin{equation}\label{eq:indexloc}
\#{\rm coker}(\loc_{0v})=p^d\cdot\#{\rm coker}(\loc_{\bullet v}).
\end{equation}
Combining \eqref{eq:indexfE} and \eqref{eq:indexloc} yields the desired result.
\end{proof}

\subsubsection{Anticyclotomic Selmer group}

Now we let $E/\Q$ be any elliptic curve as in $\S\ref{subsec:KS-Heeg}$. 

Let $\Sigma$ be a finite set of places of $K$ containing $\infty$ and the primes dividing $Np$, and such that all finite primes in $\Sigma$ split in $K$.
For a  $\Z_p$-module $A$, we let $A^\vee={\rm Hom}_{\rm cts}(A,\Q_p/\Z_p)$ denote its Pontryagin dual. Put
\[
M:=T\otimes_{\Z_p}\Lambda^\vee=T_pE\otimes_{\Z_p} \Lambda^{\vee},
\]
where $\Lambda^{\vee}$ is equipped with a $G_K$-action via the inverse of the tautological character $G_K\rightarrow\Lambda^\times$. 

\begin{defi}Let $G_{K,\Sigma}$ denote the Galois group of the maximal extension of $K$ unramified outside $\Sigma$. We define the \emph{Greenberg Selmer group} of $M$ as follows:
\begin{equation}\notag
\rH^1_{\Fcal_{\rm Gr}}(K, M)=\ker \biggl\{\rH^1(G_{K,\Sigma},M) \to {\rH^1(K_v,M)}\times \prod_{w\in\Sigma,w\nmid p}\rH^1(K_w, M)\biggr\}.
\end{equation}
Similarly, we let $\mathcal{F}_{\rm Gr}$ denote the Selmer structure on $V_\alpha$ defined for $w\in\Sigma$ by
\[
\rH^1_{\Fcal_{\rm Gr}}(K_w,V_\alpha):=
\begin{cases}
0&\textrm{if $w=v$,}\\[0.2em]
\rH^1(K_{\bar{v}},V_\alpha) & \textrm{if $w=\bar{v}$,}\\[0.2em]
\rH^1_f(K_w,V_\alpha)& \textrm{else,}
\end{cases}
\]
and let $\Fcal_{\rm Gr}$ also denote the resulting  Selmer structure on  $W_\alpha$, yielding the Selmer group $\rH^1_{\Fcal_{\rm Gr}}(K,W_\alpha)$ (so in particular, its classes restricted to $\bar{v}$ land in the maximal divisible submodule $\rH^1(K_{\bar{v}},W_\alpha)_{\rm div}$).
\end{defi}

Put
\begin{align*}
\X_{\rm Gr}(E/K_\infty):=\rH^1_{\Fcal_{\rm Gr}}(K, M)^{\vee}.
\end{align*}
It is a standard fact that $\X_{\rm Gr}(E/K_\infty)$ is a finitely generated $\Lambda$-module. 
%
We next recall a twisted anticyclotomic variant of the Euler characteristic calculation for Selmer groups \cite[\S{4}]{greenberg-cetraro}. 

\begin{thm}[Anticyclotomic control theorem]\label{controlthm}
Let $\al:\Gamma\rightarrow\Z_p^\times$ be a character as in Theorem~\ref{thm:BDP} such that $\Lcal_p^{\rm BDP}(f/K)(\al^{-1})\neq 0$ and $\alpha\equiv 1\;({\rm mod}\,p^m)$. Then $\X_{\rm Gr}(E/K_\infty)$ is $\Lambda$-torsion, and if $\Fcal_E\in\Lambda$ generates ${\rm char}_\Lambda(\mathfrak{X}_{\rm Gr}(E/K_\infty))$, then for $m\gg 0$ we have
\[
\#\bigl(\Z_p /\Fcal_E(\al^{-1})\bigr) = 
\#\sha(W_{\al^{-1}}/K)\cdot \#{\rm coker}(\loc_{v})^2 \cdot \prod_{w\mid N}c_w^{(p)}(\al^{-1})\cdot\#\rH^0(\Q_p,E[p^\infty])^2,
\]
where: 
\begin{itemize}
\item $c_w^{(p)}(\al^{-1})=\#\rH^1_{\rm ur}(K_w,W_{\al^{-1}})$ 
is the $p$-part of the Tamagawa number, 
\item $\sha(W_{\al^{-1}}/K)=\rH^1_{\Fcal_{\rm ord}}(K,W_{\al^{-1}})/\rH^1_{\Fcal_{\rm ord}}(K, W_{\al^{-1}})_{\rm div}$ is the Bloch--Kato Shafarevich--Tate group, 
\item $\loc_{v}:\rH^1_{\Fcal_{\rm ord}}(K,T_{\al})\to \rH^1(K_v, \Fil_w^+(T_{\al}))_{/\rm tors}$ is the restriction map at $v$ composed with the projection.
\end{itemize}
\end{thm}

\begin{proof}
This follows from $\al$-twisted analogues of \cite[Prop.\,3.2.1]{jsw} and \cite[Thm.\,3.3.1]{jsw}. 

More precisely, we begin by verifying the following properties (i)--(iii): 
\begin{itemize}
\item[(i)] $\rH^1_{\Fcal_{\rm ord}}(K, W_{\al^{-1}})_{\rm div}\simeq \Q_p/\Z_p$,
\item[(ii)] $\rH^1_{{\rm ord}}(K_v, W_{\al^{-1}})\simeq  \Q_p/\Z_p$, 
\item[(iii)] {$\loc_v: \rH^1_{\Fcal_{\rm ord}}(K, W_{\al^{-1}})_{\rm div} \to \rH^1_{\Fcal_{\rm ord}}(K_v, W_{\al^{-1}})$ is surjective.}
\end{itemize}

The Heegner class $\bkappa_1^{\rm Heeg}$ is the image of $\bkappa_{01}^{\rm Heeg}$ under the map $\rH^1(K,\mathbf{T}_0)\rightarrow\rH^1(K,\mathbf{T})$ induced by an isogeny $E_0\rightarrow E$, and similarly as in the proof of Lemma~\ref{lemma:erlinv}, the non-vanishing of $\Lcal_p^{\rm BDP}(f/K)(\al^{-1})$ implies the non-vanishing of $\bkappa^{\rm Heeg}_1(\al)$. The proof of (i) thus follows from \cite[Thm.~6.1.1]{eisenstein_cyc}. 
For (ii), since $\rH^1_{\rm ord}(K_v,W_{\al^{-1}})$ is divisible by definition, it suffices to show that it has $\Z_p$-corank one. But this is immediate from local Tate duality and Tate's local Euler characteristic formula, noting that because $p$ is a prime of good reduction and $\alpha$ is pure of weight $0$, the invariant subspaces $\rH^0(K_v,{\rm Fil}_v^+(W_{\al^{-1}}))$ and $\rH^0(K_v,{\rm Fil}_v^+(W_{\al^{-1}})^\vee(1))$ are both finite.
%
Finally, also similarly as in the proof of Lemma~\ref{lemma:erlinv}, the non-vanishing of $\Lcal_p^{\rm BDP}(f/K)(\al^{-1})$ implies that the map $\rH^1_{\Fcal_{\rm ord}}(K,T_{\al})\rightarrow\rH^1_{\Fcal_{\rm ord}}(K_v,T_{\al})_{/{\rm tors}}$ is nonzero, from which (iii) follows. The same proof as \cite[Prop.~3.2.1]{jsw} thus yields that $\rH^1_{\Fcal_{\rm Gr}}(K,W_{\al^{-1}})$ is finite, with
\begin{equation}\label{eq:321}
\#\rH^1_{\Fcal_{\rm Gr}}(K,W_{\al^{-1}})=\#\sha(W_{\al^{-1}}/K)\cdot\#{\rm coker}({\rm loc}_{v})^2.
\end{equation}

Now put $M_{\al^{-1}}=M\otimes_{\Z_p}\Z_p(\al^{-1})=T_pE(\alpha^{-1})\otimes_{\Z_p}\Lambda^\vee$ and $\X_{\rm Gr}(E(\al^{-1})/K_\infty)=\rH^1_{\Fcal_{\rm Gr}}(K,M_{\al^{-1}})^\vee$, and denote by ${\rm Tw}_{\al^{-1}}:\Lambda\rightarrow\Lambda$ the $\Z_p$-linear isomorphism given by $\gamma\mapsto\al^{-1}(\gamma)\gamma$ for $\gamma\in\Gamma$. In view of (i)--(iii), 
the hypotheses of Theorem~3.3.1 of \cite{jsw} are satisfied by $V_{\al^{-1}}$, from which we conclude\footnote{taking $\Sigma=\emptyset$ in \emph{loc.\,cit.}} that $\X_{\rm Gr}(E(\al^{-1})/K_\infty)$ is $\Lambda$-torsion, and letting $\Fcal_{E(\al^{-1})}\in\Lambda$ be a generator of ${\rm char}_\Lambda(\X_{\rm Gr}(E(\al^{-1})/K_\infty))$ we have
\begin{equation}\label{eq:331}
\#\bigl(\Z_p /\Fcal_{E(\al^{-1})}(0)\bigr)= 
\#\rH^1_{\Fcal_{\rm Gr}}(K,W_{\al^{-1}})\cdot \prod_{w\mid N}c_w^{(p)}(\al^{-1})\cdot\#\rH^0(K_v,W_{\al^{-1}})\cdot\#\rH^0(K_{\bar{v}},W_{\al^{-1}}).
\end{equation}
Since clearly ${\rm Tw}_{\al^{-1}}(\Fcal_{E})$ gives a generator of ${\rm char}_\Lambda(\X_{\rm Gr}(E(\al^{-1})/K_\infty)$ and $\rH^0(\Q_p,W_{\al^{-1}})=\rH^0(\Q_p,E[p^\infty])$ for $\alpha$ sufficiently close to $1$, the combination of (\ref{eq:321}) and (\ref{eq:331}) yields the equality in the theorem.
\end{proof}

As in Theorem~\ref{controlthm}, denote by $c_w^{(p)}(\al)$ the $p$-part of the Tagamawa number of $W_{\al}$ at $w$, and similarly put $c_w^{(p)}=\#\rH^1_{\rm ur}(K_w,W)$ (see also Remark~\ref{rem:tam}).

\begin{lemma}\label{lemma:tam}
Let $w\nmid p$ be a finite prime of $K$. If  $\al:\Gamma\rightarrow\Z_p^\times$ satisfies $\al\equiv 1\pmod{ p^m}$, then $c_w^{(p)}(\al)\equiv c_w^{(p)}\pmod{p^m}$. 
\end{lemma}

\begin{proof}
By definition, we have 
\[
\rH^1_{\rm ur}(K_w,W_{\al^{}})=\rH^1(\mathbb{F}_w,(W_{\al^{}})^{I_w})\simeq(W_{\al^{}})^{I_w}/({\rm Fr}_w-1)(W_{\al^{}})^{I_w},
\] 
where $\mathbb{F}_w$ is the residue field of $K_w$ and $I_w\subset G_{K_w}$ is the inertia subgroup at $w$. Since $\al$ is unramified at $w$, we have $(W_{\al^{}})^{I_w}=W^{I_w}$, 
and using that $\al \equiv 1 \pmod{p^m}$, we obtain 
\[
\rH^1_{\rm ur}(K_w,W_{\al^{}})/p^m\rH^1_{\rm ur}(K_w,W_{\al^{}})\simeq \rH^1_{\rm ur}(K_w,W)/p^m \rH^1_{\rm ur}(K_w,W),
\]
from which the desired result follows.
\end{proof}

\begin{rmk}\label{rmktam} 
If $\al:\Gamma\rightarrow\Z_p^\times$ satisfies $\al\equiv 1 \pmod{p^m}$ with $m> \sum_{w\mid N}{\rm ord}_p(c^{(p)}_w)$, Lemma~\ref{lemma:tam} implies the equality 
\[
\prod_{w\mid N}c_w^{(p)}(\al) = \prod_{w\mid N} c^{(p)}_w.
\]
\end{rmk}

\subsubsection{The Main Conjecture}\label{subsec:anticycmc}
We recall the statement of the anticyclotomic Main Conjecture for $\mathcal{L}_p^{\rm BDP}(f/K)^2$. Note that this can be seen as a special case of the Iwasawa Main Conjecture for $p$-adic deformations of motives formulated by Greenberg \cite{greenberg-motives}.

\begin{conj}
\label{conj:anticyc}
Suppose $K$ satisfies {\rm (\ref{eq:intro-Heeg})}, {\rm (\ref{eq:intro-disc})}, and {\rm (\ref{eq:intro-spl})}.
Then $\X_{\rm Gr}(E/K_\infty)$ is $\Lambda$-torsion, and
\[
{\rm char}_\Lambda(\X_{\rm Gr}(E/K_\infty))\Lambda^{\rm ur}=(\Lcal_p^{\rm BDP}(f/K)^2)
\]
as ideals in $\Lambda^{\rm ur}$.
\end{conj}


\begin{rem}\label{rem:ac-rational}
In the following, we shall refer to the statement of Conjecture~\ref{conj:anticyc} with $\Lambda^{\rm ur}$ replaced by $\Lambda^{\rm ur}\otimes\Q_p$ as the \emph{rational anticyclotomic Main Conjecture}. 
\end{rem}


We record a consequence of some of the previous results in the case $E=E_0$.

\begin{cor}\label{cor:index-sha}
Let $\al:\Gamma\rightarrow\Z_p^\times$ be a character as in 
Theorem~\ref{thm:BDP} such that $\Lcal_p^{\rm BDP}(f/K)(\al^{-1})\neq 0$ and $\alpha\equiv 1\pmod{p^m}$, and let $\bkappa_{01}^{\rm Heeg}(\alpha)$ be the bottom class of the Kolyvagin system $\bkappa^{\rm Heeg}_0(\alpha)=\{\bkappa_{0n}^{\rm Heeg}(\alpha)\}_n$ for the optimal quotient $\pi_0:X_0(N)\rightarrow E_0$ associated to $f$. If Conjecture~\ref{conj:anticyc} holds, then for $m\gg 0$ we have
\[
\#({\rm H}^1_{\Fcal_{\rm ord}}(K,T_{\al})/\Z_p\cdot\bkappa^{\rm Heeg}_{01}(\al))^2 =\#\sha(W_{\al^{-1}}/K)\cdot\prod_{w\mid N}c_w^{(p)}(\al^{-1}) \cdot \#\rH^0(\Q_p,E_0[p^\infty])^4.
\]
\end{cor}

\begin{proof}
Since $\#\rH^0(\Q_p,E_0[p^\infty])=\#(\Z_p/(1-\alpha_p))$, the combination of  Lemma~\ref{lemma:erlinv} and Theorem~\ref{controlthm} yields the result.
\end{proof}

We conclude this section by recording some results on Conjecture~\ref{conj:anticyc}. The first cases of the conjecture were proved by X.\,Wan \cite{wan-heegner}, assuming that the $G_K$-action on $T_pE$ is surjective and some mild ramification hypotheses on $E[p]$. 
More recently, we have the following.

\begin{thm}
\label{thm:howard-HP}
Suppose $K$ satisfies {\rm (\ref{eq:intro-Heeg})}, {\rm (\ref{eq:intro-disc})}, and {\rm (\ref{eq:intro-spl})}.
\begin{itemize}
\item[(i)] If $E[p]^{ss}=\mathbb{F}_p(\phi)\oplus\mathbb{F}_p(\psi)$ as $G_\Q$-modules with $\phi\vert_{G_p}\neq\mathds{1},\omega$, then Conjecture~\ref{conj:anticyc} holds.
\item[(ii)] If $p>3$ satisfies {\rm (\ref{eq:irr})}, then the rational anticyclotomic Main Conjecture holds. 
\item[(iii)] If $p>3$ satisfies {\rm (\ref{eq:irred})}, then Conjecture~\ref{conj:anticyc} holds.
\end{itemize}
\end{thm}

\begin{proof}
This is shown in \cite{eisenstein_cyc} in situation (i), and in \cite{bcs} in situations (ii) and (iii).
\end{proof}


\subsection{The Kolyvagin system bound with error term}

Again we let $E/\Q$ be any elliptic curve as in $\S\ref{subsec:KS-Heeg}$ in the isogeny class associated to $f$. For a positive integer $e$ put
\[ 
\mathcal{L}_e=\{\ell\in\mathcal{L}_0\;\colon\;a_\ell\equiv\ell+1\equiv 0\;({\rm mod}\,p^e)\}.
\] 
We shall need the following mild extension of \cite[Thm.~6.1.1]{eisenstein_cyc}. 

\begin{thm}\label{thm:bound}
Suppose $\mathcal{L}\subset\mathcal{L}_{\rm Heeg}$ satisfies $\mathcal{L}_e\subset\mathcal{L}$ for $e\gg 0$. 
Let $\alpha:\Gamma\to \Z_p^{\times}$ be an anticyclotomic character such that $\alpha\equiv 1\;({\rm mod}\,p^m)$. Suppose that there is a collection of cohomology classes 
\[
\{\tilde{\kappa}_n\in\rH^1(K,T_\alpha/I_nT_\alpha)\;\colon\;n\in\mathcal{N}(\mathcal{L})\}
\]
with $\tilde{\kappa}_1\neq 0$ and that there is an integer $t\geq 0$, independent of $n$, such that 
$\{p^t\tilde{\kappa}_n\}_{n\in\mathcal{N}(\mathcal{L})}$ is a Kolyvagin system for $T_\alpha$ and the Selmer structure $\Fcal_{\rm ord}$. 
Then there exist non-negative integers $\CM$ and $\CE$, depending only on $T_pE$, such that if
$m\geq \CM$ then $\rH^1_{\Fcal_{\rm ord}}(K,T_{\alpha})$ has $\Z_p$-rank one, and there is a finite $\Z_p$-module $M$ such that
\[
{\rm H}^1_{\Fcal_{\rm ord}}(K,W_{\alpha^{-1}})\simeq \Q_p/\Z_p \oplus M\oplus M
\]
with 
\[
{\rm length}_{\Z_p}(M)\leq \ind(\tilde{\kappa}_1)+\CE,
\]
where $\ind(\tilde{\kappa}_1)={\rm length}_{\Z_p}\bigl({\rm H}^1_{\Fcal_{\rm ord}}(K,T_{\alpha})/\Z_p\cdot\tilde{\kappa}_{1}\bigr)$. Moreover, if {\rm (\ref{eq:irred})} holds, then $\mathcal{E}=0$.
\end{thm}

\begin{proof}
Take $f\gg 0$ so that $\tilde{\kappa}_1\not\equiv 0\;({\rm mod}\,p^f)$ and  such that $\tilde{\mathcal{L}}:=\mathcal{L}_{f+t}$ is contained in $\mathcal{L}$, and put $\tilde{\mathcal{N}}=\mathcal{N}(\tilde{\mathcal{L}})$. As in the proof of \cite[Thm.~2.2.2]{howard-gl2-type}, the collection of classes $\{\tilde{\kappa}_n\}_{n\in\tilde{\mathcal{N}}}$ gives a (non-trivial) Kolyvagin system for $(T/p^fT,\mathcal{F}_{\rm ord},\tilde{\mathcal{L}})$, and a slight modification of the argument yielding \cite[Thm.~6.1.1]{eisenstein_cyc} applied to this system 
yields the result. More precisely, the rank one statement follows as in \cite[Thm 3.3.8]{eisenstein} choosing $k\gg 0$ and larger than $t+f$. For the bound, it is then enough to have a Kolyvagin system modulo $p^f$: the choice of the integer $k$ as in (6.5) of \cite{eisenstein_cyc} and larger than $t+f$ guarantees, for example, that, when working with the image of the system in $\rH^1(K, T^{(k)}_{\alpha})$, we find
\[
\ord(\loc_{\ell}(p^t\tilde{\kappa}_\ell))=\ord(\loc_{\ell}(p^t\tilde{\kappa}_1))\geq \ord(p^t\tilde{\kappa}_1) -  e \gneq 0,
\] 
where $\ell$ is the first prime produced via the \v{C}ebotarev argument in the inductive argument.
Therefore $\ord(\loc_{\ell}(\tilde{\kappa}_\ell))=\ord(\loc_{\ell}(\tilde{\kappa}_1)) \geq  \ord(\tilde{\kappa}_1) -  e$, which is precisely the key relation we use in the proof. Hence the argument to produce the sequence of integers $1=n_0, \ell=n_1, \dots, n_\rho \in \mathcal{N}^{(k)}$ satisfying (a)-(g) carries through.

The claim that $\mathcal{E}=0$ when (\ref{eq:irred}) holds follows from \cite[Rem.~3.3.5]{eisenstein}.
%
\end{proof}


\begin{rem}
Under (\ref{eq:irred}), Theorem~\ref{thm:bound} recovers \cite[Thm.~2.2.2]{howard-gl2-type}. Still under a big image hypothesis, a similar result is given \cite[Thm.~9.6.1]{LZ-quadratic} for Kolyvagin systems over the cyclotomic Iwasawa algebra.
%
%
\end{rem}

\section{Proofs of Theorem~\ref{thmintroKoly} and Theorem~\ref{thmintroKoly-div}}\label{sec:Kolyproof}


\subsection{Kolyvagin's conjecture}\label{sec:Kolyf}  

\begin{proof}[Proof of Theorem~\ref{thmintroKoly}]

We shall first proof the result for the Heegner point Kolyvagin system $\kappa_0^{\rm Heeg}=\{\kappa_{0n}^{\rm Heeg}\}_n$ attached to the optimal quotient $\pi_0:X_0(N)\rightarrow E_0$.

Assume that the rational anticyclotomic Main Conjecture  
holds (see Remark~\ref{rem:ac-rational}). In particular, letting $\Fcal_p(E_0/K)\in\Lambda$ be any generator of ${\rm char}_\Lambda(\mathfrak{X}_{\rm Gr}(E_0/K_\infty))$ we have the `lower bound' divisibility\footnote{Note 
that the proof of Theorem~\ref{thmintroKoly} only requires the divisibility $\Lcal_p^{\rm BDP}(f/K)^2\mid\Fcal_p(E/K)$ in $\Lambda^{\rm ur}[1/p]$.}
\begin{equation}\label{eq:rational-div}
\Fcal_p(E_0/K)\cdot p^R=\Lcal_p^{\rm BDP}(f/K)^2\cdot h
\end{equation}
for some $R\geq 0$ and $h\in\Lambda^{\rm ur}$. Let also
$$x:= v_p(\#\rH^0(\Q_p,E_0[p^\infty]))= v_p(1-\alpha_p).$$ 

Assume for contradiction that
 $\kappa_{0n}^{\rm Heeg}=0$ for all $n\in\mathcal{N}_{\rm Heeg}$. Take $t\in \Z_{>0}$ such that
\begin{equation}\label{eq:t}
t> \tfrac{1}{2}\sum_{w\mid N}\ord_p(c_w^{(p)}) +R+ \CE,
\end{equation}
where $\CE$ is the error term from Theorem \ref{thm:bound}. Consider a  character $\alpha=\alpha_{m}: \Gamma \to \Z_p^{\times}$ as in 
Definition \ref{def:alpha_m}
such that the following conditions hold:
\begin{itemize}
\item[(a)] $\Lcal_p^{\rm BDP}(f/K)(\al^{-1})\neq 0$, 
\item[(b)] $m> \CM + 2t+2x$,
\item[(c)] The conclusion of Corollary~\ref{cor:index-sha} holds.
\end{itemize}
By the last claim in Theorem~\ref{thm:BDP}, 
we can always find such $\alpha$. In particular, by Lemma~\ref{lemma:erlinv} we then have $\bkappa_{01}^{\rm Heeg}(\alpha)\neq 0$, and moreover  
\[
\prod_{w\mid N}c_w^{(p)}(\alpha) = \prod_{w\mid N} c_w^{(p)}
\]
by Lemma \ref{lemma:tam} and Remark~\ref{rmktam}.

Put $\mathcal{N}_m=\mathcal{N}(\mathcal{L}_m)$, 
so $M(n)\geq m$ for all $n\in\cN_m$. Since we assume that all $\kappa_{0n}^{\rm Heeg}$ vanish, by Lemma~\ref{lemmacongruence} and the assumption $m>t+2x$ we find that $\bkappa_{0n}^{\rm Heeg}(\alpha)\equiv 0\;({\rm mod}\,p^{t+2x})$ for every $n\in\mathcal{N}_m$. We can therefore consider a collection of cohomology classes $\{\tilde{\kappa}_{0n,\alpha}\in\rH^1(K,T_{0\alpha}/p^{M(n)}T_{0\alpha})\}_{n\in\mathcal{N}_m}$ such that
\[
p^{t+2x}\cdot\tilde{\kappa}_{0n,\alpha}=\bkappa_{0n}^{\rm Heeg}(\alpha).
\] 
Note that these classes are not canonically defined, but their orders (and that of their localisations) are.
Since $\rH^1(K,T_{0\alpha})$ is torsion-free by hypothesis (\ref{eq:intro-tor}), the nonvanishing of $\bkappa_{01}^{\rm Heeg}(\alpha)$ implies that the bottom class $\tilde{\kappa}_{01,\alpha}$ is non-zero and its index is equal to $\ind(\bkappa_{01}^{\rm Heeg}(\alpha))-t-2x$. Hence from Theorem~\ref{thm:bound} we get
\begin{equation}\label{eq:boundalpha-t}
\tfrac{1}{2}\,{\rm length}_{\Z_p}(\sha(W_{\alpha^{-1}}/K))\leq \ind(\bkappa_{01}^{\rm Heeg}(\alpha)) -t-2x +\CE.
\end{equation}

On the other hand, from Lemma~\ref{lemma:erlinv}, Theorem~\ref{controlthm} with $E=E_0$, and (\ref{eq:rational-div}), we obtain
\begin{equation}\label{eq:mainconjalpha}
\ind(\bkappa_{01}^{\rm Heeg}(\alpha))\leq\tfrac{1}{2}{\rm length}_{\Z_p}(\sha(W_{\alpha^{-1}}/K))+\tfrac{1}{2}\sum_{w\mid N}\ord_p(c_w^{(p)})+\tfrac{R}{2}+2x,
\end{equation}
using $m>\sum_{w\mid N}{\rm ord}_p(c_w^{(p)})$ to apply the observation in Remark~\ref{rmktam}. Combining \eqref{eq:boundalpha-t} and \eqref{eq:mainconjalpha}, we get
\[
t\leq \tfrac{1}{2}\sum_{w\mid N}{\rm ord}_p(c_w^{(p)})+ R+\CE,
\]
which contradicts \eqref{eq:t}, thereby proving the result in this case.

Now let $E/\Q$ be any elliptic curve in the isogeny class attached to $f$ and let $\pi: X_{0}(N) \rightarrow E$ be a modular parametrisation. 
The associated Kolyvagin system classes $\kappa_n^{\rm Heeg}$ are images of $\kappa_{0n}^{\rm Heeg}$ under the natural map $\rH^1(K,T_0^{M(n)})\rightarrow\rH^1(K,T^{(M(n))})$ induced by an isogeny $\phi:E_0\rightarrow E$, which depend on $E$ and $\phi$. Let $p^d$ denote the $p$-part of ${\rm deg}(\phi)$. If $\kappa_{n}^{\rm Heeg}=0$ for all $n$, then $p^{d}\cdot\kappa_{0n}^{\rm Heeg}=0$ for all $n$. Hence, repeating the same argument as above but with $m>\mathcal{M}+2t+2x+d$ in (b), we would arrive at a contradiction and so the proof concludes. 
\end{proof}

\begin{rem}\label{rem:upperbound}
In particular, if $E[p]$ satisfies (\ref{eq:irred}), the above argument shows that 
$\mathscr{M}_{\infty,\cN}\leq\tfrac{1}{2}\sum_{w\mid N}{\rm ord}_p(c_w^{(p)})$, 
where $\mathscr{M}_{\infty,\cN}=\mathscr{M}_\infty(\{\kappa_n^{\rm Heeg}\}_n)$ is the divisibility index of $\{\kappa_n^{\rm Heeg}\}_{n\in\mathcal{N}}$ 
for any $\cN \supset \cN(\cL_m)$.
\end{rem}

\begin{rem}
After the completion of this paper, we learnt from D.~Loeffler that an argument similar to the above first appeared in his work with S.~Zerbes \cite{LZ-quadratic} on the Iwasawa main conjecture for quadratic Hilbert modular forms.
%
We thank D.~Loeffler for pointing out the similarities between our independent methods.
\end{rem}

\subsection{The refined Kolyvagin's conjecture}
\subsubsection{Set-up} Let $R$ be the ring of integers of a finite extension of $\Q_p$ with maximal ideal $\fm$ and uniformiser $\varpi\in\fm$. Let $\alpha: \Gamma_{K}\to R^{\times}$ be a character such that $\alpha\equiv 1\;({\rm mod}\,\fm^m)$.
%
%
In this section, we do not assume $m\gg 0$, in particular we could also have $m=1$. Since we shall assume \eqref{eq:irred}, without loss of generality we may assume that $E=E_0$ is the optimal quotient of $X_0(N)$ associated to $f$. Put $T_{\alpha}=T_pE\otimes_{\Z_p} R(\alpha)$. 

For the purposes of the proof of Theorem~\ref{thmintroKoly-div} the case $R=\Z_p$ will suffice, but as they may be of independent interest, the preliminary results in $\S\ref{subsec:prelim-div}$ below will be proved in this more general setting. Note that the definitions of $\S\ref{subsec:KS-Heeg}$ readily extend to this more general $R$ (see \cite[\S\S{3.2-3}]{eisenstein}). Similarly as in the introduction, if $\boldsymbol{\kappa}=\{\kappa_n\}_{n\in\cN}$ is a Kolyvagin system for 
$(T_\alpha,\Fcal_{\rm ord},\mathcal{L})$,
for each $n\in\mathcal{N}=\mathcal{N}(\mathcal{L})$, define $\mathscr{M}(n)\in\Z_{\geq 0}\cup\{\infty\}$ by $\mathscr{M}(n)=\infty$ if $\kappa_n=0$, and by
\[
\mathscr{M}(n)=\max\{\mathscr{M}\colon\kappa_n\in \fm^{\mathscr{M}}\rH^1(K,T_\alpha/I_nT_\alpha)\}=\ind(\kappa_n)
\]
otherwise. For any $\cN'\subset \cN$, put $\mathscr{M}_{r,\cN'}=\min\{\mathscr{M}(n)\colon\nu(n)=r, n\in\cN'\}$, where $\nu(n)$ denotes the number of prime factors of $n$. 

One shows, similarly as in \cite{kolystructure}, that $\mathscr{M}_{r,\cN'}\geq\mathscr{M}_{r+1,\cN'}\geq 0$ for all $r\geq 0$. Put 
\[
\mathscr{M}_{\infty,\cN'}(\bkappa)=\lim_{r\to\infty}\mathscr{M}_{r,\cN'}.
\]
We let $\mathscr{M}_{\infty}(\boldsymbol{\kappa})=\mathscr{M}_{\infty,\cN}(\boldsymbol{\kappa})$. Note that $\mathscr{M}_{\infty,\cN'}(\bkappa)\geq  \mathscr{M}_{\infty}(\boldsymbol{\kappa})$ for every $\cN'\subset \cN$.

\subsubsection{Preliminaries}
\label{subsec:prelim-div}
In addition to the ingredients that went into the proof of Theorem~\ref{thmintroKoly}, our approach to the refinement in Theorem~\ref{thmintroKoly-div} relies on the following two results: the independence of the integer $\mathscr{M}_{\infty,\cN'}$ from $\cN'\subset \cN$ and an exact formula for the length of $\sha(W_{\alpha^{-1}}/K)$ in the spirit of Kolyvagin's structure theorem \cite[Thm.~1]{kolystructure}.

\begin{prop}\label{prop:M} Assume that $p\geq 3$ and that \eqref{eq:irred} holds.  Assume also that $\cL\subset\cL_0$ is such that $\cL_e \subset \cL$ for some $e>0$. Let $\boldsymbol{\kappa}=\{\kappa_n\}_{n\in\cN}$ be a Kolyvagin system for $(T,\Fcal,\mathcal{L})$. Then for any $e\gg 0$, we have 
\[
\mathscr{M}_{\infty,\cN^{(e)}}(\boldsymbol{\kappa}) = \mathscr{M}_{\infty,\cN}(\boldsymbol{\kappa}).
\]
\end{prop}


\begin{thm}\label{thm:kol-control}  
In the setting of Proposition~\ref{prop:M}, if $\kappa_1\neq 0$ then
\[
{\rm length}_{R}\bigl(\sha(W_{\alpha^{-1}}/K)\bigr)=2\bigl(\mathscr{M}_0(\boldsymbol{\kappa})-\mathscr{M}_{\infty}(\boldsymbol{\kappa})\bigr),
\]
where $\mathscr{M}_0(\boldsymbol{\kappa})=\ind(\kappa_1)$.
\end{thm}
\begin{rmk}
Note that this 
theorem applies to any anticylotomic character $\alpha$ of $\Gamma$, as will be clear from the proof below: 
the need to distinguish non-trivial characters sufficiently close\footnote{with respect to \cite[Thm. 3.3.8]{eisenstein}, where essentially the maximum $m$ such that $\alpha\equiv 1 \mod \fm^m$ was part of the error term} to the trivial character in \cite[6.1.1]{eisenstein_cyc}  disappears when \eqref{eq:irred} holds. In particular, Theorem \ref{thm:kol-control} applies for characters $\alpha:\Gamma \to \Lambda/(\mathfrak{P}+p^m)^{\times}$ for any height one prime ideal $\mathfrak{P}$ of $\Lambda$. 
\end{rmk}

The rest of the subsection establishes the above results. 

We begin with some preparation. Let $\Fcal=\Fcal_{\rm ord}$ and $T=T_{\alpha}$. For any $k\geq 1$, put $R^{(k)}= R/\fm^k$, $T^{(k)}= T/\fm^k$, $\mathcal{L}^{(k)}=\{\ell\in \cL: I_\ell\subset \fm^k\}$, and $\mathcal{N}^{(k)}=\mathcal{N}(\cL^{(k)})$, the set of squarefree products of primes in $\cL^{(k)}$.  
Recall that, by \cite[Prop.~3.3.2]{eisenstein}, there is an integer $\epsilon\in\{0,1\}$ such that for all $k$ and every
every $n\in\cN^{(k)}$ there is an $R^{(k)}$-module $M^{(k)}(n)$ such that
\begin{equation}\label{eq:str}
\rH^1_{\CF(n)}(K,T^{(k)}) \simeq (R/\fm^k)^\epsilon\oplus M^{(k)}(n)\oplus M^{(k)}(n).
\end{equation}
Moreover, by \cite[Thm.~3.3.8]{eisenstein}, if $\boldsymbol{\kappa}=\{\kappa_n\}_{n\in\cN}$ is a Kolyvagin system for $(T,\Fcal,\mathcal{L})$ with $\kappa_1\neq 0$, then $\epsilon=1$ and for $k\gg 0$, ${\rm exp}(M^{(k)}(1))<k$. We also denote by $\boldsymbol{\kappa}^{(k)}=\{\kappa_n^{(k)}\}_{n\in\cN^{(k)}}$ the Kolyvagin system for $(T^{(k)},\Fcal,\mathcal{L}^{(k)})$ obtained from $\boldsymbol{\kappa}$ by reduction modulo $\mathfrak{m}^k$.

We adopt the following notation: a finite torsion $R$-module $X$ is isomorphic to a sum of cyclic $R$-modules:  $X \simeq \oplus_{i=1}^{s(X)} R/\fm^{d_i(X)}$ for some uniquely-determined
integers $d_i(X)\geq 0$, with $d_1(X)\geq d_2(X)\geq \cdots \geq d_{s(X)}(X)$. For any $x\in X$, let the order and index of $x$ in $X$ be
$$
\ord(x):=\min\{n\geq 0: \varpi^n\cdot x =0\}, \ \ \ \ind(x):= \max\{t\geq 0: x\in\fm^t X\}.
$$
For an integer $t\geq 0$ we let $\rho_t(X)= \# \{i: d_i(X) > t\}$. Assume $\alpha\equiv 1 \mod \fm^m$ for some $m\geq 0$. For $n\in\cN^{(k)}$ with $k\geq m$, we let
\begin{equation}\label{eq:rho0}
\rho(n)^\pm:=\rho_0(\rH^1_{\F(n)}(K, T^{(m)})^\pm), \ \ \  \rho(n): = \rho(n)^+ +\rho(n)^-;
\end{equation}
i.e., $\rho(n)^\pm$ counts the number of summands of the $\pm$-components of $\rH^1_{\F(n)}(K, T^{(m)})$, on which, since $\alpha\equiv 1 \mod \fm^m$, there is a well-defined action of the complex conjugation. Note that, by \cite[Lem. 3.3.1]{eisenstein}, we are identifying $$\rH^1_{\F(n)}(K, T^{(m)})\simeq \rH^1_{\F(n)}(K, T^{(k)})[\fm^m]$$ and, in particular, $\rho(n)$ does not depend on $k$ or $\alpha$. 

Note also that, by \eqref{eq:str}, if $\rho(n)=1$, we must have $\rH^1_{\CF(n)}(K,T^{(k)}) \simeq R/\fm^k$, i.e. $M^{(k)}(n)=\{0\}$. If this is the case, we say that $n$ is a \emph{core vertex}, as in \cite[Def. 4.1.8]{mazrub} (and \cite{zanarella} in the Heegner point case).

\begin{proof}[Proof of Proposition \ref{prop:M}]
The essential  idea of this proof can be found in the work of Kolyvagin. Indeed, the case $\alpha =1$ can be easily extracted from the proof of \cite[Prop.~5.2]{McCallum}.

Let $r\geq 0$ be an integer such that $\mathscr{M}_{r,\cN}(\boldsymbol{\kappa})= \mathscr{M}_{\infty,\cN}(\boldsymbol{\kappa})$.
Let $\delta:= \mathscr{M}_{r,\cN}(\boldsymbol{\kappa})$ and $\cN_r :=\{n\in \cN: \nu(n)=r\}$. By definition, there exists some $n\in \cN_r$ such that $\mathscr{M}(n) =\delta$. Necessarily $n \in \cN^{(\delta+1)}$, 
so we may assume that 
$e > \delta+1$.  Suppose there exists a prime $\ell_0\mid n$ such that $\ell_0\not\in \mathcal{L}_e$. We will show that there exists a prime 
$\ell\in \mathcal{L}_e$ such that
$n' = \ell n/\ell_0 \in \cN_r$ is such that $\mathscr{M}(n') = \delta$. Replacing $n$ with $n'$ and repeating as necessary, we arrive at an $n\in \cN_r^{(e)}$ such that $\mathscr{M}(n) =\delta$, whence the proposition. 

We have $\kap^{(\delta+1)}_n\in \fm^{\delta}\rH^1_{\cF(n)}(K,T^{(\delta+1)})$ but $\kap^{(\delta+1)}_n \not\in \fm^{\delta+1}\rH^1_{\cF(n)}(K,T^{(\delta+1)})$.
From the identification $$\rH^1_{\cF(n)}(K,T^{(1)}) = \rH^1_{\cF(n)}(K,T^{(\delta+1)})[\fm],$$ we see that $\kap^{(\delta+1)}_n$ is identified with a non-zero 
class $0\neq \bar\kap \in \rH^1_{\cF(n)}(K,T^{(1)})$. Complex conjugation acts on $T^{(1)}$ and hence on $\rH^1_{\cF(n)}(K,T^{(1)})$ and we may fix
a sign $s = \pm$ such that $\bar\kap^s\neq 0$.

Let $0\neq c \in \rH^1_{\cF(n/\ell_0)^{\ell_0}}(K,T^{(1)})^{-s}$. Here the superscript $\ell_0$ on the $\cF(n/\ell_0)$ means that we impose no restrictions on the classes at the prime $\ell_0$. The existence of such a $c$ follows easily from global duality and the fact that $\ell_0 \in \cL$. 

Now let $\ell\in\cL_e$, $\ell\nmid n$, be a prime such that $\loc_\ell(\bar\kap^s)\neq 0$ and $\loc_\ell(c)\neq 0$. The existence of such a prime $\ell$ follows from \cite[Prop. 6.3.1]{eisenstein_cyc} applied with $m=1$ and $k=e$ in the notation of \emph{op. cit.} and noting that the error term is zero since we are assuming \eqref{eq:irred} holds. Then $n\ell \in \cN_{r+1}^{(\delta+1)}$ and so $\kap_{n\ell}^{(\delta+1)}$ is defined. 
By the Kolyvagin relations, $\loc_\ell(\kap_{n\ell}^{(\delta+1)})$ has the same order as $\loc_\ell(\kap_n^{(\delta+1)})$, and the latter equals $1$ as 
$\loc_\ell(\bar\kap)\neq 0$. In particular $\kap_{n\ell}^{(\delta+1)}\not\in\mathfrak{m}^{\delta+1}\rH^1_{\cF(n\ell)}(K,T^{(\delta+1)})$. On the other hand, $n\ell \in \cN_{r+1}$,
so $$\mathscr{M}(n\ell)\geq \mathscr{M}_{r+1,\cN} = \mathscr{M}_{r,\cN} = \delta$$ by the choice of $r$. It follows that $\kap_{n\ell}^{(\delta+1)} \in \fm^\delta\rH^1_{\cF(n\ell)}(K,T^{(\delta+1)})$
and so defines a non-zero $\bar\kap_\ell \in \rH^1_{\cF(n\ell)}(K,T^{(1)})$.
The Kolyvagin relations relate $\loc_\ell(\kappa_n)$ with $\loc_\ell(\kappa_{n\ell})$ via the usual finite-singular (iso)morphism. Since this morphism intertwines the eigenspaces for complex conjugation\footnote{see for example \cite[5.15.1]{nekovar}, recalling that $\ell\in\cL_{e}$  and $e>\delta+1$ imply that $\frob_\ell$ acts as complex conjugation on $T^{(\delta+1)}$}, this shows that $\loc_\ell(\bar\kap_\ell^\pm)\neq 0$ if and only if $\loc_\ell(\bar\kap^\mp)\neq 0$, so
$\loc_\ell(\bar\kap_\ell^{-s}) \neq 0$. 

By global duality we have
$$
0 = \sum_{v} \langle \loc_v(c), \loc_v(\bar\kap_\ell^{-s})\rangle_v  = \langle \loc_\ell(c), \loc_\ell(\bar\kap_\ell^{-s})\rangle_\ell  + \langle \loc_{\ell_0}(c), \loc_{\ell_0}(\bar\kap_\ell^{-s})\rangle_{\ell_0}.
$$
As $0\neq \loc_\ell(c)\in \rH^1_\mathrm{ur}(K,T^{(1)})^{-s}$ and $0\neq \loc_\ell(\bar\kap_\ell^{-s}) \in \rH^1_{\mathrm{tr}}(K,T^{(1)})^{-s}$, it follows
that $\langle \loc_\ell(c), \loc_\ell(\bar\kap_\ell^{-s})\rangle_\ell\neq 0$. This in turn implies that $\langle \loc_{\ell_0}(c), \loc_{\ell_0}(\bar\kap_\ell^{-s})\rangle_{\ell_0}\neq 0$, which means that $\loc_{\ell_0}(\bar\kap_\ell^{-s})\neq 0$. So the order of $\loc_{\ell_0}(\kap_{n\ell}^{(\delta+1)})$ is $1$. But then the Kolyvagin system relations imply that 
$\loc_\ell(\kap_{n\ell/\ell_0}^{(\delta+1)})$ has order $1$, and so $\mathscr{M}(n\ell/\ell_0) = \delta$ by the same reasoning that showed that $\mathscr{M}(n\ell) = \delta$. 
\end{proof}

For the proof of Theorem \ref{thm:kol-control}, we will also need the following result.
\begin{lemma}\label{lemma:j} Assume $p>3$.
There exists $d\in \Z_{\geq 0} \cup \{\infty\}$ independent on $k$ such that 
\begin{center}
$\langle\kappa_n^{(k)}\rangle = \fm^{d+ \length_R (M^{(k)}(n))}\rH^1_{\CF(n)}(K,T^{(k)})$ for every $n\in \cN^{(2k-1)}$.
\end{center}
If there exists $k$ and $n\in \cN^{(2k-1)}$ such that $\rho(n)=1$ and $\kappa_n^{(k)}\neq 0$, then $d \in \Z_{\geq 0}$.
\end{lemma}
\begin{proof}
The proof runs along the lines of 
 \cite[$\S$ 2.2]{zanarella}. 
 
 First, one shows that for every two core vertexes $n,m\in \cN^{(k)}$, the $R^{(k)}$-module spanned by $\kappa_n^{(k)},\kappa_m^{(k)}$ are isomorphic, i.e.,
\[
\rho(n)=\rho(m)=1 \ \ \Rightarrow \ \ R/\fm^{k-\ind(\kappa_n^{(k)})}\simeq \langle \kappa_n^{(k)} \rangle \simeq \langle \kappa_m^{(k)} \rangle\simeq R/\fm^{k-\ind(\kappa_m^{(k)})}.
\] 
This follows as in \cite[Cor. 2.2.13]{zanarella}. 
From this one deduces that for any $k$ and any $n\in \cN^{(k)}$ with $\rho(n)=1$,
\[
\langle \kappa_n^{(k)} \rangle \simeq \fm^{d_k} R/\fm^{k}
\]
for some $d_k\leq k$ independent of $n$. One sees that $d_{k+1}\geq d_k$ and that if $d_k<k$ (i.e. $\kappa_n^{(k)}\neq 0$), then $d_{k+1}=d_k$. Letting $d =\lim_{k\to\infty}d_k$, we find
\[
n\in \cN^{(k)}, \ \ \rho(n)=1 \Rightarrow \langle \kappa_n^{(k)} \rangle \simeq \fm^{d} R/\fm^{k}
\] 
for $d$ independent on $n$ and $k$. Note that if there exists $n$ such that $\kappa_n^{(k)}\neq 0$, $d$ is finite. One then shows that 
\begin{equation}\label{eq:span}
\langle\kappa_n^{(k)}\rangle = \fm^{d+ \length_R (M^{(k)}(n))}\rH^1_{\CF(n)}(K,T^{(k)}) \ \ \text{if $n\in \cN^{(2k-1)}$}.
\end{equation}
Following the proof of \cite[Lemma 2.3.1]{zanarella}, one proves this statement 
without the assumption that $\rho(n)=1$
using \cite[Lemma 1.6.4]{howard} and induction on $r(k,n)=k+\max\{\rho(n)^+,\rho(n)^-\}$. 
%
\end{proof}

\begin{proof}[Proof of Theorem \ref{thm:kol-control}]
As in the proof of \cite[Thm. 6.1.1]{eisenstein_cyc}, let $s(n)= \dim_{\mathbb{F}}\rH^1_{\F(n)}(K,T^{(1)})-1 = 2\cdot\dim_\mathbb{F}M^{(1)}(n)[\fm]$, where $\mathbb{F}=R/\fm$. Recall that, as in (\ref{eq:rho0}),
if $X$ is a finite torsion $R$-module, we denote by $\rho_0(X)$ the number of its (non-zero) summands and by $\rho_0(n)$ the sum of the number of summands of the $+$ and $-$ eigenspaces of $\rH^1_{\F(n)}(K, T^{(m)})$.
Let $\rho = \rho_0(1)$. Note that $\rho \geq \dim_{\mathbb{F}}\rH^1_{\F}(K,T^{(1)}) = s(1)+1$. 

By \cite[Thm 3.3.8]{eisenstein}, we can write $\rH^1_{\Fcal}(K,W_{\alpha^{-1}})\simeq {\rm Frac}(R)/R \oplus M(1)\oplus M(1)$, where $M(1)$ is a finite torsion $R$-module.
Let $k$ be a fixed integer such that 
\begin{equation}\label{eq:kbig}
k > \mathrm{length}_{R}(M(1)) + \ind(\kap_1)+m
\end{equation}
(recall that $\alpha\equiv 1 \pmod{\fm^m}$).
In particular, $\rH^1_{\Fcal}(K,T^{(k)})\simeq R^{(k)} \oplus  M(1)\oplus M(1)$ and  $\dim_{\mathbb{F}}\rH^1_{\F(n)}(K,T^{(k)})[\fm]-1= s(1)$.
In \emph{op. cit.} we found a sequence of integers $1=n_0, n_1,..., n_\rho \in \cN^{(k)}$ (which we can actually assume to be in $\cN^{(2k-1)}$ applying \cite[Prop. 6.3.1]{eisenstein_cyc} with a different choice of $k$) satisfying certain conditions (denoted by (a)-(g) in \emph{op. cit.}), in particular, if \eqref{eq:irred} holds, the error term $e$ is zero and we found:
\begin{itemize}
\item[(e)] $\ind(\kappa_{n_{i-1}}^{(k)})\geq \ind(\kappa_{n_i}^{(k)})$;
\item[(f)]  $\rho(n_i) \leq \rho(n_{i-1})$, and
$\rho(n_i) = \rho(n_{i-1})>1$ only if $\rho_0(\rH^1_{\F(n_{i})}(K,T^{(m)})^\pm)\geq 1$;
\item[(g)] if $\rho(n_{i-2})>1$ then 
$\rho(n_{i}) < \rho(n_{i-2})$.
\end{itemize}
From (g), we get that
$\rho(n_\rho) = 1$, implying in particular that 
\[
\rH^1_{\CF(n_\rho)}(K,T^{(k)}) \simeq R/\fm^k, \ \  \text{ i.e.} \ \  M^{(k)}(n_\rho)= 0.
\]
Note that we do not need to assume $m$ to be large in this case, as the proof of (B) in \cite{eisenstein_cyc} only requires $m>\mathcal{M}$ and $\mathcal{M}=0$ when \eqref{eq:irred} holds. In particular, we can take $m=1$ and work with the $\fm$-torsion in the case where the character $\alpha$ is attached to a height one prime ideal of $\Lambda$ different from $(p)$.

Moreover, since $\rho(n_\rho) = 1$, $n_\rho \in\mathcal{N}^{(2k-1)}$ and $\kappa^{(k)}_{n_{\rho}}<k$ (by (e) and the choice \eqref{eq:kbig}),   we can apply Lemma \ref{lemma:j} and find that $\langle \kappa_1^{(k)}\rangle \simeq \fm^{d+ \length_R (M^{(k)}(1))}\rH^1_{\CF}(K,T^{(k)})$ for $d=\ind(\kappa_{n_\rho}^{(k)})$. Therefore
\begin{equation}\label{eq:j+len}
j+\mathrm{length}_{R}(M(1))=\ind(\kappa_1) < k-\length_R(M(1)) < k, \ \ \ j = \ind(\kappa_{n_\rho}^{(k)}).
\end{equation}
In order to prove the theorem, we need to show that $\ind(\kappa_{n_\rho}^{(k)})=\mathscr{M}_{\infty}(\boldsymbol{\kappa})$.
Since $\ind(\kap_{n_\rho}^{(k)}) < k$, we have 
$$
\mathscr{M}(n_\rho) = \ind(\kap_{n_\rho}) = \ind(\kap_{n_\rho}^{(k)}) = j.
$$
Let $\cN' = \cN^{(2k-1)}$. Let $i\geq 0$ and let $n\in \cN'_i$ such that $\mathscr{M}(n) = \mathscr{M}_{\rho+i,\cN'}$.  
As $\mathscr{M}(n) \geq \ind(\kap^{(k)}_{n})$ by the definition of $\mathscr{M}(n)$ and since
$\ind(\kap^{(k)}_{n})  = \min\{k,j+\length_R(M(n))\}$ by Lemma \ref{lemma:j}, we have 
$$
\mathscr{M}_{\rho+i,\cN'} = \mathscr{M}(n) \geq \ind(\kap^{(k)}_{n}) =  \min\{k,j+\length_R(M(n))\} \geq j = \mathscr{M}(n_\rho)\geq \mathscr{M}_{\rho,\cN'}.
$$
Taking $i=0$, it follows that $\mathscr{M}_{\rho,\cN'} = \mathscr{M}(n_\rho)$.  Furthermore, since $\mathscr{M}_{\rho+i,\cN'} \leq \mathscr{M}_{\rho,\cN'}$, 
it follows that $\mathscr{M}_{\rho+i,\cN'} = \mathscr{M}_{\rho,\cN'}$ for all $i\geq 0$. Hence $\mathscr{M}_{\infty,\cN'} = \mathscr{M}_{\rho,\cN'} = \mathscr{M}(n_\rho)$. It then follows from Proposition \ref{prop:M} that
$\mathscr{M}_\infty(\boldsymbol{\kappa}) = \mathscr{M}_{\infty,\cN'} = \mathscr{M}(n_\rho)$.
\end{proof}

\subsubsection{The refined Kolyvagin conjecture}

\begin{proof}[Proof of Theorem~\ref{thmintroKoly-div}]
Assume \eqref{eq:irred}, $p>3$ and Conjecture~\ref{conj:anticyc}. By the nonvanishing of $\Lcal_p^{\rm BDP}(f/K)$, 
we can choose a character $\alpha=\alpha_m:\Gamma\rightarrow\Z_p^\times$ as in $\S\ref{subsec:tw}$ with $m>\sum_{w\mid N}{\rm ord}_p(c_w^{(p)})+\mathscr{M}_\infty(\kappa^{\rm Heeg})$, and such that $\Lcal_p^{\rm BDP}(f/K)(\alpha^{-1})\neq 0$. By Corollary~\ref{cor:index-sha} and Remark~\ref{rmktam} we then have 
\begin{equation}\label{eq:control-divalpha}
{\rm length}_{\Z_p}(\sha(W_{\alpha^{-1}}/K))=2\cdot\ind(\bkappa_1^{\rm Heeg}(\alpha))-4x-\sum_{w\mid N}{\rm ord}_p(c_w^{(p)}),
\end{equation}
where $x= v_p(\#\rH^0(\Q_p,E[p^\infty]))= v_p(1-\alpha_p)$.
Together with Theorem~\ref{thm:kol-control} for the Kolyvagin system $\bkappa(\alpha):=\{\bkappa^{\rm Heeg}_n(\alpha)\}_{n\in\mathcal{N}}$, where $\mathcal{N}=\mathcal{N}(\cL_{m})$,
\[
\mathscr{M}_{\infty,\cN}(\bkappa(\alpha))=\tfrac{1}{2}\sum_{w\mid N}{\rm ord}_p(c_w^{(p)})+2x.
\]
By Remark~\ref{rem:upperbound} we have $\mathscr{M}_{\infty,\cN}(\kappa^{\rm Heeg}) \leq \tfrac{1}{2}\sum_{w\mid N}{\rm ord}_p(c_w^{(p)})$. If the inequality was strict, there would exist $n\in\mathcal{N}$ such that $\ind(\kappa_n^{\rm Heeg})\lneq \tfrac{1}{2}\sum_{w\mid N}{\rm ord}_p(c_w^{(p)})$. By the congruence of Lemma \ref{lemmacongruence}, $\ind(\kappa_n^{\rm Heeg})+2x=\ind(\bkappa_n^{\rm Heeg}(\alpha))$, and hence by the choice of $m$ and the definition of $\mathscr{M}_{\infty,\cN}(\bkappa(\alpha))$, this would imply $\mathscr{M}_{\infty,\cN}(\bkappa(\alpha))\lneq \tfrac{1}{2}\sum_{w\mid N}{\rm ord}_p(c_w^{(p)})+2x$, giving a contradiction. Therefore we have $\mathscr{M}_{\infty,\cN}(\kappa^{\rm Heeg}) = \tfrac{1}{2}\sum_{w\mid N}{\rm ord}_p(c_w^{(p)})$ and, by Proposition \ref{prop:M}, we find
\[
\mathscr{M}_{\infty,\cN_{\rm Heeg}}(\kappa^{\rm Heeg})= \tfrac{1}{2}\sum_{w\mid N}{\rm ord}_p(c_w^{(p)}).
\]
\end{proof}

\begin{rmk}
Note that, conversely, the refined Kolyvagin conjecture together with Theorem \ref{thm:kol-control} implies the ``anticyclotomic Iwasawa main conjecture at $\alpha_m$'' for $m\gg 0$, namely 
\[
\Fcal_p(E/K)(\alpha_m)\sim_p\Lcal_p^{\rm BDP}(f/K)^2(\alpha_m),
\]
where $\Fcal_p(E/K)\in\Lambda$ is any generator of ${\rm char}_\Lambda(\mathfrak{X}_{\rm Gr}(E/K_\infty))$. This follows repeating the proof of
Theorem \ref{thmintroKoly-div} backwards and applying Lemma~\ref{lem:erl} and Theorem~\ref{controlthm}. 
\end{rmk}

\section{The cyclotomic analogue: the non-triviality of Kato's Kolyvagin system}
In this section we work with the cyclotomic Iwasawa algebra over $\Q$. More precisely, let $\Q_{\infty}$ be the unique $\Z_p$-extension of $\Q$, let $ \Gamma_\Q=\gal(\Q_{\infty}/\Q)\simeq \Z_p$, and
\[
\Lambda=\Lambda_\Q=\Z_p[[\Gamma_\Q]].
\]
\subsection{The Kolyvagin system of Kato classes}
We recall the construction of the Kolyvagin system of classes $\kappa_{n}=\kappa_{n}^{\rm Kato}\in \rH^1(\Q,T/I_nT)$ derived from Kato's Euler system and their Iwasawa-theoretic analogues. 
\subsubsection{Selmer structures}\label{selstr2}
Let $E/\Q$ be an elliptic curve of conductor $N$, without complex multiplication, 
and let $p\nmid 2N$ be a prime of good ordinary reduction for $E$. We assume 
\begin{equation}\label{eq:noQtors}
E[p]^{ss} = \mathbb{F}_p(\phi)\oplus\mathbb{F}_p(\phi^{-1}\omega) \implies \phi\neq 1 ,\omega,\tag{${\rm tor}_{\mathbb{Q}}$}
\end{equation}
where $\omega$ is the mod $p$ cyclotomic character. This is equivalent to the assertion that $E'(\Q)[p] =0$ for all curves $E'$ in the isogeny class of $E$.

By assumption the elliptic curve does not have CM, so there exists $\tau\in G_\Q$ such that $V_pE/(\tau-1)V_pE\simeq \Q_p$ since it follows by Serre's open image theorem \cite{serre} that $\left(\begin{smallmatrix}1&x\\0&1\end{smallmatrix}\right)$ is in the image of $\rho_E$ for some $0\neq x\in \Z_p$. We fix $\tau$ such that $\rho_E(\tau)=\left(\begin{smallmatrix}1&x\\0&1\end{smallmatrix}\right)$ with $$t':=v_p(x)=\min_m\{m:\left(\begin{smallmatrix}1&y\\0&1\end{smallmatrix}\right)\in {\rm Im}(\rho_E), v_p(y)=m\}.$$ In particular, we have
\begin{equation}\label{eq:tau}
\tau_{|\mu_{p^{\infty}}}=1 \ \ \ \text{and}  \ \ \ T_pE/(\tau-1)T_pE\simeq\Z_p \oplus \Z_p/p^{t'}.\tag{$\tau$}
\end{equation}

For a character $\alpha:\Gamma_\Q\to R^{\times}$ with values in the ring of integers $R$ of a finite extension $\Phi/\Q_p$, we consider the $G_\Q$-modules
\begin{equation}\label{eq:TVW}
T_\alpha:=T_pE\otimes_{\Z_p}R(\alpha),\quad
V_\alpha:=T_\alpha\otimes_{R}\Phi,\quad W_\alpha:=T_\alpha\otimes_{R}\Phi/R\simeq V_\alpha/T_\alpha,
\end{equation}
where $R(\alpha)$ is the free $R$-module of rank one on which $G_\Q$ acts via the projection $G_\Q\twoheadrightarrow\Gamma_\Q$ composed with $\alpha$, and the $G_\Q$-action on $T_\alpha$ is via $\rho_\alpha = \rho_E\otimes\alpha$.  Recall that as $E$ has ordinary reduction at $p$, there is a unique $G_{\Q_p}$-stable $\Zp$-summand  $\Fil_p^+(T_pE) \subset T_pE$
of rank one such that $T_pE/\Fil_p^+(T_pE)$ is unramified. We put $\Fil_p^+(T_\alpha) := \Fil_p^+(T_pE)\otimes_{\Z_p} R(\alpha)$, $\Fil_p^+(V_\alpha): = \Fil_p^+(T_\alpha)\otimes_R\Phi$,
and $\Fil_p^+(W_\alpha) := \Fil_p^+(T_\alpha)\otimes_R\Phi/R$.

Let $M$ denote any of the modules in (\ref{eq:TVW}) or a quotient $T_\alpha/IT_\alpha$ for a non-zero ideal $I\subset R$.
We define a \emph{Selmer structure} $\mathcal{F}$ on $M$ to be a finite set $\Sigma=\Sigma(\mathcal{F})$ of places of $\Q$ containing $\infty$, the prime $p$, and the primes where $M$ is ramified, together with a choice of a  submodule of \emph{local conditions} $\rH^1_\cF(\Q_w,M)\subset\rH^1(\Q_w,M)$ for every $w\in\Sigma$, similarly as in $\S$\ref{secsel}. The associated \emph{Selmer group} is then defined by
\[
\rH^1_{\Fcal}(\Q,M):={\rm ker}\biggl\{\rH^1(\Q^\Sigma/\Q,M)\rightarrow\prod_{w\in\Sigma}\frac{\rH^1(\Q_w,M)}{\rH^1_{\Fcal}(\Q_w,M)}\biggr\},
\]
where $\Q^\Sigma$ is the maximal extension of $\Q$ unramified outside $\Sigma$. 
The local conditions of interest in this setting are the following:
\begin{itemize}
\item For a finite prime $\ell\neq p$, the \emph{finite} (or \emph{unramified}) local condition:
\[
\rH^1_f(\Q_\ell,V_{\alpha}):={\rm ker}\bigl\{\rH^1(\Q_\ell,V_{\alpha})\rightarrow\rH^1(\Q_\ell^{\rm ur},V_{\alpha})\bigr\}.
\]
\item The \emph{ordinary} condition at $p$:
\[
\rH^1_{\rm ord}(\Q_p,V_\alpha):={\rm im}\bigl\{\rH^1(\Q_p,{\rm Fil}_p^+(V_\alpha))\rightarrow\rH^1(\Q_p,V_\alpha)\bigr\},
\]
which in our case is known to agree with the Bloch--Kato \emph{finite} local condition at $p$:
\[
\rH^1_f(\Q_p,V_{\alpha}):={\rm ker}\bigl\{\rH^1(\Q_p,V_{\alpha})\rightarrow\rH^1(\Q_p,V_{\alpha}\otimes_{\Q_p}\mathbb{B}_{\rm cris})\bigr\},
\]
where $\mathbb{B}_{\rm cris}$ is Fontaine's ring of crystalline periods.
\item The \emph{relaxed condition} at $\ell$: $\rH^1_{\rm rel}(\Q_\ell,V_{\alpha})=\rH^1(\Q_\ell,V_{\alpha})$.
\item The \emph{strict condition} at $\ell$: $\rH^1_{\rm str}(\Q_\ell,V_{\alpha})=\{0\}$.
\end{itemize}

The corresponding local conditions for $M\in\{T_\alpha,W_\alpha,T_\alpha/IT_\alpha\}$ are defined by propagation. (Thus, for example, $\rH^1_{\rm str}(\Q_\ell,T_\alpha)=\rH^1(\Q_\ell,T_\alpha)_{\rm tor}$ and $\rH^1_{\rm rel}(\Q_\ell,W_\alpha)=\rH^1(\Q_\ell,W_\alpha)_{\rm div}$.) 

For the module $T_\alpha/IT_\alpha$ with $I\subset R$ a non-zero ideal,  we also consider:
\begin{itemize}
\item The \emph{transverse} local condition at a prime $\ell\neq p$ is
\[
\rH^1_{\rm tr}(\Q_\ell,T_\alpha/IT_\alpha):={\rm ker}\bigl\{\rH^1(\Q_\ell,T_\alpha/IT_\alpha)\rightarrow\rH^1(\Q_\ell(\mu_{\ell}),T_\alpha/IT_\alpha)\bigr\}.
\]
\end{itemize}

Given a Selmer structure $\mathcal{F}$ on $V_\alpha$, for $\circ\in \{ {\rm rel},{\rm ord},{\rm str}\}$ we put
\[
\rH^1_{\Fcal_{\circ}}(\Q_{\ell},V_\alpha):=
\begin{cases}
\rH^1_{\circ}(\Q_p,V_\alpha) & \textrm{if $\ell =p$,}\\[0.2em]
\rH^1_\cF(\Q_{\ell},V_\alpha)& \textrm{else,}
\end{cases}
\]
and let $\Fcal_{\circ}$ also denote the Selmer structures similarly defined on $T_\alpha$, $W_\alpha$, and $T_\alpha/IT_\alpha$.

For $n\in \Z$ squarefree and coprime to $p$, and $I\subset R$ a nonzero ideal we also define the Selmer structure $\Fcal(n)$ as follows:
\[
\rH^1_{\Fcal(n)}(\Q_{\ell},T_\alpha/IT_\alpha)
=\begin{cases}
\rH^1_{\rm tr}(\Q_{\ell},T_\alpha/IT_\alpha) &\text{if }\ell\mid n, \\[0.2em]
\rH^1_{\Fcal}(\Q_{\ell},T_\alpha/IT_\alpha) & \text{if }\ell\nmid n.
\end{cases}
\]

Also, for $M^*=\Hom(M,\mu_{p^{\infty}})$ (so in particular $M^*\simeq W_{\alpha^{-1}}$ for $M=T_\alpha$),  
we let $\mathcal{F}^*$ be the Selmer structure on $M^*$ determined by the orthogonal complements of $\rH^1_{\mathcal{F}}(\Q_\ell,M)$ under local Tate duality 
\[
\rH^1(\Q_\ell,M)\times\rH^1(\Q_\ell,M^*)\rightarrow\rH^2(\Q_\ell,\mu_{p^\infty}) = \Q_p/\Z_p.
\]

Finally, we define the Bloch--Kato Selmer structure $\cF_{\rm{BK}}$ to be the Selmer structure given by
$\rH^1_{\cF_{\rm{BK}}}(\Q_\ell,M) = \rH^1_f(\Q_\ell,M)$ for all primes $\ell$ and $M\in\{V_\alpha,T_\alpha,W_\alpha,T_\alpha/IT_\alpha\}$.
It is known that $\cF_{\rm{BK}}^*$ agrees with the Bloch--Kato Selmer structure for $M^*$. 
From here on, we put  
\[
\Fcal=\Fcal_{\rm{BK},\operatorname{rel}},\quad\Fcal^*=\Fcal_{\rm{BK},\operatorname{str}}.
\] 
Later on, in the proof of Proposition~\ref{prop:control-alpha}, we shall also need to consider the Selmer structure $\Fcal_{\rm ord} := \Fcal_{\rm{BK},\rm{ord}}$.

\subsubsection{Kato's Kolyvagin system for $T_f$}

Let $f=\sum_{n=1}^\infty a_nq^n\in S_2(\Gamma_0(N))$ be the newform attached to $E$ with associated $p$-adic Galois representation $V_f$, constructed as the maximal quotient of $\rH^1_{\rm et}(\overline{Y_1(N)},\Q_p(1))$ on which the Hecke operators $T_n$ act as multiplication by $a_n$. After possibly replacing $E$ by an isogenous elliptic curve, we suppose that the $p$-adic Tate module $T_pE$ is isomorphic  to the $\Z_p$-lattice $T_f\subset V_f$ generated by the image of $\rH^1_{\rm et}(\overline{Y_1(N)},\Z_p(1))$ in $V_f$. {So $E$ may be assumed to be the elliptic curve $E_\bullet$ in the notation of \S\ref{ss, BDP-f}}.

Put { $T=T_pE = T_f$} , 
and let 
\[
\{z_n\in \rH^1(\Q(\mu_n),T)\}_n
\]
be the collection of classes obtained by specialising the classes $\mathbf{z}_n\in\rH^1(\Q(\mu_n),\mathbf{T})$ of Theorem~\ref{thm:kato} below at the trivial character.
%
Here $n$ runs over the positive integers coprime to $pN$. For varying $n$, it follows from \cite[Prop.~8.12]{kato-euler-systems} that the classes $z_n$ satisfy the following \emph{norm relations}:
\begin{equation}\label{eqNR}
\text{cores}^{\Q(\mu_{n\ell})}_{\Q(\mu_{n})}(z_{n\ell})=\begin{cases}
z_n &\ell\mid n,\\[0.2em]
P_{\ell}(\frob_{\ell}^{-1})z_n &\text{otherwise},
\end{cases}
\end{equation}
where $P_{\ell}(x)=\det(1-{\rm Frob}_{\ell}^{-1}x\,|\,T)$ is the characteristic polynomial of 
a geometric Frobenius at $\ell$.

We outline a construction of the Kolyvagin system derived from Kato's Euler system classes. 
{Denote by $\mathcal{L}_E$ the set of primes $\ell$ such that $(\ell,pN)=1$ and the image of $\frob_\ell$ in $\Gal(\Q(E[p^{t'+1}])/\Q)$ equals $\tau$.
}
Fix a subset $\mathcal{\cL}\subset\mathcal{L}_E$, and let $\mathcal{N}=\mathcal{N}(\cL)$ be the set of squarefree products of primes $\ell\in\mathcal{L}$. 
For $\ell\in\mathcal{L}$ define the ideal
\begin{equation}\label{eq:I-ell}
I_{\ell}=(\ell-1,a_\ell-\ell-1)\subset \Z_p,
\end{equation}
and for $n\in\mathcal{N}$ let $I_n=\sum_{\ell\mid n}I_\ell$. Let $G_n={\rm Gal}(\Q(\mu_n)/\Q)$. The Kolyvagin derivative operator $D_n\in\Z[G_n]$ at $n\in\mathcal{N}$ is defined by 
\[
D_{n}=\prod_{\ell \mid n}\sum_{i=1}^{\ell-2} i \sigma_{\ell}^i,
\] 
where $\sigma_{\ell}$ is a generator of $G_\ell$. Then one easily checks that the natural image of $D_nz_n$ in $\rH^1(\Q(\mu_n),T/I_nT)$ is fixed by $G_n$, and $\kappa_n^{\rm Kato}\in\rH^1(\Q,T/I_nT)$ is defined to be its preimage under the restriction map 
\begin{equation}\label{eq:res-map}
\rH^1\left(\Q, T / I_n T\right) \xrightarrow{\rm res}  \rH^1(\Q(\mu_n), T/I_nT)^{G_n}.
\end{equation}
By (\ref{eq:noQtors}) and our assumptions on $n$, the restriction map 
(\ref{eq:res-map}) is an isomorphism, and so $\kappa_n^{\rm Kato}$ is well-defined.
%
%
%
One can show that after a slight modification (see Theorem~3.24 and $\S{6.2}$  in \cite{mazrub} and \cite[Thm.~4.3.1]{rubin-ES-KS}) the resulting classes -- still denoted $\kappa_n^{\rm Kato}$ --
satisfy
\[
\kappa_n^{\rm Kato}\in \rH^1_{\Fcal(n)}(\Q, T / I_n T)
\]
and form a Kolyvagin system for $(T,\Fcal,\mathcal{L})$ in the sense of \cite{mazrub}, i.e., they satisfy the \emph{finite-singular relations}, which in particular imply that
\begin{equation}\label{eq:finsingkato}
\ord(\loc_\ell(\kappa_n^{\rm Kato}))=\ord(\loc_\ell(\kappa^{\rm Kato}_{n\ell})) \ \ \text{ for every $n\in \cN, \ell\in\cL$ such that $\ell\nmid n$}.
\end{equation} 
%
%
Moreover, similarly as in the setting of Heegner points, one can show that 
reducing such a Kolyvagin system modulo $p^k$ gives a Kolyvagin system for $E[p^k]$ (see \cite[Prop. 5.2.9 and Prop. 6.2.2]{mazrub}).


\begin{rmk} 
{
In the proof of \cite[Theorem~3.2.4]{mazrub}, building a  Kolyvagin system out of the Euler system classes, it is required that the Kolyvagin primes $\ell$ be such that $T/(\frob_\ell -1)T$ is a cyclic
$\Z_p$-module (i.e., $t'=0$ in the notations of $\S\ref{selstr2}$). This is used in [\emph{op.\,cit.}, Prop.~A.15] to deduce that the map ${\rm res}_{I_n}$ in the notations therein  is injective. In our setting (where $t'$ may be positive), the same argument shows that ${\rm res}_{I_n}$ is injective when  
restricted to $p^{t'}\rH^1_f(\Q_\ell,T/{I_n}T)$. This is because for primes in $\Lcal_E$ we have $T/(\frob_\ell -1)\simeq \Z_p \oplus \Z/p^{t'}$.
From this we easily conclude that if $\kappa'_n$ is the class defined as in \cite[(33)]{mazrub} from Kato's Euler system class $z_n$, then the collection of classes $p^{t'}\kappa_n'$ for varying $n$ gives the Kolyvagin system; this is the system we denoted $\{\kappa_n^{\rm Kato}\}_n$ above.
}
\end{rmk}

\begin{rmk}[Kato's Kolyvagin system for arbitrary $E$] {In this subsection we explained construction of a Kolyvagin system for an elliptic curve in the prime-to-$p$ isogeny class corresponding to the lattice $T_f$, i.e. the prime-to-$p$ isogeny class of the elliptic curve  $E_\bullet$ 
as in \S\ref{ss, BDP-f}.  
If $E/\Q$ is another curve in the isogeny class, then one obtains a Kolyvagin system by taking images of the classes $z_n\in\rH^1(\Q(\mu_n),T_pE_\bullet)$ under the natural map $T_pE_\bullet\rightarrow T=T_pE$ induced by an isogeny of minimal degree. (We take this as definition of the Kolyvagin system for $E$, which is well-defined up to multiplying the classes by a fixed $p$-adic unit.) For proof of Theorem \ref{thmintroKato}, it suffices\footnote{Indeed, one may proceed as in the last paragraph of the proof of Theorem~\ref{thmintroKoly} in \S\ref{sec:Kolyf}.} to consider the case 
 $E=E_\bullet$ (see the start of the proof of Theorem \ref{thmintroKato} in \S\ref{ss, koly-error}), which will be the focus of the rest of this section}. 
\end{rmk}

\subsubsection{The $\Lambda$-adic Kolyvagin system}
We equip the Iwasawa algebra $\Lambda$ with the $G_\Q$-action given by the character $\Psi:G_\Q\rightarrow\Lambda^\times$ arising from the projection $G_\Q\twoheadrightarrow\Gamma_\Q$, and put 
\[
\mathbf{T}=T\otimes_{\Z_p}\Lambda
\]
equipped with the diagonal $G_\Q$-action. 
%
%
%
Since $\Gamma_{\Q}$ is a pro-$p$ group,  assumption \eqref{eq:noQtors} implies that $E(\Q_\infty)[p]=0$. Therefore the same construction as above directly applied to the classes $\mathbf{z}_n\in\rH^1(\Q(\mu_n),\mathbf{T})$ of Theorem~\ref{thm:kato} below gives classes
\[
\bkappa_{n}^{\rm Kato}\in \rH^1(\Q,\mathbf{T}/I_n\mathbf{T})
\] 
forming a Kolyvagin system for $(\mathbf{T},\Fcal_\Lambda,\mathcal{L})$, where $\Fcal_\Lambda$ is the Selmer structure given by $\rH^1_{\Fcal_\Lambda}(\Q_\ell,\mathbf{T})=\rH^1(\Q_\ell,\mathbf{T})$ for all $\ell\in\Sigma$ (see \cite[$\S\S{5.3}, 6.2$]{mazrub}). In particular, note that by construction $\bkappa_1^{\rm Kato}=\mathbf{z}_1$.

\subsubsection{Kolyvagin system for cyclotomic twists}

We are interested in the image of this system in the cohomology of $T_{\alpha}$ for characters $\alpha:\Gamma_\Q\to \Z_p^{\times}$ with $\alpha\equiv 1 \pmod{p^m}$ and $m\gg 0$.

For any $\alpha:\Gamma_\Q\rightarrow\Z_p^\times$, the ideal 
$\p = (\gamma-\alpha(\gamma))$ is a height one prime ideal of $\Lambda$ with $\Lambda/\p\simeq\Z_p$. Composing $\Psi:G_\Q\rightarrow\Lambda^\times$ with the natural map $\Lambda\rightarrow \Lambda/\p$ recovers the character $\alpha$. 
By \cite[$\S$5.3]{mazrub} (see in particular Corollary 5.3.15 in \emph{op.cit.}), the image of the Kolyvagin system $\{\bkappa_{n}^{\rm Kato}\}_{n\in\mathcal{N}}$ via the natural map $\mathbf{T}\to T_\alpha$ obtained by tensoring $\Lambda\rightarrow\Lambda/\p$ with $T$, specialises to a Kolyagin system
\[
\{\bkappa_{n}^{\rm Kato}(\alpha)\}_{n\in\mathcal{N}}
\]
for $(T_\alpha,\F,\mathcal{L})$. The following is the analogue of Lemma \ref{lemmacongruence}, and is proved in the same manner.

\begin{lemma}\label{lemmacongruence2}
Suppose $\alpha\equiv 1\;({\rm mod}\,p^m)$. For all $n\in\mathcal{N}^{(m)}$ we have $\bkappa_{n}^{\rm Kato}(\alpha)\equiv \kappa_n^{\rm Kato}\;({\rm mod}\,p^m)$. 
\end{lemma}
 
Finally, to conclude the parallel with $\S\ref{seckolylambda}$, we recall the result on the nonvanishing of $\bkappa_{1}^{\rm Kato}$ and, therefore, of $\bkappa_{1}^{\rm Kato}(\alpha)$ for $m\gg 0$. 

\begin{thm}[Kato--Rohrlich]\label{thm:kato-rohrlich}
For $m\gg 0$ and any $\alpha\equiv 1\;({\rm mod}\,p^m)$ with $\alpha \neq 1$, the class $\bkappa_{1}^{\rm Kato}(\alpha)$ is non-zero.
\end{thm}

\begin{proof}
This follows from a combination of Rohrlich's results \cite{rohrlich-cyc}, implying the nonvanishing of the Mazur--Swinnerton-Dyer $p$-adic $L$-function $\mathcal{L}_p^{\rm MSD}(E/\Q)$, and Kato's explicit reciprocity law \cite[Thm.~16.6]{kato-euler-systems} (which we recall in the next section) relating $\bkappa_{1}^{\rm Kato}$ with $\mathcal{L}_p^{\rm MSD}(E/\Q)$ .
\end{proof}

\subsection{Cyclotomic Main Conjecture and results}  

\subsubsection{The $p$-adic $L$-function} 

We recall the existence of the cyclotomic $p$-adic $L$-function attached to  $E$. 

Given a modular parametrisation $\pi:X_0(N)\rightarrow E$, we denote by $c_E\in\Z$ the corresponding Manin constant. Let $\omega_E$ be a minimal differential on $E$. Pick generators $\delta^\pm$ of $H_1(E,\Z)^\pm$, and define the Néron periods $\Omega_E^\pm$ by 
\[
\Omega_E^\pm=\int_{\delta^\pm}\omega_E.
\]
We normalise the $\delta^\pm$ so that $\Omega_E^+\in\R_{>0}$ and $\Omega_E^-\in i\R_{\geq 0}$.
We also let $\alpha_p$ be the $p$-adic unit root of $x^2-a_p(f)x+p$, where $a_p(f)$ is the $p$-th Fourier coefficient of the modular form $f$ attached to $E$.

\begin{thm}[Mazur--Swinnerton-Dyer \cite{M-SwD}, W\"uthrich \cite{wuthrich-int}]\label{thm:MSDpadicL}
There exists an element $\Lcal^{\rm MSD}_p(E/\Q)\in\Lambda_\Q$ such that for any finite order character $\chi$ of $\Gamma_\Q$ of conductor $p^r$, we have
\[
\Lcal^{\rm MSD}_p(E/\Q)(\chi)=\begin{cases}
\frac{p^r}{\tau(\overline{\chi})\alpha_p^r}\cdot\frac{L(E,\overline{\chi},1)}{\Omega_E^+} &\text{ if }r>0,\\
(1-\alpha_p^{-1})^2\cdot\frac{ L(E,1)}{\Omega_E^+} &\text{ if }r=0,
\end{cases}
\]
where $\tau(\overline{\chi})=\sum_{a\;{\rm mod}\;p^r}\overline{\chi}(a)e^{2\pi ia/p^r}$ is the usual Gauss sum.
\end{thm}

\subsubsection{Kato's explicit reciprocity law}
We recall Kato's construction of an Euler system for $T$ and its relation to $\mathcal{L}_p^{\rm MSD}(E/\Q)$ via the Coleman map
\[
{\rm Col}_\infty:\rH^1_{s}(\Q_p,\mathbf{T})
\rightarrow\Lambda
\]
of e.g. \cite[Appendix]{rubindurham}, where $\rH^1_{s}(\Q_p,\mathbf{T}):=\rH^1(\Q_p,\mathbf{T})/\rH^1(\Q_p,{\rm Fil}_p^+(T)\otimes_{\Z_p}\Lambda)$.

\begin{thm}\label{thm:kato}
There exists a collection of cohomology classes 
\[
\mathbf{z}_n\in\rH^1(\Q(\mu_n),\mathbf{T}),
\]
with $n$ running over the positive integers prime to $pN$, satisfying the norm-relations
\[
{\rm cores}^{\Q(\mu_{n\ell})}_{\Q(\mu_{n})}(\mathbf{z}_{n\ell})=\begin{cases}
\mathbf{z}_n &\ell\mid n,\\[0.2em]
P_{\ell}(\frob_{\ell}^{-1})\mathbf{z}_n &\text{otherwise},
\end{cases}
\]
where $P_{\ell}(x)=\det(1-{\rm Frob}_{\ell}^{-1}x\,|\,T)$ is the characteristic polynomial of 
a geometric Frobenius at $\ell$. Moreover,
\[
{\rm Col}_\infty({\rm loc}_s(\mathbf{z}_1))=\Lcal_p^{\rm MSD}(E/\Q),
\]
where ${\rm loc}_s:\rH^1(\Q,\mathbf{T})\rightarrow\rH^1_s(\Q_p,\mathbf{T})$ is the restriction map at $p$ composed with the natural projection.
\end{thm}

\begin{proof}
Since we assume \eqref{eq:noQtors}, this folows from \cite[Thm.~6.1]{kataoka} (see also [\emph{op.\,cit}, Thm.~6.4, Thm.~7.14]). 
\end{proof}

In particular, specialising at cyclotomic characters, we can obtain the following result. 

\begin{lemma}\label{prop:erlKato}
Let $\alpha:\Gamma_\Q\to\Z_p^{\times}$ be a crystalline such that $\Lcal_p^{\rm MSD}(E/\Q)(\alpha^{-1})\neq 0$. 
Then
\[
\# \left(\Z_p/\Lcal^{\rm MSD}_p(E/\Q)(\alpha^{-1}) \right)\,\sim_p\,  \# \bigl(\rH^1_{s}(\Q_p,T_{\alpha})/\Z_p\cdot\loc_s(\bkappa^{\rm Kato}_{1}({\alpha}))\bigr),
\]
where $\rH^1_{s}(\Q_p,T_{\alpha}):=\rH^1(\Q_p,T_{\alpha})/\rH^1_{\rm ord}(\Q_p,T_{\alpha})$ and $\loc_s$ is the composite of the localisation map $\loc_p:\rH^1(\Q,T_{\alpha})\to \rH^1(\Q_p,T_{\alpha})$ with the projection to $\rH^1_{s}(\Q_p,T_{\alpha})$.
\end{lemma}

\begin{proof} 
This follows from a similar argument as in Lemma~\ref{lem:erl}. The Coleman map ${\rm Col}_\infty$ is an injection with finite cokernel in our setting (see \cite[Prop.~17.11]{kato-euler-systems}), and we claim that the specialisation map $\rH^1_{s}(\Q_p,\mathbf{T})\rightarrow\rH^1_{s}(\Q_p,T_\alpha)$ is surjective. Indeed, from the definitions we have a commutative diagram with exact rows:
\[
\xymatrix{
0\ar[r] &\rH^1_s(\Q_p,\mathbf{T})\ar[r]\ar[d]^a& \rH^1(\Q_p,\Fil_p^-(\mathbf{T}))\ar[d]^b\ar[r]^e & \rH^2(\Q_p,\Fil_p^+(\mathbf{T}))\ar[d]^c \\
0\ar[r] & \rH^1_s(\Q_p,T_\alpha)\ar[r]& \rH^1(\Q_p,\Fil_p^-(T_\alpha))\ar[r]^f & \rH^2(\Q_p,\Fil_p^+(T_\alpha)),
}
\]
where we have written $\Fil_p^\pm(\mathbf{T})$ for $\Fil_p^\pm(T)\otimes_{\Z_p}\Lambda$.  The vertical arrows are the natural specialisation maps. 
By Tate's local duality, $\rH^2(\Q_p,\Fil_p^+(\mathbf{T}))$ is dual to 
\[
\rH^0(\Q_p,\Fil_p^- (T)\otimes_{\Z_p}\Lambda^*)=\rH^0(\Q_{\infty,p},\tilde{E}[p^\infty])=\tilde{E}(\mathbb{F}_p)[p^\infty],
\] 
using that $\mathbb{Q}_{\infty,p}/\mathbb{Q}_p$ is totally ramified for last equality. The same argument shows that $\rH^2(\mathbb{Q}_p,\Fil_p^+(T_\alpha))$ is dual to $\tilde{E}(\mathbb{F}_p)[p^\infty]=0$. Under the above identifications, the map $c$ is dual to the inclusion. Thus replacing the right-most column with the images of the maps $e$ and $f$ yields 
a commutative diagram with short exact rows and columns and with the right-most arrow an injection. Applying the snake lemma to this diagram then yields
$\mathrm{coker}(a)\hookrightarrow \mathrm{coker}(b)$. The latter cokernel is given by $\rH^2(\mathbb{Q}_p,\Fil_p^-(\mathbf{T}))[\gamma-\alpha(\gamma)]$, which is dual to $\rH^0(\mathbb{Q}_{\infty,p},\Fil_p^+(E[p^\infty]))/(\gamma-\alpha(\gamma))$. Since the $\mathbb{Q}_{p,n}$-rational $p$-torsion in the formal group of $E$ is trivial for all $n$, this shows that $b$ is surjective and hence that $a$ is surjective, as claimed. The result now follows from Theorem~\ref{thm:kato}.
\end{proof}

\begin{rmk} 
Note that crystalline characters satisfying the assumptions of Lemma \ref{prop:erlKato} 
are just the powers $\epsilon^n$ for $n\geq 0$ an integer such that  $n\equiv 0 \mod (p-1)$
 and $\epsilon$ the $p$-adic cyclotomic character. 
\end{rmk}
\subsubsection{Rational cyclotomic Main Conjecture}

The Pontryagin dual $\Lambda^{\vee}$ is equipped with a $G_\Q$-action via the inverse of the character $\Psi:G_\Q\rightarrow\Lambda^\times$ arising from the projection $G_\Q\twoheadrightarrow\Gamma_\Q$. 
%
%

Let $M=T_pE\otimes \Lambda^{\vee}$, and put
\[
\rH^1_{\ord}(\Q_p,M)={\rm im}\bigl\{\rH^1(\Q_p,{\rm Fil}^+(T_pE)\otimes \Lambda^{\vee})\rightarrow\rH^1(\Q_p,M)\bigr\}.
\]
The \emph{ordinary Selmer group} $\rH^1_{\Fcal_{\rm ord}}(\Q,M)$ is
\begin{equation}\label{eq:defidiscretecyc}
\rH^1_{\Fcal_{\rm ord}}(\Q,M):=\ker \biggl\{\rH^1(\Q^\Sigma/\Q,M) \to \prod_{w\in\Sigma,w\nmid p}\rH^1(\Q_w,M)\times\frac{\rH^1(\Q_{p},M)}{\rH^1_{\rm ord}(\Q_{p},M)}\biggr\},\nonumber
\end{equation}
and we write $\X_{\rm ord}(E/\Q_{\infty})$ 
for its Pontryagin dual.
%
%

We recall the statement of the Iwasawa Main Conjecture formulated by Mazur in \cite{mazur-towers}.

\begin{conj}
\label{conj:cyc} The module $\X_{\rm ord}(E/\Q_{\infty})$ is $\Lambda$-torsion, with
\[
{\rm char}_\Lambda\left( \X_{\rm ord}(E/\Q_{\infty})\right)=(\Lcal^{\rm MSD}_p(E/\Q))
\]
as ideals in $\Lambda$. 
\end{conj}

In the following, we refer to the statement of Conjecture~\ref{conj:cyc} with $\Lambda$ replaced by $\Lambda\otimes\Q_p$ 
as the \emph{rational cyclotomic Main Conjecture}. It follows from the results of \cite{schneider-isogenies,PR-isogenies} that this is invariant under isogenies.

The first cases of Conjecture~\ref{conj:cyc} were proved by Rubin \cite{rubinmainconj} when $E/\Q$ has complex multiplication. In the non-CM case, the conjecture was  proved in \cite{kato-euler-systems,skinner-urban} (residually irreducible case) and \cite{kato-euler-systems,greenvats} (residually reducible case) under mild hypotheses. More recently, we have the following.

\begin{thm}
Let $E/\Q$ be an elliptic curve, and $p>2$ a prime of good ordinary reduction for $E$. 
\begin{itemize}
\item[(i)] If $E[p]^{ss}=\mathbb{F}_p(\phi)\oplus\mathbb{F}_p(\psi)$ as $G_\Q$-modules with $\phi\vert_{G_p}\neq\mathds{1},\omega$, then Conjecture~\ref{conj:cyc} holds.
\item[(ii)] If $E[p]$ satisfies {\rm (\ref{eq:irr})}, then the rational cyclotomic Main Conjecture holds.
\end{itemize}
\end{thm}

\begin{proof}
This is shown in \cite{eisenstein_cyc} in case (i), and in \cite{wan-hilbert} in case (ii).
\end{proof}

\subsection{Nonvanishing of Kato's Kolyvagin system}\label{sec:proofkato}

In this section we prove Theorem~\ref{thmintroKato} in the Introduction. The key ingredients are: 
\begin{enumerate}
\item The nonvanishing of $\bkappa^{\rm Kato}_{1}$ (Theorem~\ref{thm:kato-rohrlich})
\item For a character $\alpha:\Gamma_\Q\rightarrow\Z_p^\times$ sufficiently close to $1$, an estimate on the divisibility index of $\bkappa^{\rm Kato}_1(\alpha)$ in terms the length of the dual Selmer group and the Tamagawa factors of $E$ (Proposition~\ref{prop:control-alpha} below). 
\item A Kolyvagin system bound with controlled error terms  (Theorem~\ref{thm:mazurrubin} below).
\end{enumerate}
The cyclotomic Main Conjecture\footnote{In fact, as will be clear from the proof, the ``lower bound'' on the size the Selmer group predicted by Conjecture~\ref{conj:cyc} (or its rational version) suffices for the application to nonvanishing.} enters into the proof of (2). With these ingredients in hand, the proof of Theorem~\ref{thmintroKato} proceeds along  similar lines as in $\S\ref{sec:Kolyproof}$.
%
%

\subsubsection{Cyclotomic control theorem}

As above, 
we consider the $G_\Q$-modules 
\begin{equation}
T_\alpha:=T_pE\otimes_{\Z_p}\Zp(\alpha),\quad
V_\alpha:=T_\alpha\otimes_{\Z_p}\Q_p,\quad W_\alpha:=T_\alpha\otimes_{\Z_p}\Q_p/\Z_p\nonumber
\end{equation}
for a character $\alpha:\Gamma_\Q\rightarrow \Z_p^\times$. Recall that we set $\Fcal=\Fcal_{\rm{BK},\rm rel}$ and $\Fcal^*=\Fcal_{\rm{BK},\rm str}$, and for now $E=E_\bullet$.

\begin{prop}\label{prop:control-alpha} 
Suppose $\alpha:\Gamma_\Q\rightarrow\Z_p^\times$ be such that $\alpha\equiv 1\pmod{p^m}$. There exists a positive integer $\mathcal{M}$ independent of $\alpha$ such that if $m\geq\mathcal{M}$ and Conjecture~\ref{conj:cyc} holds,  then 
\[
{\rm length}_{\Z_p}(\rH^1_{\Fcal^*}(\Q, W_{\alpha^{-1}}))+\sum_{\ell\mid N} \ord_p(c_\ell(\alpha^{-1}))+2h= [\rH^1_{\Fcal}(\Q,T_{\alpha}):\Zp\cdot\bkappa^{\rm Kato}_1(\alpha)],
\]
where $c_\ell(\alpha^{-1})=\#\rH^1_{\rm ur}(\Q_\ell,W_{\alpha^{-1}})$ and $h={\rm ord}_p(\#\rH^0(\Q_p,E[p^\infty]))$. 
\end{prop}


\begin{proof}
By Rohrlich's results \cite{rohrlich-cyc}, after possibly excluding finitely many $\alpha$ we may assume  $\Lcal_p^{\rm MSD}(E/\Q)(\alpha)$ and $\Lcal_p^{\rm MSD}(E/\Q)(\alpha^{-1})$ are both nonzero. 
The Poitou--Tate duality gives rise to the exact sequence
\begin{equation}\label{eq:PT-5}
0\rightarrow\rH^1_{\Fcal^*}(\Q,W_{\alpha^{-1}})\rightarrow\rH^1_{\Fcal_{\rm ord}}(\Q,W_{\alpha^{-1}})\xrightarrow{{\rm loc}_p}\rH^1_{\rm ord}(\Q_p,W_{\alpha^{-1}})\xrightarrow{\beta}\rH^1_{\Fcal}(\Q,T_{\alpha})^\vee\rightarrow\rH^1_{\Fcal_{\rm ord}}(\Q,T_{\alpha})^\vee\rightarrow 0.
\end{equation}

By the cyclotomic Main Conjecture and Mazur's control theorem, the non-vanishing of $\Lcal_p^{\rm MSD}(E/\Q)(\alpha)$ and $\Lcal_p^{\rm MSD}(E/\Q)(\alpha^{-1})$ implies that $\rH^1_{\Fcal_{\rm ord}}(\Q,W_{\alpha^{-1}})$ and $\rH^1_{\Fcal_{\rm ord}}(\Q,T_{\alpha})$ are both finite. From (\ref{eq:PT-5}), it follows that $\rH^1_{\Fcal^*}(\Q,W_{\alpha^{-1}})$ is also finite, and that $\rH^1_{\Fcal}(\Q,T_{\alpha})$ has $\Z_p$-rank one; since the latter is torsion-free by \eqref{eq:noQtors}, we have in fact $\rH^1_{\Fcal}(\Q,T_{\alpha})\simeq \Zp$. Moreover, by local Tate duality the map $\beta$  is identified with the Pontryagin dual of the natural map
\[
{\rm loc}_{s}:\rH^1_{\Fcal}(\Q,T_{\alpha})\rightarrow\rH^1_{s}(\Q_p,T_{\alpha}):=\frac{\rH^1(\Q_p,T_{\alpha})}{\rH^1_{\rm ord}(\Q,T_{\alpha})},
\]
and from the above we see that ${\rm loc}_{s}$ has finite cokernel, with
\begin{equation}\label{eq:local-tate}
\#{\rm im}({\rm loc}_p)=\#{\rm coker}({\rm loc}_{s}).
\end{equation}
Since by Lemma \ref{prop:erlKato} the class $\bkappa^{\rm Kato}_{1}(\alpha)\in\rH^1_{\Fcal}(\Q,T_{\alpha})$ is nonzero and has non-torsion image in $\rH^1_{s}(\Q_p,T_{\alpha})$, we find 
\begin{equation}\label{eq:kato-erl}
\#{\rm coker}({\rm loc}_s)=\frac{[\rH^1_{s}(\Q_p,T_{\alpha}):\Z_p\cdot{\rm loc}_s(\bkappa_1^{\rm Kato}(\alpha))]}{[\rH^1_{\Fcal}(\Q,T_{\alpha}):\Z_p\cdot\bkappa_1^{\rm Kato}(\alpha)]}=
\frac{\#(\Z_p/\Lcal_p^{\rm MSD}(E/\Q)(\alpha^{-1}))}{[\rH^1_{\Fcal}(\Q,T_{\alpha}):\Z_p\cdot\bkappa^{\rm Kato}_1(\alpha)]}.
\end{equation} 

On the other hand, letting $\Fcal_E\in\Lambda$ be a generator of ${\rm char}_\Lambda(\X_{\rm ord}(E/\Q_\infty))$, by a variant of \cite[Thm.~4.1]{greenberg-cetraro} incorporating the twist by $\alpha$ with $m\gg 0$, we have
\begin{equation}\label{eq:control}
\begin{aligned}
\#\Zp/(\Fcal_E(\alpha^{-1}))&\sim_p\#\rH^1_{\Fcal_{\rm ord}}(\Q,W_{\alpha^{-1}})\cdot\prod_{\ell\vert N}c_\ell(\alpha^{-1})\cdot(\#\rH^0(\Q_p,E[p^\infty]))^2\\
&\sim_p\#\rH^1_{\Fcal^*}(\Q,W_{\alpha^{-1}})\cdot\#{\rm coker}({\rm loc}_s)\cdot\prod_{\ell\vert N}c_\ell(\alpha^{-1})\cdot(\#\rH^0(\Q_p,E[p^\infty]))^2,
\end{aligned}
\end{equation}
using (\ref{eq:PT-5}) and (\ref{eq:local-tate}) for the second equality. Since Conjecture~\ref{conj:cyc} implies 
\[
\#\Z_p/(\Fcal_E(\alpha^{-1}))=\#\Z_p/(\Lcal_p^{\rm MSD}(E/\Q)(\alpha^{-1}))
\] 
the result now follows from (\ref{eq:kato-erl}) and (\ref{eq:control}).
\end{proof}

We have the following analogue of Lemma~\ref{lemma:tam}.

\begin{lemma}\label{lem:tamQ}
Assume that $\alpha:\Gamma_\Q\rightarrow\Z_p^\times$ is such that $\alpha\equiv 1\;({\rm mod}\,p^m)$,  and $\ell\nmid p$ is a finite prime. Then $c_\ell(\alpha) \equiv c_\ell\;({\rm mod}\,p^m)$, where $c_\ell$ is the $p$-part of the Tamagawa factor of $E$ at $\ell$.
\end{lemma}

\begin{proof}
The proof is exactly the same as the one of Lemma \ref{lemma:tam}.
\end{proof}

Thus taking $\alpha:\Gamma_{\Q}\to\Z_p^{\times}$ as in Proposition~\ref{prop:control-alpha} sufficiently close to $1$, from Lemma~\ref{lem:tamQ} we arrive at
\begin{equation}\label{eq:cycMCalpha}
{\rm length}_{\Z_p}(\rH^1_{\Fcal^*}(\Q, W_{\alpha^{-1}}))+ \sum_{\ell\mid N} \ord_p(c_\ell)+2h= \ind(\bkappa^{\rm Kato}_1(\alpha),\rH^1_{\Fcal}(\Q,T_{\alpha})),
\end{equation}
where $h={\rm ord}_p(\#\rH^0(\Q_p,E[p^\infty]))$.


\subsubsection{Kolyvagin system bound with error term}\label{ss, koly-error}

The last ingredient we need is an extension of \cite[Thm.~5.2.2]{mazrub} with error terms. For a positive integer $e$ we put
\[
\mathcal{L}_{E,e}=\{\ell\in\mathcal{L}_E\;:\;I_\ell\subset p^e\Z_p\},
\]
where $I_\ell\subset\Z_p$ is as in (\ref{eq:I-ell}).

\begin{thm}
\label{thm:mazurrubin} 
Suppose $\cL\subset\cL_E$ satisfies $\cL_{E,e}\subset\cL$ for $e\gg 0$. Let $\alpha:\Gamma_\Q\rightarrow \Z_p^\times$ be a cyclotomic character such that $\alpha\equiv 1\;({\rm mod}\,p^m)$. Suppose that there is a collection of cohomology classes
\[
\{\tilde{\kappa}_n\in\rH^1(\Q,T_\alpha/I_nT_\alpha)\;:\;n\in\mathcal{N}\}
\]
with $\tilde{\kappa}_1\neq 0$ and that there is an integer $t\geq 0$, independent of $n$, such that $\{p^t\tilde{\kappa}_n\}_{n\in\mathcal{N}}\in{\mathbf{KS}}(T_\alpha,\Fcal,\mathcal{L})$. Then $\rH^1_{\Fcal}(\Q,T_\alpha)$ has $\Zp$-rank $1$, $\rH^1_{\Fcal^*}(\Q,W_{\alpha^{-1}})$ is finite,  and there is a non-negative integer $\mathcal{E}$ depending only on $T_pE$ 
such that 
\[
{\rm length}_{\Z_p}(\rH^1_{\Fcal^*}(\Q,W_{\alpha^{-1}}))\leq \ind(\tilde{\kappa}_1)+\mathcal{E},
\]
where $\ind(\tilde{\kappa}_1)={\rm length}_{\Zp}(\rH^1_{\Fcal}(\Q,T_{\alpha})/\Z_p\cdot\tilde{\kappa}_{1})$.
\end{thm}

\begin{proof}
By the same argument as in the proof of Theorem~\ref{thm:bound}, the result for $t=0$ easily implies the result for any $t\geq 0$, so it suffices to prove the former. Under hypothesis (\ref{eq:irred}), the result is shown  in \cite[Thm.~5.2.2]{mazrub}; 
the proof in the general case 
is given  in $\S\ref{subsec:MR-refined}$ (see Theorem~\ref{thm:boundcycleis}).
\end{proof}

Granted the results in $\S\ref{subsec:MR-refined}$, we are now ready to conclude the proof of Theorem~\ref{thmintroKato}.

\begin{proof}[Proof of Theorem \ref{thmintroKato}] 

Kato's Kolyvagin system for an arbitrary elliptic curve $E/\Q$ in the isogeny class attached to $f_E$ is obtained from the images of the classes $z_n\in\rH^1(\Q(\mu_n),T_pE_\bullet)$ under the natural map $T_pE_\bullet\rightarrow T=T_pE$ induced by an isogeny, so just as in the last paragraph of proof of Theorem~\ref{thmintroKoly} it suffices to consider the case $E=E_\bullet$. Thus arguing by contradiction, we assume that $\kappa_n=0$ for every $n\in\mathcal{N}$, and take $t$ such that
\[
t>\sum_{\ell\mid N}{\rm ord}_p(c_\ell)+\mathcal{E}+2h,
\]
where $\mathcal{E}$ and $h$ are as in Theorem~\ref{thm:mazurrubin} and \eqref{eq:cycMCalpha}, respectively. Take $\alpha:\Gamma_\Q\rightarrow\Z_p^\times$, $\alpha\equiv 1\;({\rm mod}\,p^m)$, as in Proposition~\ref{prop:control-alpha}, with $m\geq t$
and such that \eqref{eq:cycMCalpha} holds. Then, letting $\mathcal{N}=\mathcal{N}(\mathcal{L}_m)$, from Lemma~\ref{lemmacongruence2} we deduce the existence of a collection of cohomology classes $\{\tilde{\kappa}_{n,\alpha}\in\rH^1(\Q,T_\alpha/I_nT_\alpha)\}_{n\in\mathcal{N}}$ defined by the relation $p^t\cdot\tilde{\kappa}_{n,\alpha}=\bkappa_{n}^{\rm Kato}(\alpha)$. By Theorem~\ref{thm:kato-rohrlich} and Theorem~\ref{thm:mazurrubin} we obtain
\[
{\rm length}_{\Zp}(\rH^1_{\Fcal^*}(\Q,W_{\alpha^{-1}}))\leq{\rm ind}(\bkappa^{\rm Kato}_{1}(\alpha),\rH^1_{\Fcal}(\Q,T_\alpha))-t+\mathcal{E},
\] 
which by our choice of $t$ contradicts (\ref{eq:cycMCalpha}). This concludes the proof of Theorem~\ref{thmintroKato} assuming Conjecture~\ref{conj:cyc}. A straightforward modification of the above argument leads to the same conclusion just assuming the rational cyclotomic Main Conjecture.
\end{proof}

 
\begin{rem}\label{rem:refined-kato}
Denote by $\mathscr{M}_\infty^{\rm Kato}$ the divisibility index of $\{\kappa_n^{\rm Kato}\}$. 
Similarly as in the proof of Theorem~\ref{thmintroKoly-div}, the above argument combined with Mazur--Rubin's structure theorem for $\rH^1_{\Fcal^*}(\Q,W_{\alpha^{-1}})$ in terms of $\bkappa^{\rm Kato}(\alpha)$ (see \cite[Thm.~5.2.12]{mazrub}) shows that under the following hypotheses: 
\begin{itemize}
\item[(i)] (\ref{eq:irred}) holds (so that $\mathcal{E}=0$ in Theorem~\ref{thm:mazurrubin}), 
\item[(ii)] $a_p\not\equiv 1\pmod{p}$ (so that 
$h=0$ in Proposition \ref{prop:control-alpha}),
\end{itemize}
if Conjecture~\ref{conj:cyc} holds, then
\begin{equation}\label{eq:kato-div}
\mathscr{M}_\infty^{\rm Kato}=\sum_{\ell\mid N}{\rm ord}_p(c_\ell).
\end{equation}
In particular, by the cases of Conjecture~\ref{conj:cyc} established in \cite{bcs} (using \cite[IV-23, Lem.~3]{serre-ladic} to ensure the existence of the element $\sigma$ in \cite[Thm.~1.1.2]{bcs}, (\ref{eq:kato-div}) holds under the following hypothesis in addition to (i)-(ii): $p>3$.
\end{rem}

\begin{rem} In view of 
Kato's explicit reciprocity law as in \cite[Thm.~3.11]{kim}, the equality \eqref{eq:kato-div} translates\footnote{under assumption (ii), so that $t=0$ in the notations of \emph{loc.\,cit.}} into a proof of \cite[Conj.~1.9]{kim} concerning the maximal divisibility of the analytic invariants 
$\delta_n\in\Z_p/I_n$  
introduced by Kurihara \cite{kur-Iw2012} using modular symbols. See also \cite{kurihara-sakamoto} for independent results toward this conjecture. 

%
\end{rem}

\subsection{Extension of Mazur--Rubin's Selmer group bound}
\label{subsec:MR-refined}

In this section we prove Theorem~\ref{thm:boundcycleis} below. 
The result extends \cite[Thm.\,5.2.2]{mazrub} and might be of independent interest. 

Let $R$ be the ring of integers of a finite extension of $\Z_p$ with maximal ideal $\fm$ and uniformiser $\varpi\in\fm$. Let $\alpha: \Gamma_{\Q}\to R^{\times}$ be a cyclotomic character such that $\alpha\equiv 1\;({\rm mod}\,\fm)$, and put 
\[
T=T_pE\otimes_{\Z_p} R(\alpha).
\]
Recall that under these assumptions there exists $\tau\in G_\Q$ as in \eqref{eq:tau}.
Note that, since the first condition implies $\tau_{|\Q_\infty}=1$ and $\alpha$ is a character of $\Gamma_\Q$, we also have $T/(\tau -1)T\simeq R\oplus R/\fm^{t}$, where $t=t'\cdot\rank_{\Z_p}R$.

As above, we put $\Fcal=\Fcal_{\rm{BK},\rm rel}$, $\Fcal^*=\Fcal_{\rm{BK},\rm str}$, and let $\mathcal{L}\subset\mathcal{L}_E$ be a set of primes with $\mathcal{L}_{E,e}\subset\mathcal{L}$ for $e\gg 0$. We also put $\cN=\cN(\cL)$.

\begin{thm}\label{thm:boundcycleis}
Let $E/\Q$ be an elliptic curve without complex multiplication, and let $p$ be an odd prime of good reduction for $E$ such that 
$E(\Q)[p] = 0$.
Suppose that there is a Kolyvagin system $\kappa=\{\kappa_{n}\}_{n\in\cN}\in\mathbf{KS}(T,\Fcal,\cL)$ with $\kappa_{1}\neq 0$. Then $\rH^1_{\Fcal}(\Q,T)$ has $R$-rank one, ${\rm H}^1_{\Fcal^*}(\Q,T^*)$ is finite, and there exists a non-negative integer $\CE$ depending only on $T_pE$ and ${\rm rank}_{\Z_p}(R)$ such that 
\[
{\rm length}_{R}({\rm H}^1_{\Fcal^*}(\Q,T^*))\leq\ind(\kappa_1)+\CE,
\]
where $\ind(\kappa_1)={\rm length}_{R}\bigl({\rm H}^1_{\Fcal}(K,T)/R\cdot\kappa_{1}\bigr)$. Moreover, $\CE=0$ if \eqref{eq:irred} holds.
\end{thm}



As preparation for the proof of Theorem~\ref{thm:boundcycleis}, we collect some preliminary results from \cite[$\S{4.1}$]{mazrub}, whose proof applies verbatim under the assumption $E(\Q)[p]=0$. 

\begin{lemma}
\label{lemma:structure}
Let $k>0$ and $T^{(k)}=T/\fm^k$. For every $n\in\cN$ and $0< i\leq k$ there are natural isomorphisms
\[
\rH^1_{\Fcal(n)}(\Q,T^{(k)}/\fm^iT^{(k)})\xrightarrow{\sim}\rH^1_{\Fcal(n)}(\Q,T^{(k)}[\fm^i])\xrightarrow{\sim}{\rm H}_{\Fcal(n)}^1(\Q,T^{(k)})[\fm^i]\]
and
\[
\rH^1_{\Fcal(n)^*}(\Q, T^*[\fm^i])\xrightarrow{\sim}\rH^1_{\Fcal(n)^*}(\Q, T^*)[\fm^i]
\]
induced by the maps $T^{(k)}/\fm^iT^{(k)}\xrightarrow{\varpi^{k-i}}T^{(k)}[\fm^i]\hookrightarrow T^{(k)}$ and the inclusion $T^*[\fm^i]\hookrightarrow T^*$, respectively.
\end{lemma}

\begin{proof}
This is \cite[Lem.~4.1.1]{mazrub}.
\end{proof}
Let $\mathcal{L}^{(k)}$ be the set of Kolyvagin primes modulo $\fm^k$, that is,  the primes $\ell\nmid Np$ in $\cL$ such that 
\begin{itemize}
\item[(i)] $T/(\fm^kT+ (\frob_\ell -1)T)\simeq R/\fm^k\oplus R/\fm^{\min\{k,t\}}$,
\item[(ii)] $I_\ell =(\ell -1, \det (1-\frob_{\ell}|T)) \subset \fm^k$.
\end{itemize} 
Note that we can then see the reduction modulo $\fm^k$ of the Kolyvagin system for $T$ yields a Kolyvagin system for $(T/\fm^kT, \Fcal,\mathcal{L}^{(k)})$ given by classes 
$\kappa_n^{(k)}\in \rH^1_{\Fcal(n)}(\Q, T/\fm^kT)$ for $n\in \cN^{(k)} = \cN(\cL^{(k)})$.  

\begin{thm}\label{thm:strMR} 
For every $k>0$ and $n\in \mathcal{N}^{(k)}$, we have
\[
\fm^t\rH^1_{\Fcal(n)}(\Q, T/\fm^k)\simeq R/\fm^{k-t} \oplus \fm^t\rH^1_{\Fcal(n)^*}(\Q, T^*[\fm^k]).
\]
In particular, if we write $\rH^1_{\Fcal(n)}(\Q, T/\fm^k)\simeq \bigoplus_i  R/\fm^{d_i(n)}$ and $\rH^1_{\Fcal(n)^*}(\Q, T^*[\fm^k]) \simeq \bigoplus_j  R/\fm^{d^*_j(n)}$, then we have
\[
\bigoplus_{d_i(n)>t}  R/\fm^{d_i(n)} \simeq R/\fm^{k} \oplus \left( \bigoplus_{d^*_j(n)>t}  R/\fm^{d^*_j(n)}\right).
\]
\end{thm}

\begin{proof}
Since our Selmer structure $\Fcal$ has \emph{core rank} $\chi(T,\Fcal)=1$ in the sense of \cite{mazrub} (see [\emph{op.\,cit.}, Prop.~6.2.2]), the result follows from \cite[Thm.~4.1.13]{mazrub}. {
More precisely, in the proof of Lemma 4.1.6 of {\it op. cit.} the authors use cyclicity of the local cohomology modules  $\rH^1_f(\Q_{\ell},T/\fm^k),\rH^1_s(\Q_{\ell},T/\fm^k),\rH^1_f(\Q_{\ell},T^*[\fm^k])$ and $\rH^1_s(\Q_{\ell},T^*[\fm^k])$, which 
holds in our setting only if $t=0$. In general these cohomology modules are cyclic after multiplication by $\fm^t$; following the proof of \cite{mazrub}, we then obtain the stated  result. 
}
\end{proof}
%

\subsubsection{The \v{C}ebotarev argument}\label{seccheb}
We recall the definitions of the error terms $C_1,C_2$ of \cite[$\S$3.3.1]{eisenstein}. For $U = \Z_p^\times\cap \mathrm{im}(\rho_E)$ let
\[
C_1 := \min\{v_p(u-1)\colon u\in U\}.
\]
As $U$ is an open subgroup, $C_1<\infty$.
Recall also that $\End_{\Z_p}(T_pE)/\rho_E(\Z_p[G_{\Q}])$ is a torsion $\Z_p$-module and let
\[
C_2:=\min\bigl\{ n\geq 0 \colon p^n\End_{\Z_p}(T_pE)\subset\rho_E(\Z_p[G_{\Q}])\bigr\}.
\]
Let $d=\operatorname{rank}_{\Z_p}R$ and 
\[
e:=d(C_1+C_2+t')=d(C_1+C_2)+t,
\]
where $t'$ is determined by the choice of $\tau$ in \eqref{eq:tau}.

The following result is an analogue of \cite[Prop. 3.3.6]{eisenstein}.


\begin{prop}\label{prop:prime2} 
Let $k>e$ and consider two classes $c_0\in \rH^1(\Q,T^{(k)})$ and $c_1\in \rH^1(\Q,(T^{(k)})^*)$. 
Then there exist infinitely many primes $\ell\in \cL^{(k)}$ such that 
$$
\ord(\loc_\ell(c_i)) \geqslant \ord(c_i) - e, \ \ i=0,1.
$$
\end{prop}
\begin{proof}
Let $T_E^{(k)} = T_pE\otimes R/\fm^k \simeq (T_E^{(k)})^*$ and let $L$ be the fixed field of the action of $G_\Q$ on $T_E^{(k+2)}$. 
Since $\Q(\mu_{p^{k+2}})\subset L$, we have that $\Q_{p^{k+1}}\subset L$, where $\Q_{p^{k+1}}$ is the subfield of $\Q_\infty$ such that $\gal(\Q_\infty/\Q_{p^{k+1}})\simeq p^{k+1}\Z_p$. 

We claim that $\alpha|_{\gal(\Q_\infty/\Q_{p^{k+1}})}\equiv 1\pmod{\fm^k}$, which in turn implies  
\begin{equation}\label{eq:alphaLtrivial}
\alpha|_{G_L}=1 \mod \fm^k.
\end{equation}
The claim follows from the assumption $\alpha\equiv 1\pmod{\varpi}$. If $\gamma$ is a topological generator of $\Gamma_\Q$, we need to show that $\alpha(\gamma)^{p^{k+1}}\equiv 1 \pmod{\varpi^k}$.  
Since $\alpha(\gamma)=1+\varpi x$ for some $x\in R$, we have $$\alpha(\gamma)^{p^{k+1}}= 1+ \sum_{i=1}^{p^{k+1}} {{p^{k+1}}\choose{i}} (\varpi x)^{i}.$$ Noting that $\ord_p({{p^{k+1}}\choose{i}}) \geq k+1 -\ord_p(i)$,
we find $$\ord_{\varpi}({{p^{k+1}}\choose{i}}\varpi^i)\geq (k+1)\ord_\varpi(p) -(\ord_\varpi(p)-1)(\ord_pi)\geq (k+1)\ord_\varpi(p)-(\ord_\varpi(p)-1)(k+1).$$ The claim follows.

We therefore have the following identifications
\begin{align}
&\rH^1(L, T^{(k)})=\Hom(G_L, T_E^{(k)})^{(\alpha)} \subset \Hom(G_L, T_E^{(k)}),\label{eq:resL}\\
&\rH^1(L, (T^{(k)})^*) =\Hom(G_L, T_E^{(k)})^{(\alpha^{-1})} \subset \Hom(G_L, T_E^{(k)}),\label{eq:resL2}
\end{align}
where the superscript $(\alpha^{\pm 1})$ denotes the submodule on which $G_\Q$ acts via $\alpha^{\pm 1}$.
It then follows from \cite[(6.3)]{eisenstein_cyc} that $p^{C_1}$ annihilates $\rH^1(\gal(L/\Q),T^{(k)})$ and $\rH^1(\gal(L/\Q),(T^{(k)})^*)$. Hence, using the identifications \eqref{eq:resL} and \eqref{eq:resL2}, we have
\begin{align}
&p^{C_1}\cdot \ker \bigl( \rH^1(\Q, T^{(k)})\to  \Hom(G_L, T_E^{(k)}) \bigr) = 0,\label{eq:kerm}
 \\
&p^{C_1}\cdot \ker \bigl( \rH^1(\Q, (T^{(k)})^*)\to \Hom(G_L, T_E^{(k)}) \bigr) = 0. \label{eq:kerm2}
\end{align}

We now consider the images of the $c_i$ under the following natural maps
\[
\begin{tikzcd}[row sep=0.7mm]
\rH^1(\Q,M)\arrow{r} &\Hom(G_L,T_E^{(k)}) \arrow{r} &\Hom(G_L,T_E^{(k)}/(\tau-1)T_E^{(k)})\\
c_i \arrow[maps to]{r} &f_i \arrow[maps to]{r} &f_i,
\end{tikzcd}
\]
where $M=T^{(k)}$ or $(T^{(k)})^*$ for $i=0$ or $i=1$, respectively, 
and the second map is induced by the projection $T_E^{(k)}\to T_E^{(k)}/(\tau-1)T_E^{(k)}$. By the definition of $C_2$,
the image of $f_i$ contains $p^{C_2}\operatorname{End}(T_p(E))\cdot f_i(G_L)$. 
Since $\ord(f_i) \geq \ord(c_i) - \ord_\varpi(p)C_1$ by \eqref{eq:kerm}, it follows that, letting $d_i=\ord(c_i)-e$,
\begin{equation}\label{eq:imagefi}
R\cdot f_i(G_L) \supset \varpi^{k- d_i} T_E^{(k)}.
\end{equation}
Let $J_i=\{\gamma \in G_L : \ord(f_i(\gamma))\lneq d_i)\}$. This is a subgroup of $G_L$. 

We find
\[
R\cdot f_i(J_i)\subset \varpi^{k-(d_i-1)} T_E^{(k)} + (\tau -1)T_E^{(k)} \subsetneq \varpi^{k-d_i}T_E^{(k)},
\]
where the last inclusion follows from the fact that $T_E^{(k)}/(\tau -1)T_E^{(k)}\simeq R/\fm^k\oplus R/\fm^t$;  it must be strict, because if not $(\tau -1)T_E^{(k)}=  \varpi^{k-d_i}T_E^{(k)}$ and this implies $k=t$. Combined with \eqref{eq:imagefi}, this shows that $J_i \subsetneq G_L$ with index at least $p$. Now consider 
\[
B_i=\{g \in G_L : \ord(f_i(\gamma\tau))\lneq d_i)\}.
\]
Note that, since $g\in G_L$ acts trivially on $T_E^{(k)}$, $f_i(\gamma\tau)=f_i(\gamma)+f_i(\tau)$. Therefore for any $g,g'\in G_L$, we have $f_i(g^{-1}g')=-f_i(g\tau)+f_i(g'\tau)$. It follows that $B_i$ is a coset of $J_i$. Since both $J_1$ and $J_2$ have index at least $p$ in $G_L$ and $p>2$, we have shown
\[
\text{there exists } g \in G_L\setminus (B_1\cup B_2).
\]

Fix a choice of such a $g$. We now let $\ell\nmid Np$ be any prime such that both $c_i$s are unramified at $\ell$ and the conjugacy class of $\frob_\ell$ in
$\Gal(L'/\Q)$ is the same as the one of $g\tau$, where $L'$ is the compositum of the fixed fields of the kernels of $c_1$ and $c_2$ restricted to $G_L$.  The \v{C}ebotarev density theorem implies there are infinitely many such primes. By \eqref{eq:alphaLtrivial} and  the fact that $G_L$ acts trivially on $T_E^{(k)}$, $\frob_\ell$ acts as $\tau$ on $T^{(k)}$, therefore by \eqref{eq:tau}, we obtain
\begin{itemize}
\item[(i)] $T^{(k)}/(\frob_{\ell}-1)T^{(k)}\simeq R/\fm^k\oplus R/\fm^t$ and $\det(1-\frob_\ell|T)=\det(1-\tau|T)\equiv 0\mod\fm^k$;
\item[(ii)] $\ell=\chi_{cyc}(\frob_\ell)=\chi_{cyc}(g\tau)=\det(\rho_E)(g\tau) \equiv 1 \mod \fm^k$.
\end{itemize}
In particular we have shown $\ell$ is a Kolyvagin prime for $T^{(k)}$, that is, $\ell\in \cL_E^{(k)}$. 
Moreover, since $c_0,c_1$ are both unramified at $\ell$, we have $\loc_\ell(c_i)\in \rH^1_f(\Q,T_E^{(k)})\simeq T_E^{(k)}/(\frob_\ell -1)T_E^{(k)}\simeq R/\fm^k\oplus R/\fm^t$, where the first isomorphism is given by evaluation at $\frob_\ell$ and the second one follows from (i) above. In particular, $\ord(\loc_\ell(c_i))$ equals the order of $c_i(\frob_\ell)=f_i(g\tau)$, which, since $g \in G_L\setminus (B_1\cup B_2)$, is at least $d_i$, concluding the proof.
\end{proof}
\begin{rmk}
Recall that, as it can be seen from the proof of \cite[Thm.~4.1.13]{mazrub}, the isomorphism of Theorem \ref{thm:strMR} is not canonical. Therefore if we take $c_1$ to be a class generating $R/\fm^k$ in $\rH^1_{\Fcal(n)}(\Q, T^{(k)})$ and $c_2\in \rH^1_{\Fcal^*(n)}(\Q, T^{(k)})$, even though we have (non-canonical) isomorphisms $\rH^1_f(\Q_\ell, (T^{(k)})^*)\simeq R/\fm^k\oplus R/\fm^t\simeq \rH^1_f(\Q_\ell, T^{(k)})$, with $\ell\nmid n$, the Chebotarev result does not assert some ``linear independence'' of the localisations of the classes.
\end{rmk}
As a prelude to the proof of Theorem \ref{thm:boundcycleis}, we first prove the following weaker result. 
\begin{prop}\label{prop:expbound}
Assume $\rank_{\Z_p}(\rH^1_{\Fcal}(\Q,T))=1$ and let $s_1=\ind(\kappa_1, \rH^1_{\Fcal}(\Q,T))$. For $k\gg 0$ chosen so that $k > m+s_1+2e$, we have $$\fm^{s_1+2e}\rH^1_{\Fcal^*}(\Q,(T^{(k)})^*)=0.$$
\end{prop}
\begin{proof}
Choose $k\gg s_1+2e$ such that the image of $\kappa_1$ in $\rH^1_{\Fcal}(\Q,T^{(k)})$ is non-zero and has index $s_1$. Then there exists $c_0\in\rH^1_{\Fcal}(\Q,T^{(k)})$ of order exactly $k$ such that $\varpi^{s_1}c_0 = \kappa_1$. 
Let us write
\[
\rH^1_{\Fcal^*}(\Q,(T^{(k)})^*)=\bigoplus_{i=1}^{s} R/\varpi^{d_i}\cdot c_i, \ \ d_1\geq d_2 \geq \cdots \geq d_s.
\]
This is a cyclic-module decomposition with factors of the indicated lengths $d_i$. 

We apply Proposition \ref{prop:prime2} to the classes $c_0, c_1$ to find a prime $\ell \in \cL^{(k)}$ such that
\[
\ord(\loc_{\ell}(c_0))\geq k-e, \ \ \ \ord(\loc_{\ell}(c_1))\geq d_1-e.
\]
Recall that $\rH^1_f(\Q_{\ell},T^{(k)})$ and $\rH^1_f(\Q_{\ell}, (T^{(k)})^*)$ are isomorphic to $R/\fm^k\oplus R/\fm^t$. If $d_1<2e$, then the statement holds trivially. We assume $d_1>2e>t$. Then we have
\[
0\to \rH^1_{\Fcal^*_{\ell}}(\Q,(T^{(k)})^*) \to \rH^1_{\Fcal^*}(\Q,(T^{(k)})^*) \xrightarrow{\loc_\ell} R/\fm^{x'}\oplus R/\fm^a \to 0, \ \ \ x'\geq d_1-e, a\leq t,
\]
\[
0\to \rH^1_{\Fcal}(\Q,T^{(k)}) \to \rH^1_{\Fcal^{\ell}}(\Q,T^{(k)}) \xrightarrow{\loc_\ell} R/\fm^{k-x} \oplus R/\fm^{a'}\to 0, \ \ \ x\geq x', a'\leq t,
\]
where the second exact sequence is obtained from the first one by global duality. 

Recall that $\kappa_\ell \in \rH^1_{\Fcal^{\ell}}(\Q,T^{(k)})$ and that we cannot have $\ord(\loc_{\ell}(\kappa_\ell))\leq a'$, 
since this would contradict the assumption $k > m+s_1+2e$. Then by the Kolyvagin system relation \eqref{eq:finsingkato} and the assumptions above, we obtain
\[
k-(d_1 -e)\geq k-x\geq \ord(\loc_{\ell}(\kappa_\ell))=\ord(\loc_{\ell}(\kappa_1))=\ord(\loc_{\ell}(\varpi^{s_1}c_0))\geq k-s_1-e,
\]
which proves $\operatorname{exp}(\rH^1_{\Fcal^*}(\Q,(T^{(k)})^*))=d_1 \leq s_1+2e$. 
\end{proof}

\subsubsection{The proof of Theorem \ref{thm:boundcycleis}}
The proof follows the same lines of the proof of \cite[Thm. 3.2.1]{eisenstein} and \cite[Thm. 6.1.1]{eisenstein_cyc}. In particular, exactly as in \emph{op. cit.}, one reduces to proving that, for $k$ big enough, there exists $\mathcal{E}$  depending on $E$ and $\mathrm{rank}_{\Zp}R$, but not on $\alpha$ or $k$, such that
\begin{equation}\label{eq:bound}
\tag{B}
s_1+ \mathcal{E}\geq \length_R(\rH^1_{\Fcal^*}(\Q,(T^{(k)})^*)).
\end{equation}
To do so, we will inductively choose Kolyvagin primes in $\cL^{(k)}$ by repeatedly applying Proposition \ref{prop:prime2}. 

We will abbreviate $\rH^1_{\Fcal(n)}=\rH^1_{\Fcal(n)}(\Q,T^{(k)})$ and $\rH^1_{\Fcal^*(n)}=\rH^1_{\Fcal^*(n)}(\Q,(T^{(k)})^*)$ for any $n\in \cN^{(k)}$. 
We let 
\[
s(n)=\dim_{R/\fm}(\rH^1_{\Fcal^*(n)}[\fm])
\]
and write
\[
\rH^1_{\Fcal^*(n)} = \bigoplus_{i=1}^{s(n)}R/\fm^{d_i(\rH^1_{\Fcal^*(n)})}, \ \ \ \text{where } d_1(\rH^1_{\Fcal^*(n)})\geq d_2(\rH^1_{\Fcal^*(n)})\geq \dots\geq d_{s(n)}(\rH^1_{\Fcal^*(n)}).
\]
Let 
\[
s=s(1)
\]
 and note that this depends only on $E$ in view of Lemma \ref{lemma:structure} and the assumption $\alpha\equiv 1 \mod \fm$. For any $x\geq 0$, let 
\[
\rho_{x}(n)= \# \{ i: d_i(\rH^1_{\Fcal^*(n)})\gneq x\}, \ \ \ \rho:=\rho_{3se}(1)= \# \{ i: d_i(\rH^1_{\Fcal^*})\gneq 3s e\}.
\]
We will find sequences of integers $1=n_0, n_1,..., n_{\rho} \in \cN^{(k)}$ such that
\begin{itemize}
\item[(a)] $s(n_{i-1})-2 \leq s(n_{i})\leq s(n_{i-1})+2$;
\item[(b)] $t+ d_i(\rH^1_{\Fcal^*(n_{i})}) \geq d_{i+1}(\rH^1_{\Fcal^*(n_{i-1})})$ for $1\leq i \leq \rho_t(n_{i-1})-1$;
\item[(c)] $\length_{R}(\rH^1_{\Fcal^*(n_i)})\leq \length_{R}(\rH^1_{\Fcal^*(n_{i-1})})- d_{1}(\rH^1_{\Fcal^*(n_{i-1})}) +3e$;
\item[(d)] $\ord(\kappa_{n_i})\geq \ord(\kappa_{n_{i-1}})-e$;
\item[(e)] $\ind(\kappa_{n_{i-1}})\geq \ind(\kappa_{n_i})+d_{1}(\rH^1_{\Fcal^*(n_{i-1})}) -3e$;
\item[(f)] $\rho_{x}(n_i)\geq \rho_{x+t}(n_{i-1}) -1$, for any $x\geq 0$.
\end{itemize}
In particular, applying (e) repeatedly, we find
\begin{equation}\label{eq:repeatE}
s_1 \geq \ind(\kappa_{n_1})+d_{1}(\rH^1_{\Fcal^*})-3e\geq \dots \geq \ind(\kappa_{n_{\rho}})+\sum_{i=1}^{\rho}d_{1}(\rH^1_{\Fcal^*(n_{i-1})})-3\rho e.
\end{equation}

For any $1\leq j\leq \rho$, applying (f) repeatedly, starting with $i=j-1$ and $x=3se-(j-1)t=3sd(C_1+C_2+t')-(j-1)t$ (which is bigger than $t$ as $j\leq s$), we find, for $1\leq h\leq j-1$,
\[
\rho_x(n_{j-1})\geq \rho_{x+t}(n_{j-2})-1\geq \dots \geq \rho_{x+ht}(n_{j-h-1})-h\geq \dots \geq \rho_{x+(j-1)t}(1)-(j-1)=\rho-(j-1)\geq 1.
\]
In particular, $1\leq \rho_{x+t}(n_{j-2})-1\leq \rho_{t}(n_{j-2})-1$ and, more generally, for any $1\leq h\leq j-1$, $h\leq \rho_{t}(n_{j-h-1})-1$. We can therefore apply (b) to deduce
\begin{align*}
d_1(\rH^1_{\Fcal^*(n_{j-1})}) &\geq d_2(\rH^1_{\Fcal^*(n_{j-2})}) -t \geq \dots\geq d_{h}(\rH^1_{\Fcal^*(n_{j-h})}) -(h-1)t \geq d_{h+1}(\rH^1_{\Fcal^*(n_{j-h-1})}) -ht \geq \dots \\&\geq d_j(\rH^1_{\Fcal^*})-(j-1)t.
\end{align*}
Combining this with \eqref{eq:repeatE}, we find
\begin{align*}
s_1 + 3\rho e + (s-\rho)3se +t \tfrac{\rho(\rho-1)}{2} &\geq \ind(\kappa_{n_\rho})+\sum_{j=1}^{\rho}(d_{1}(\rH^1_{\Fcal^*(n_{j-1})})+(j-1)t) + \sum_{j=\rho+1}^{s} d_{j}(\rH^1_{\Fcal^*})
\\&\geq \ind(\kappa_{n_\rho})+\sum_{j=1}^{\rho}d_{j}(\rH^1_{\Fcal^*}) + \sum_{j=\rho+1}^{s} d_{j}(\rH^1_{\Fcal^*})  \geq \length_R(\rH^1_{\Fcal^*}).
\end{align*}
Since $3\rho e + (s-\rho)3se +t \tfrac{\rho(\rho-1)}{2}$ is bounded above by $3(s^2+s)e+ts(s-1)$, which depends only on $E$ and $\rank_{\Zp}R$, we have proved the desired inequality \eqref{eq:bound}.

We now prove the existence of such integers by induction. Assume we have found\footnote{Note that for $n_0=1$ these conditions are vacuously satisfied.} integers $n_i$ for $i\leq j$ satisfying (a)-(e).

By Theorem \ref{thm:strMR}, we can write
\[
\rH^1_{\Fcal(n_j)}= R/\fm^k\cdot c_0 \oplus \left(\bigoplus_{d_i(\rH^1_{\Fcal^*(n_j)})>t}R/\fm^{d_i(\rH^1_{\Fcal^*(n_j)})}\right) \oplus D,
\]
where $\exp(D)<t<e$.
Similarly as in the proof of Proposition \ref{prop:expbound}, we apply Proposition \ref{prop:prime2} to $c_0$ and the class $c_1$ generating $R/\fm^{d_1(\rH^1_{\Fcal^*(n_j)})}$, which we identify as a summand of $\rH^1_{\Fcal^*(n_j)}$, to obtain a prime $\ell\in \cL^{(k)}$. We obtain the following exact sequences
\begin{equation}\label{eq:esF}
0\to \rH^1_{\Fcal(n_j)_{\ell}}\to \rH^1_{\Fcal(n_j)} \xrightarrow{\loc_\ell} R/\fm^{k-a} \oplus R/\fm^f\to 0, \ \ \ a\leq e, f\leq t,
\end{equation}
\begin{equation}\label{eq:esF*}
0\to \rH^1_{\Fcal^*(n_j)_{\ell}}\to \rH^1_{\Fcal^*(n_j)} \xrightarrow{\loc_\ell} R/\fm^{x'}\oplus R/\fm^{g'} \to 0, \ \ \ x'\geq d_1(\rH^1_{\Fcal^*(n_j)})-e, g'\leq t.
\end{equation}
By global duality applied to the above exact sequences, we get respectively 
\begin{equation}\label{eq:esFdual}
0\to \rH^1_{\Fcal^*(n_j)_{\ell}} \to \rH^1_{\Fcal^*(n_j\ell)}\xrightarrow{\loc_\ell} R/\fm^{a'} \oplus R/\fm^{f'}\to 0, \ \ \ a'\leq a, f'\leq t.
\end{equation}
\begin{equation}\label{eq:esF*dual}
0\to \rH^1_{\Fcal(n_j)_\ell} \to \rH^1_{\Fcal(n_j\ell)} \xrightarrow{\loc_\ell} R/\fm^{k-x} \oplus R/\fm^{g}\to 0, \ \ \ x\geq x', g\leq t,
\end{equation}
We claim (a)-(f) are satisfied for $n_{j+1}=n_j\ell$. The proof is similar to the one in \cite[$\S$3.3.3]{eisenstein}, but it is simplified by the fact that here the image of the localisation is almost cyclic (i.e., it is cyclic after multiplication by $\varpi^t$, rather than being of rank two over $R/\fm^k$) and, moreover, the ``torsion part'' of $\rH^1_{\Fcal(n_j)}$ (up to summands of small exponents) is itself a Selmer group (namely $\rH^1_{\Fcal^*(n_j)}$) thanks to Theorem \ref{thm:strMR}, and hence we can control better the behavior of the localisation\footnote{Compare the exact sequences \eqref{eq:esF},\eqref{eq:esF*},\eqref{eq:esFdual},\eqref{eq:esF*dual} to \cite[(3.16),(3.17)]{eisenstein}.}.

Using \cite[Lem. 3.3.9]{eisenstein} applied to \eqref{eq:esF*} and \eqref{eq:esFdual} and the fact that $\rH^1_{\mathcal{G}}[\fm]\simeq \rH^1_{\mathcal{G}}(\Q,\bar{T})$ for all the Selmer groups involved (thanks to Lemma \ref{lemma:structure}), we obtain (a). One easily obtains (c) combining again \eqref{eq:esF*} and \eqref{eq:esFdual}. 

In order to prove (b), we note that applying \cite[Lem. 3.3.9]{eisenstein} to the first inclusion in \eqref{eq:esFdual}, we find $d_i(\rH^1_{\Fcal^*(n_j)_{\ell}})\leq d_i(\rH^1_{\Fcal^*(n_j\ell)})$ for $1\leq i\leq s_i(n_j)$, where we let $d_{s_i(n_j)}(\rH^1_{\Fcal^*(n_j)_{\ell}})=0$ if the number of summands of $\rH^1_{\Fcal^*(n_j)_{\ell}}$ is smaller than $s(n_j)$, which happens if and only if $\loc_\ell(\rH^1_{\Fcal^*(n_j)}[\fm])\neq 0$.
From \eqref{eq:esF*} we obtain the exact sequence 
\[
0\to C:=(\rH^1_{\Fcal^*(n_j)_{\ell}}\cap p^t\rH^1_{\Fcal^*(n_j)}) \to p^t\rH^1_{\Fcal^*(n_j)} \xrightarrow{\loc_\ell} R/\fm^{x'-t} \to 0.
\]
We then let $M_0\simeq p^t \cdot R/\fm^{d_{i_0}(\rH^1_{\Fcal^*(n_j)})}$ be a summand of $p^t\rH^1_{\Fcal^*(n_j)}$ surjecting onto $R/\fm^{x'-t}$. By \eqref{eq:esF*}
we then have a surjection $C\twoheadrightarrow  p^t\rH^1_{\Fcal^*(n_j)}/M_0\simeq \oplus^{\rho_t(n_j)}_{i\neq i_0}R/\fm^{d_{i}(\rH^1_{\Fcal^*(n_j)})-t}$. 
Dually, we obtain an injection $ \oplus^{\rho_t(n_j)}_{i\neq i_0}R/\fm^{d_{i}(\rH^1_{\Fcal^*(n_j)})-t} \hookrightarrow C\hookrightarrow \rH^1_{\Fcal^*(n_j)_{\ell}}$, to which we can apply \cite[Lem. 3.3.9]{eisenstein} to obtain $d_i(\rH^1_{\Fcal^*(n_j)_{\ell}})\geq d_i(\rH^1_{\Fcal^*(n_j)})$ if $i\lneq i_0$ and $t+d_i(\rH^1_{\Fcal^*(n_j)_{\ell}})\geq d_{i+1}(\rH^1_{\Fcal^*(n_j)})$ if $i_0\leq i\leq \rho_t(n_j)-1$. In particular, we obtain
\[
 d_i(\rH^1_{\Fcal^*(n_j\ell)}) +t \geq d_i(\rH^1_{\Fcal^*(n_j)_{\ell}}) \geq d_{i+1}(\rH^1_{\Fcal^*(n_j)}) \ \ \ \text{for every } 1\leq i\leq \rho_t(n_j)-1,
\]
which yields (b) for $i=j+1$. This inequality for $i=\rho_{x+t}(n_{j})-1$ also yields (f).

It remains to show (d) and (e). The proof is very similar to the proof of (d) and (e) in \cite[$\S$3.3.3]{eisenstein}, and we leave the details to the interested reader.

\bibliography{references}
\bibliographystyle{alpha}
\end{document}